\documentclass[a4paper,12pt]{article}

\usepackage{amscd}
\usepackage{amssymb,amsfonts,amsmath,amsthm}
\usepackage{verbatim}

\usepackage{amsmath}
\usepackage{amsthm}
\usepackage{amssymb}

\usepackage{latexsym}
\usepackage{array}
\usepackage{enumerate}

\numberwithin{equation}{section}

\usepackage{amsmath,amssymb,mathrsfs,amsthm}
\usepackage[all]{xy}
\usepackage{graphicx}
\usepackage{ulem}
\usepackage{ascmac}
\usepackage{mathtools}

\newcommand{\rHom}{\mathrm{RHom}}
\newcommand{\BDC}{{\mathbf{D}}^{\mathrm{b}}}

\newcommand{\Mod}{\mathrm{Mod}}
\newcommand{\Hom}{\mathrm{Hom}}

\newcommand{\CC}{\mathbb{C}}

\newcommand{\RR}{\mathbb{R}}

\newcommand{\ZZ}{\mathbb{Z}}
\newcommand{\D}{\mathcal{D}}

\newcommand{\F}{\mathcal{F}}

\newcommand{\M}{\mathcal{M}}
\newcommand{\N}{\mathcal{N}}

\newcommand{\Ker}{\operatorname{Ker}}
\newcommand{\Coker}{\operatorname{Coker}}
\newcommand{\Image}{\operatorname{Im}}

\newcommand{\an}{{\rm an}}

\newcommand{\id}{{\rm id}}
\newcommand{\supp}{{\rm supp}}

\newcommand{\Perv}{{\rm Perv}}
\newcommand{\Sol}{{\rm Sol}}
\newcommand{\DR}{{\rm DR}}

\newcommand{\tl}[1]{\widetilde{#1}}

\newcommand{\simto}{\overset{\sim}{\longrightarrow}}

\newcommand{\op}{{\mbox{\scriptsize op}}}
\newcommand{\SD}{\mathcal{D}}

\newcommand{\SO}{\mathcal{O}}

\newcommand{\SM}{\mathcal{M}}
\newcommand{\SN}{\mathcal{N}}
\newcommand{\SL}{\mathcal{L}}

\newcommand{\SF}{\mathcal{F}}

\newcommand{\SH}{\mathcal{H}}
\newcommand{\calI}{\mathcal{I}}

\newcommand{\Conn}{\mathrm{Conn}}
\newcommand{\Modcoh}{\mathrm{Mod}_{\mbox{\rm \scriptsize coh}}}
\newcommand{\Modhol}{\mathrm{Mod}_{\mbox{\rm \scriptsize hol}}}
\newcommand{\Modrh}{\mathrm{Mod}_{\mbox{\rm \scriptsize rh}}}

\newcommand{\BDCcoh}{{\mathbf{D}}^{\mathrm{b}}_{\mbox{\rm \scriptsize coh}}}

\newcommand{\BDChol}{{\mathbf{D}}^{\mathrm{b}}_{\mbox{\rm \scriptsize hol}}}
\newcommand{\BDCrh}{{\mathbf{D}}^{\mathrm{b}}_{\mbox{\rm \scriptsize rh}}}
\newcommand{\BDCmero}{{\mathbf{D}}^{\mathrm{b}}_{\mbox{\rm \scriptsize mero}}}
\newcommand{\DChol}{\mathbf{D}_{\mbox{\rm \scriptsize hol}}}

\newcommand{\DD}{\mathbb{D}}

\newcommand{\Dotimes}{\overset{D}{\otimes}}
\newcommand{\Dboxtimes}{\overset{D}{\boxtimes}}
\newcommand{\Potimes}{\overset{+}{\otimes}}
\newcommand{\Pboxtimes}{\overset{+}{\boxtimes}}
\newcommand{\rhom}{{\bfR}{\mathcal{H}}om}
\newcommand{\rihom}{{\bfR}{\mathcal{I}}hom}
\newcommand{\Prihom}{{\bfR}{\mathcal{I}}hom^+}

\newcommand{\I}{{\rm I}}
\newcommand{\che}[1]{\check{#1}}
\newcommand{\var}[1]{\overline{#1}}
\newcommand{\BEC}{{\mathbf{E}}^{\mathrm{b}}}

\newcommand{\ZEC}{{\mathbf{E}}^{\mathrm{0}}}
\newcommand{\ZECmero}{{\mathbf{E}}^{\mathrm{0}}_{\mbox{\rm \scriptsize mero}}}
\newcommand{\BECmero}{{\mathbf{E}}^{\mathrm{b}}_{\mbox{\rm \scriptsize mero}}}
\newcommand{\Q}{\mathbf{Q}}

\newcommand{\EE}{\mathbb{E}}
 
\newcommand{\bs}{\backslash}

\newcommand{\bfR}{\mathbf{R}}
\newcommand{\bfL}{\mathbf{L}}
\newcommand{\bfD}{\mathbf{D}}
\newcommand{\bfX}{\mathbf{X}}
\newcommand{\bfY}{\mathbf{Y}}
\newcommand{\rmR}{{\rm R}}
\newcommand{\rmE}{{\rm E}}
\newcommand{\rmD}{{\rm D}}

\newcommand{\bfE}{\mathbf{E}}

\renewcommand{\Re}{\operatorname{Re}}

\newcommand{\reg}{{\rm reg}}
\newcommand{\sh}{{\rm sh}}
\newcommand{\RH}{{\rm RH}}
\newcommand{\ord}{\operatorname{ord}}

\newcommand{\AC}{(\mathbf{AC})}
\newcommand{\rami}{{\rm rm}}
\newcommand{\modi}{{\rm md}}
\newcommand{\blow}{{\rm bl}}
\newcommand{\pt}{{\rm pt}}

\newtheorem{theorem}{Theorem}[section]

\newtheorem{corollary}[theorem]{Corollary}
\newtheorem{lemma}[theorem]{Lemma}
\newtheorem{sublemma}[theorem]{Sublemma}
\newtheorem{proposition}[theorem]{Proposition}

\newtheorem{notation}[theorem]{Notatiton}
\theoremstyle{definition}
\newtheorem{definition}[theorem]{Definition}
\theoremstyle{remark}
\newtheorem{remark}[theorem]{\sc Remark}

\title{Note On The Algebraic Irregular Riemann--Hilbert Correspondence
\footnote{{\bf 2010 Mathematics 
Subject Classification: }32C38, 32S60, 35A27}}

\author{Yohei ITO
\footnote{Graduate School of Mathematical Science, The University of 
Tokyo, 3-8-1, Komaba, 
Meguro, Tokyo, 153-8914, Japan. 
E-mail: yitoh@ms.u-tokyo.ac.jp }}

\date{}

\begin{document}
\maketitle

\begin{abstract}
The subject of this paper is an algebraic version of
the irregular Riemann--Hilbert correspondence which was mentioned in \cite{Ito19}.
In particular, we prove an equivalence of categories between
the triangulated category $\BDChol(\D_X)$ of holonomic $\D$-modules
on a smooth algebraic variety $X$ over $\CC$
and the one $\BEC_{\CC-c}(\I\CC_{X_\infty})$ of
algebraic $\CC$-constructible enhanced ind-sheaves
on a bordered space $X^\an_\infty$.
Moreover we show that there exists a t-structure on the triangulated category
$\BEC_{\CC-c}(\I\CC_{X_\infty})$
whose heart is equivalent to the abelian category of holonomic $\D$-modules on $X$.
Furthermore we shall consider simple objects of its heart and minimal extensions of objects of its heart.
\end{abstract}

\section{Introduction}
After the appearance of the regular Riemann--Hilbert correspondence of Kashiwara \cite{Kas84},
Beilinson and Bernstein developed systematically
a theory of regular holonomic $\D$-modules on smooth algebraic varieties
over the complex number field $\CC$ and
obtained an algebraic version of the Riemann--Hilbert correspondence stated as follows:
Let $X$ be a smooth algebraic variety over $\CC$.
We denote by $X^\an$ the underling complex analytic manifold of $X$,
by $\BDCrh(\D_X)$ the triangulated category of regular holonomic $\D_X$-modules on $X$ and 
by $\BDC_{\CC-c}(\CC_X)$ the one of algebraic $\CC$-constructible sheaves on $X^\an$.
Then there exists an equivalence of triangulated categories
\[\Sol_X \colon \BDCrh(\D_X)^{\op}\simto\BDC_{\CC-c}(\CC_X),
\hspace{17pt}
\Sol_X(\M) := \Sol_{X^\an}(\M^\an),\]
see \cite{Be, Bor} and also \cite{Sai89} for the details.
Here $\Sol_{X^\an}(\cdot) := \rhom_{\D_{X^\an}}(\ \cdot\ , \SO_{X^\an})$
$\colon\BDC(\D_{X^\an})\to\BDC(\CC_{X^\an})$ is the solution functor
on the complex analytic manifold $X^\an$
and $\M^\an$ is the analytification of $\M$ (see \S \ref{sec-alg} for the definition).
The triangulated category $\BDC_{\CC-c}(\CC_X)$ has a t-structure
$\big({}^p\bfD^{\leq0}_{\CC-c}(\CC_X), {}^p\bfD^{\geq0}_{\CC-c}(\CC_X)\big)$
which is called the perverse t-structure by \cite{BBD}, see also \cite[Theorem 8.1.27]{HTT08}.
Let us denote by $$\Perv(\CC_X) :=
{}^p\bfD^{\leq0}_{\CC-c}(\CC_X)\cap {}^p\bfD^{\geq0}_{\CC-c}(\CC_X)$$ its heart
and call an object of $\Perv(\CC_X) $ an algebraic perverse sheaf on $X^\an$.
The above equivalence induces an equivalence of categories between
the abelian category $\Modrh(\D_X)$ of regular holonomic $\D$-modules on $X$
and the one $\Perv(\CC_X)$ of algebraic perverse sheaves on $X^\an$.

On the other hand,
the problem of extending the analytic regular Riemann--Hilbert correspondence of Kashiwara to cover
the case of analytic holonomic $\D$-modules with irregular singularities had been open for 30 years.
After a groundbreaking development in the theory of irregular meromorphic connections by 
Kedlaya \cite{Ked10, Ked11} and Mochizuki \cite{Mochi09, Mochi11},
D'Agnolo and Kashiwara established the Riemann--Hilbert correspondence
for analytic irregular holonomic $\SD$-modules in \cite{DK16}.
For this purpose, they introduced enhanced ind-sheaves extending the classical notion
of ind-sheaves introduced by Kashiwara and Schapira in \cite{KS01}. 
Let $\bfX$ be a complex analytic manifold
(In this paper,
we use bold letters for complex manifolds to avoid confusion with algebraic varieties).
We denote by $\BDChol(\D_{\bfX})$ the triangulated category of
holonomic $\D_\bfX$-modules on $\bfX$ and by $\BEC_{\RR-c}(\I\CC_\bfX)$
the one of $\RR$-constructible enhanced ind-sheaves on $\bfX$ (see Definition \ref{def2.6}).
We set $\Sol_\bfX^{\rmE}(\M) := \rihom_{\D_\bfX}(\M, \SO_\bfX^{\rmE})$.
Here $\SO_\bfX^{\rmE}$ is the enhanced ind-sheaf of tempered holomorphic functions,
see \cite[Definition 8.2.1]{DK16} and also \S\ref{sec-2} for the details.
Then D'Agnolo and Kashiwara proved that
the enhanced solution functor $\Sol_\bfX^{\rmE}$ induces a fully faithful embedding
\[\Sol_\bfX^{\rmE} \colon \BDChol(\D_\bfX)^{\op}\hookrightarrow\BEC_{\RR-c}(\I\CC_\bfX).\]
Moreover, in \cite{DK16-2}
they gave a generalized t-structure
$\big({}^{\frac{1}{2}}\bfE_{\RR-c}^{\leq c}(\I\CC_\bfX),
{}^{\frac{1}{2}}\bfE_{\RR-c}^{\geq c}(\I\CC_\bfX)\big)_{c\in\RR}$
on $\BEC_{\RR-c}(\I\CC_\bfX)$
and proved that
the enhanced solution functor induces a fully faithful embedding
of the abelian category $\Modhol(\D_\bfX)$ of holonomic $\D$-modules on $\bfX$
into its heart ${}^{\frac{1}{2}}\bfE_{\RR-c}^{\leq 0}(\I\CC_\bfX)\cap
{}^{\frac{1}{2}}\bfE_{\RR-c}^{\geq 0}(\I\CC_\bfX)$.
On the other hand, Mochizuki proved that
the image of $\Sol_{\bfX}^{\rmE}$ can be characterized  by the curve test in \cite{Mochi16}.
In \cite{Ito19}, the author defined $\CC$-constructability for enhanced ind-sheaves on $\bfX$
and proved that they are nothing but objects of the image of $\Sol_\bfX^{\rmE}$,
see \S \ref{sec2.8} for the details.
Namely, we obtain an equivalence of categories
between the triangulated category $\BDChol(\D_\bfX)$ of holonomic $\D$-modules on $\bfX$
and the one $\BEC_{\CC-c}(\I\CC_\bfX)$ of $\CC$-constructible enhanced ind-sheaves on $\bfX$:
\[\Sol_\bfX^{\rmE} \colon \BDChol(\D_\bfX)^{\op}\simto \BEC_{\CC-c}(\I\CC_\bfX).\]
Remark that Kuwagaki introduced another approach
to the irregular Riemann--Hilbert correspondence
via irregular constructible sheaves which are defined by $\CC$-constructible sheaf
with coeficients in a finite version of the Novikov ring and special gradings in \cite{Kuwa18}.

Therefore, it seems to be important to establish
an algebraic irregular Riemann--Hilbert correspondence 
on an a smooth algebraic variety.
Although it may be known by experts, it is not in the literature to our knowledge.
Thus, we want to prove the algebraic irregular Riemann--Hilbert correspondence in this paper.
Let $X$ be a smooth algebraic variety over $\CC$ and
denote by $\BDChol(\D_X)$ the triangulated category of holonomic $\D_X$-modules on $X$.
\begin{definition}[Definition \ref{def-AC}]
We say that an enhanced ind-sheaf $K\in\ZEC(\I\CC_{X^\an})$ satisfies the condition $\AC$
if there exists an algebraic stratification $\{X_\alpha\}_{\alpha\in A}$ of $X$ such that
$\pi^{-1}\CC_{(\var{X}^\blow_\alpha)^\an\setminus D_\alpha^\an}
\otimes \bfE (b^\an_\alpha)^{-1}K$
has a modified quasi-normal form along $D_\alpha^\an$
for any $\alpha\in A$ (see Definition \ref{def-modi}),
where $b_\alpha \colon \var{X}^\blow_\alpha \to X$ is a blow-up of $\var{X_\alpha}$
along $\partial X_\alpha \colon= \var{X_\alpha}\setminus X_\alpha$,
$D_\alpha := b_\alpha^{-1}(\partial X_\alpha)$
and $D_\alpha^\an := \big(\var{X}_\alpha^\blow\big)^\an\setminus
\big(\var{X}_\alpha^\blow\setminus D_\alpha\big)^\an$.
\end{definition}
Let us denote by $\BEC_{\CC-c}(\I\CC_X)$ the full triangulated subcategory
of $\BEC(\I\CC_{X^\an})$ consisting of objects whose cohomologies
satisfy the condition $\AC$, see \S\ref{sec-AC} for the details.
Then we obtain an essential surjective functor
\[\Sol_X^\rmE \colon \BDChol(\D_X)^\op\to\BEC_{\CC-c}(\I\CC_X)\]
by Propositions \ref{prop3.4} and \ref{prop3.5}.
This is not fully faithful in general.
However if $X$ is complete it is fully faithful.
\begin{theorem}[Theorem \ref{thm-main-1}]
Let $X$ be a smooth complete algebraic variety over $\CC$.
There exists an equivalence of triangulated categories
\[\Sol_X^{\rmE} \colon \BDChol(\D_X)^{\op}\simto \BEC_{\CC-c}(\I\CC_X).\]
\end{theorem}

Now, let us come back to the general case. 
Thanks to Hironaka's desingularization theorem \cite{Hiro} (see also \cite[Theorem 4.3]{Naga}),
for any smooth algebraic variety $X$ over $\CC$
we can take a smooth complete algebraic variety $\tl{X}$ such that $X\subset \tl{X}$
and $D \colon= \tl{X}\setminus X$ is a normal crossing divisor of $\tl{X}$.
Let us consider a bordered space $X^\an_\infty = (X^\an, \tl{X}^\an)$,
see \S\ref{subsec2.2} for the definition.
We denote by $\BEC(\I\CC_{X^\an_\infty})$
the triangulated category of enhanced ind-sheaves on $X^\an_\infty$,
see \S \ref{subsec2.5} for the details.
Note that $\BEC(\I\CC_{X^\an_\infty})$ does not depend on the choice of $\tl{X}$
and there exists an equivalence of triangulated categories
\[\xymatrix@C=55pt{
\BEC(\I\CC_{X^\an_\infty})\ar@<0.7ex>@{->}[r]^-{\bfE j_{!!}}\ar@{}[r]|-{\sim}
&
\{K\in\BEC(\I\CC_{\tl{X}^\an})\ |\ \pi^{-1}\CC_{X^\an}\otimes K\simto K\}
\ar@<0.7ex>@{->}[l]^-{\bfE j^{-1}},
}\]
where $j\colon X^\an_\infty \to \tl{X}^\an$ is a morphism of bordered spaces 
given by the open embedding $X\hookrightarrow \tl{X}$,
see \S \ref{subsec2.5} for the details.
We shall denote the open embedding $X\hookrightarrow \tl{X}$ by the same symbol $j$
and set 
\[\Sol_{X_\infty}^\rmE(\M) := 
\bfE j^{-1}\Sol_{\tl{X}}^\rmE(\bfD j_\ast\M)
\in\BEC(\I\CC_{X^\an_\infty})\]
for any $\M\in\BDC(\D_X)$.

\begin{definition}[Definition \ref{def3.11}]
We say that an enhanced ind-sheaf $K\in\BEC(\I\CC_{X^\an_\infty})$ is
algebraic $\CC$-constructible on $X_\infty^\an$
if $\bfE j_{!!}K \in\BEC(\I\CC_{\tl{X}^\an})$ is an object of $\BEC_{\CC-c}(\I\CC_{\tl{X}})$.

Let us denote by $\BEC_{\CC-c}(\I\CC_{X_\infty})$ the triangulated category of them.
\end{definition}

The following result is the main theorem of this paper:
\begin{theorem}[Theorem \ref{thm-main-2}]
For any $\M\in\BDChol(\D_X)$, the enhanced solution complex $\Sol_{X_\infty}^\rmE(\M)$
of $\M$ is an algebraic $\CC$-constructible enhanced ind-sheaf.
On the other hand, 
for any algebraic $\CC$-constructible enhanced ind-sheaf $K\in\BEC_{\CC-c}(\I\CC_{X_\infty})$,
there exists $\M\in\BDChol(\D_X)$
such that $$K\simeq \Sol_{X_\infty}^{\rmE}(\M).$$
Moreover, we obtain an equivalence of triangulated categories
\[\Sol_{X_\infty}^{\rmE} \colon \BDChol(\D_X)^{\op}\simto \BEC_{\CC-c}(\I\CC_{X_\infty}).\]
\end{theorem}

Moreover we show that there exists a t-structure
on the triangulated category $\BEC_{\CC-c}(\I\CC_{X_\infty})$
whose heart is equivalent to the abelian category $\Modhol(\D_X)$
of holonomic $\D$-modules on $X$ as follows.
Let us denote by 
$\sh_{X^\an_\infty} \colon\BEC(\I\CC_{X^\an_\infty}) \to \BDC(\CC_{X^\an})$
the sheafification functor on $X^\an_\infty$ and
by $\rmD_{X^\an_\infty}^{\rmE} \colon
\BEC(\I\CC_{X^\an_\infty})^{\op}\to\BEC(\I\CC_{X^\an_\infty})$
the duality functor for enhanced ind-sheaves on $X^\an_\infty$,
see \S\ref{subsec2.4} for the details.
We set 
\begin{align*}
{}^p\bfE^{\leq0}_{\CC-c}(\I\CC_{X_\infty}) & :=
\{K\in\BEC_{\CC-c}(\I\CC_{X_\infty})\ |\
\sh_{X^\an_\infty}(K)\in{}^p\bfD^{\leq0}_{\CC-c}(\CC_X)\},\\
{}^p\bfE^{\geq0}_{\CC-c}(\I\CC_{X_\infty}) & :=
\{K\in\BEC_{\CC-c}(\I\CC_{X_\infty})\ |\
\rmD_{X^\an_\infty}^{\rmE}(K)\in{}^p\bfE^{\leq0}_{\CC-c}(\I\CC_{X_\infty})\}\\
&\ =
\{K\in\BEC_{\CC-c}(\I\CC_{X_\infty})\ |\
\sh_{X^\an_\infty}(K)\in{}^p\bfD^{\geq0}_{\CC-c}(\CC_X)\},
\end{align*}
where the pair $\big({}^p\bfD^{\leq0}_{\CC-c}(\CC_X),
{}^p\bfD^{\geq0}_{\CC-c}(\CC_X)\big)$
is the perverse t-structure on $\BDC_{\CC-c}(\CC_X)$.
Then we obtain the second main theorem of this paper:
\begin{theorem}[Theorem \ref{thm-main-3}]
The pair $\big({}^p\bfE^{\leq0}_{\CC-c}(\I\CC_{X_\infty}),
{}^p\bfE^{\geq0}_{\CC-c}(\I\CC_{X_\infty})\big)$
is a t-structure on $\BEC_{\CC-c}(\I\CC_{X_\infty})$
and
its heart $$\Perv(\I\CC_{X_\infty}) :=
{}^p\bfE^{\leq0}_{\CC-c}(\I\CC_{X_\infty})\cap{}^p\bfE^{\geq0}_{\CC-c}(\I\CC_{X_\infty})$$
is equivalent to the abelian category $\Modhol(\D_X)$
of holonomic $\D$-modules on $X$.
\end{theorem}
Furthermore the pair $\big({}^p\bfE^{\leq0}_{\CC-c}(\I\CC_{X_\infty}),
{}^p\bfE^{\geq0}_{\CC-c}(\I\CC_{X_\infty})\big)$
is related to the generalized t-structure
$\big({}^{\frac{1}{2}}\bfE_{\RR-c}^{\leq c}(\I\CC_{X^\an_\infty}),
{}^{\frac{1}{2}}\bfE_{\RR-c}^{\geq c}(\I\CC_{X^\an_\infty})\big)_{c\in\RR}$
on $\BEC_{\RR-c}(\I\CC_{X^\an_\infty})$ as follows:
\begin{align*}
{}^p\bfE_{\CC-c}^{\leq 0}(\I\CC_{X_\infty}) &=
{}^{\frac{1}{2}}\bfE_{\RR-c}^{\leq 0}(\I\CC_{X^\an_\infty})\cap
\BEC_{\CC-c}(\I\CC_{X_\infty}),\\
{}^p\bfE_{\CC-c}^{\geq 0}(\I\CC_{X_\infty}) &=
{}^{\frac{1}{2}}\bfE_{\RR-c}^{\geq 0}(\I\CC_{X^\an_\infty})\cap
\BEC_{\CC-c}(\I\CC_{X_\infty}).
\end{align*}

\begin{remark}
In the case of quasi-projective variety, Kuwagaki \cite{Kuwa18} established an algebraic version
of the irregular Riemann--Hilbert correspondence.
\end{remark}

Furthermore, let us consider simple objects on $\Perv(\I\CC_{X_\infty})$
for a smooth algebraic variety $X$ over $\CC$.
Thanks to the theory of minimal extensions of holonomic $\D$-modules,
by using Theorem 1.5 (Theorem \ref{thm-main-3})
we can describe simple objects of $\Perv(\I\CC_{X_\infty})$ as below.

\begin{definition}[Definition \ref{def-simple}]
Let $X$ be a smooth algebraic variety over $\CC$.
A non-zero algebraic enhanced ind-sheaf $K\in\Perv(\I\CC_{X_\infty})$ is called simple
if it contains no subobjects in $\Perv(\I\CC_{X_\infty})$ other than $K$ or $0$.
\end{definition}

Note that for any simple algebraic perverse sheaf $\SF\in\Perv(\CC_X)$
the natural embedding $e_{X^\an_\infty}(\SF)\in\Perv(\I\CC_{X_\infty})$
is also simple by Proposition \ref{prop3.35}.

In this paper, we shall say that $K\in\ZEC(\I\CC_{X^\an_\infty})$ is an enhanced local system on $X_\infty$
if for any $x\in X$ there exist an open neighborhood $U\subset X$ of $x$ and a non-negative integer $k$
such that $K|_{U_\infty^\an}\simeq(\CC_{U^\an_\infty}^\rmE)^{\oplus k}.$
Note that for any enhanced local system $K$ on $X_\infty$, 
we have $K[d_X]\in\Perv(\I\CC_{X_\infty})$. 

\begin{proposition}[Proposition \ref{prop3.37} (2)]
For any simple object $K$ of $\Perv(\I\CC_{X_\infty})$,
there exist a locally closed smooth connected subvariety $Z$ of $X$
and a simple enhanced local system $L$ on $Z_\infty$ such that
the natural embedding $i_Z\colon Z\hookrightarrow X$ is affine 
and $$K\simeq\Image\big(\bfE i_{Z^\an_\infty!!}L[d_Z]\to\bfE i_{Z^\an_\infty\ast}L[d_Z]\big),$$
where a morphism $\bfE i_{Z^\an_\infty!!}L[d_Z]\to\bfE i_{Z^\an_\infty\ast}L[d_Z]$ in $\Perv(\I\CC_{X_\infty})$
is induced by a canonical morphism $\bfE i_{Z^\an_\infty!!}\to\bfE i_{Z^\an_\infty\ast}$ of functors
$\bfE i_{Z^\an_\infty!!}, \bfE i_{Z^\an_\infty\ast}\colon \Perv(\I\CC_{Z_\infty})\to\Perv(\I\CC_{X_\infty})$.
\end{proposition}

Moreover we shall consider the image of a canonical morphism
\[{}^p\bfE i_{Z^\an_\infty!!}K\to{}^p\bfE i_{Z^\an_\infty\ast}K\]
for $K\in\Perv(\I\CC_{Z_\infty})$ and a locally closed smooth subvariety $Z$ of $X$
(not necessarily the natural embedding $i_Z\colon Z\hookrightarrow X$ is affine)
where we set 
\begin{align*}
{}^p\bfE i_{Z^\an_\infty\ast} := 
{}^p\SH^0\circ\bfE i_{Z^\an_\infty\ast}\colon\Perv(\I\CC_{Z_\infty})\to\Perv(\I\CC_{X_\infty}).\\
{}^p\bfE i_{Z^\an_\infty!!} := 
{}^p\SH^0\circ\bfE i_{Z^\an_\infty!!}\colon\Perv(\I\CC_{Z_\infty})\to\Perv(\I\CC_{X_\infty}).
\end{align*}

In this paper,
we shall define a minimal extension of $K\in\Perv(\I\CC_{Z_\infty})$ along $Z$ as follows.
\begin{definition}[Definition \ref{def-minimal}]
In the situation as above,
for any $K\in\Perv(\I\CC_{Z_\infty})$,
we call the image of the canonical morphism $${}^p\bfE i_{Z^\an_\infty!!}K\to{}^p\bfE i_{Z^\an_\infty\ast}K$$
the minimal extension of $K$ along $Z$,
and denote it by ${}^p\bfE i_{Z^\an_\infty!!\ast}K$.
\end{definition}

Note that if $Z$ is closed then
there exists an isomorphism ${}^p\bfE i_{Z^\an_\infty!!\ast}K
\simeq\bfE i_{Z^\an_\infty\ast}K\simeq\bfE i_{Z^\an_\infty!!}K$
in $\Perv(\I\CC_{X_\infty})$
and there exists an equivalence of categories:
\[\xymatrix@C=75pt{
\Perv_{Z}(\I\CC_{X_\infty})\ar@<0.7ex>@{->}[r]^-{{}^p\bfE i_{Z^\an_\infty}^{-1}}\ar@{}[r]|-{\sim}
&
\Perv(\I\CC_{Z_\infty})
\ar@<0.7ex>@{->}[l]^-{\bfE i_{Z^\an_\infty!!}},
}\]
where $\Perv_{Z}(\I\CC_{X_\infty})$ is a full subcategory of $\Perv(\I\CC_{X_\infty})$
consisting of objects whose support is contained in $Z^\an$.
This is nothing but a counter part of the Kashiwara's equivalence of categories of a holonomic case,
see Proposition \ref{prop3.29} and Remark \ref{rem3.31} for details.
Thus the minimal extension ${}^p\bfE i_{Z^\an_\infty!!\ast}K$
of a simple object $K$ of $\Perv(\I\CC_{Z_\infty})$ along a closed smooth subvariety $Z$ is also simple.

On the other hand, if $U$ is open whose complement $X\setminus U$ is a smooth subvariety,
then the minimal extension along $U$ is characterized as follows:
\begin{proposition}[Proposition \ref{prop3.41}]
In the situation as above,
the minimal extension ${}^p\bfE i_{UZ^\an_\infty!!\ast}K$ of $K\in\Perv(\I\CC_{U_\infty})$ along $U$
is characterized as the unique algebraic enhanced perverse ind-sheaf $L$ on $X_\infty$ satisfying the conditions
\begin{itemize}
\setlength{\itemsep}{-2pt}
\item[\rm(1)]
$\bfE i_{U^\an_\infty}^{-1}L\simeq K$,
\item[\rm(2)]
$\bfE i_{W^\an_\infty}^{-1}L\in\bfE^{\leq-1}_{\CC-c}(\I\CC_{W_\infty})$,
\item[\rm(3)]
$\bfE i_{W^\an_\infty}^{!}L\in\bfE^{\geq1}_{\CC-c}(\I\CC_{W_\infty})$.
\end{itemize}
\end{proposition}
Furthermore in the situation as above, the minimal extension ${}^p\bfE i_{U^\an_\infty!!\ast}K$
of a simple object $K$ of $\Perv(\I\CC_{U_\infty})$ along $U$ is also simple by Corollary \ref{cor3.45}.

Therefore by using Lemma \ref{lem3.41}, we obtain the following results.
\begin{theorem}[Theorem \ref{thm-main-4} (2)]
Let $X$ be a smooth algebraic variety over $\CC$ and $Z$ a locally closed smooth subvariety of $X$
$($not necessarily the natural embedding $i_Z\colon Z\hookrightarrow X$ is affine$)$.
We assume that $Z = U\cap W$
where $U\subset X$ is an open subset whose complement  $X\setminus U$ is smooth
and $W\subset X$ is a closed subvariety.

Then the minimal extension ${}^p\bfE i_{Z^\an_\infty!!\ast}K$
of a simple object $K$ of $\Perv(\I\CC_{Z_\infty})$ along $Z$ is also simple.
\end{theorem}

\section*{Acknowledgement}
I would like to thank Dr. Tauchi of Kyushu University
for many discussions and giving many comments.

\section{Preliminary Notions and Results}\label{sec-2}
In this section,
we briefly recall some basic notions
and results which will be used in this paper. 

\subsection{Ind-Sheaves}
Let us recall some basic notions on ind-sheaves.
References are made to Kashiwara-Schapira \cite{KS01, KS06}. 

Let $M$ be a good topological space
(i.e., a locally compact Hausdorff space
which is countable at infinity and has finite soft dimension). 
We denote by $\Mod(\CC_M)$
the abelian category of sheaves of $\CC$-vector spaces on $M$
and by $\I\CC_M$ that of ind-sheaves on it.
Then there exists
a natural exact embedding $\iota_M \colon \Mod(\CC_M)\hookrightarrow\I\CC_M$,
which we sometimes omit.
It has an exact left adjoint $\alpha_M$, 
that has in turn an exact fully faithful
left adjoint functor $\beta_M$.
Let us denote by $\BDC(\I\CC_M)$ the derived category of ind-sheaves on $M$.
For a continuous map $f \colon M\to N$,
we have the Grothendieck operations
$\otimes$, $\rihom$, $f^{-1}$, $\rmR f_\ast$, $f^!$, $\rmR f_{!!}$.  
We also set $\rhom := \alpha_M\circ\rihom$.

\subsection{Ind-Sheaves on Bordered Spaces}\label{subsec2.2}
We shall recall a notion of ind-sheaves on a bordered space.
For the details, we refer to D'Agnolo-Kashiwara \cite{DK16}.

A bordered space is a pair $M_{\infty} = (M, \che{M})$ of
a good topological space $\che{M}$ and an open subset $M\subset\che{M}$.
A morphism $f \colon (M, \che{M})\to (N, \che{N})$ of bordered spaces
is a continuous map $f \colon M\to N$ such that the first projection
$\che{M}\times\che{N}\to\che{M}$ is proper on
the closure $\var{\Gamma}_f$ of the graph $\Gamma_f$ of $f$ 
in $\che{M}\times\che{N}$. 
We say that
the morphism $f \colon (M, \che{M})\to (N, \che{N})$ of bordered spaces is semi-proper
if the second projection $\che{M}\times\che{N}\to\che{N}$ is proper on $\var{\Gamma}_f$.
Moreover the morphism $f \colon (M, \che{M})\to (N, \che{N})$ of bordered spaces is proper
if it is semi-proper and $f\colon M\to N$ is proper.
Note that if a continuous map $f \colon M\to N$ extends
to a continuous map $\che{f} \colon \che{M}\to\che{N}$,
then $f$ induces a morphism $(M, \che{M})\to(N, \che{N})$ of bordered spaces.

The category of good topological spaces is embedded into that
of bordered spaces by the identification $M = (M, M)$. 
Note that we have the morphism $j_{M_\infty} \colon M_\infty\to \che{M}$
of bordered spaces given by the embedding $M\hookrightarrow \che{M}$.
We sometimes denote $j_{M_\infty}$ by $j$ for short.
For a locally closed subset $Z\subset M$ of $M$,
we set $Z_\infty \colon= (Z, \var{Z})$ where $\var{Z}$ is the closure of $Z$ in $\che{M}$
and denote by $i_{Z_\infty} \colon Z_\infty\to \var{Z}$ the morphism of bordered spaces
given by the embedding $Z\hookrightarrow \var{Z}$.

Let us denote by $\I\CC_{M_\infty}$ the abelian category of 
 ind-sheaves on a bordered spaces $M_{\infty} = (M, \che{M})$
 and denote by $\BDC(\I\CC_{M_\infty})$ the triangulated category of them.
Note that there exists the standard t-structure
$\big(\bfD^{\leq 0}(\I\CC_{M_\infty}), 
\bfD^{\geq 0}(\I\CC_{M_\infty})\big)$
 of $\BDC(\I\CC_{M_\infty})$
which is induced by the standard t-structure on $\BDC(\I\CC_{\che{M}})$. 
For a morphism $f \colon M_\infty\to N_\infty$ 
of bordered spaces, 
we have the Grothendieck operations 
$ \otimes, \rihom, \rmR f_\ast, \rmR f_{!!}, f^{-1}, f^! $.
Note that there exists an embedding functor 
$\iota_{M_\infty} \colon \BDC(\CC_M) \hookrightarrow \BDC(\I\CC_{M_\infty})$.
We sometimes write $\BDC(\CC_{M_\infty})$ for $\BDC(\CC_{M})$,
when considered as a full subcategory of $\BDC(\I\CC_{M_\infty})$.
Note also that the embedding functor $\iota_{M_\infty}$ has a left adjoint functor
$\alpha_{M_\infty} \colon \BDC(\I\CC_{M_\infty})\to \BDC(\CC_M)$.
Remark that if $f\colon M_\infty\to N_\infty$ is semi-proper then there exists an isomorphism
$\bfR f_{!!}\circ\iota_{M_\infty}\simto\iota_{N_\infty}\circ\bfR f_!$ of functors
$\BDC(\CC_{M})\to\BDC(\I\CC_{N_\infty})$.

\subsection{Enhanced Ind-Sheaves on Bordered Spaces I}\label{subsec2.4}
We shall recall some basic notions of enhanced ind-sheaves on bordered spaces and results on it.
Reference are made to \cite{KS16-2} and \cite{DK16-2}.
Moreover we also refer to D'Agnolo-Kashiwara \cite{DK16} and Kashiwara-Schapira \cite{KS16}
for the notions of enhanced ind-sheaves on good topological spaces.

Let $M_\infty = (M, \che{M})$ be a bordered space.
We set $\RR_\infty := (\RR, \var{\RR})$ for 
$\var{\RR} := \RR\sqcup\{-\infty, +\infty\}$,
and let $t\in\RR$ be the affine coordinate. 
We consider the morphisms of bordered spaces
\[M_\infty\times\RR^2_\infty\xrightarrow{p_1,\ p_2,\ \mu}M_\infty
\times\RR_\infty\overset{\pi}{\longrightarrow}M_\infty\]
given by the maps $p_1(x, t_1, t_2) := (x, t_1)$, $p_2(x, t_1, t_2) := (x, t_2)$,
$\mu(x, t_1, t_2) := (x, t_1+t_2)$ and $\pi (x,t) := x$. 
Then the convolution functors for 
ind-sheaves on $M_\infty \times \RR_\infty$ are defined by
\begin{align*}
F_1\Potimes F_2 & := \rmR\mu_{!!}(p_1^{-1}F_1\otimes p_2^{-1}F_2),\\
\Prihom(F_1, F_2) & := \rmR p_{1\ast}\rihom(p_2^{-1}F_1, \mu^!F_2).
\end{align*}
Now we define the triangulated category 
of enhanced ind-sheaves on a bordered space $M_\infty$ by 
$$\BEC(\I\CC_{M_\infty}) :=
\BDC(\I\CC_{M_\infty \times\RR_\infty})/\pi^{-1}\BDC(\I\CC_{M_\infty}).$$
The quotient functor
\[\Q_{M_\infty} \colon \BDC(\I\CC_{M_\infty\times\RR_\infty})\to\BEC(\I\CC_{M_\infty})\]
has fully faithful left and right adjoints 
$\bfL_{M_\infty}^\rmE,\bfR_{M_\infty}^\rmE \colon
\BEC(\I\CC_{M_\infty}) \to\BDC(\I\CC_{M_\infty\times\RR_\infty})$
 defined by 
\[\bfL_{M_\infty}^\rmE(\Q_{M_\infty}(F)) := (\CC_{\{t\geq0\}}\oplus\CC_{\{t\leq 0\}})
\Potimes F ,\hspace{10pt} \bfR_{M_\infty}^\rmE(\Q_{M_\infty}(F)) 
:=\Prihom(\CC_{\{t\geq0\}}\oplus\CC_{\{t\leq 0\}}, F).\]
We sometimes denote $\Q_{M_\infty}$ (resp.\ $\bfL_{M_\infty}^\rmE, \bfR_{M_\infty}^\rmE$ )
by $\Q$ (resp.\ $\bfL^\rmE, \bfR^\rmE$) for short.
Moreover they induce equivalences of categories
\begin{align*}
\bfL_{M_\infty}^\rmE &\colon \BEC(\I\CC_{M_\infty})\simto
\{F\in\BDC(\I\CC_{M_\infty\times\RR_\infty})\ |\
(\CC_{\{t\geq0\}}\oplus\CC_{\{t\leq0\}})\Potimes F\simto F\},\\
\bfR_{M_\infty}^\rmE &\colon \BEC(\I\CC_{M_\infty})\simto
\{F\in\BDC(\I\CC_{M_\infty\times\RR_\infty})\ |\ F\simto
\Prihom(\CC_{\{t\geq0\}}\oplus\CC_{\{t\leq0\}}, F)\},
\end{align*}
respectively,
where $\{t\geq0\}$ stands for $\{(x, t)\in M\times\RR\ |\ t\geq0\}\subset \che{M}\times\var{\RR}$ 
and $\{t\leq0\}$ is defined similarly.
Then we have the following standard t-structure on $\BEC(\I\CC_{M_\infty})$
which is induced by the standard t-structure on $\BDC(\I\CC_{M_\infty\times\RR_\infty})$
\begin{align*}
\bfE^{\leq 0}(\I\CC_{M_\infty}) & = \{K\in \BEC(\I\CC_{M_\infty})\ | \ 
\bfL_{M_\infty}^{\rmE}K\in \bfD^{\leq 0}(\I\CC_{M_\infty\times\RR_\infty})\},\\
\bfE^{\geq 0}(\I\CC_{M_\infty}) & = \{K\in \BEC(\I\CC_{M_\infty})\ | \ 
\bfL_{M_\infty}^{\rmE}K\in \bfD^{\geq 0}(\I\CC_{M_\infty\times\RR_\infty})\}.
\end{align*}
We denote by 
\[\SH^n \colon \BEC(\I\CC_{M_\infty})\to\bfE^0(\I\CC_{M_\infty})\]
the $n$-th cohomology functor, where we set 
$\bfE^0(\I\CC_{M_\infty}) :=
\bfE^{\leq 0}(\I\CC_{M_\infty})\cap\bfE^{\geq 0}(\I\CC_{M_\infty})$.

The convolution functors are also defined for enhanced ind-sheaves on bordered spaces.
We denote them by the same symbols $\Potimes$, $\Prihom$. 
For a morphism $f \colon M_\infty \to N_\infty $ of bordered spaces, we 
can define also the operations 
$\bfE f^{-1}$, $\bfE f_\ast$, $\bfE f^!$, $\bfE f_{!!}$ 
for enhanced ind-sheaves on bordered spaces. 
For example, 
by the natural morphism $f_{\RR_\infty}\colon M_\infty \times \RR_{\infty} \to 
N_\infty \times \RR_{\infty}$ of bordered spaces associated to 
$f$ we set $\bfE f_\ast \big( \Q_{M_\infty}F\big)=
\Q_{N_\infty}\big(\rmR f_{\RR_\infty\ast}F\big)$
for $F\in\BDC(\I\CC_{M_\infty\times\RR_\infty})$. 
The other operations are defined similarly. 
We thus obtain the six operations $\Potimes$, $\Prihom$,
$\bfE f^{-1}$, $\bfE f_\ast$, $\bfE f^!$, $\bfE f_{!!}$ 
for enhanced ind-sheaves on bordered spaces.
Note that there exists a morphism $\bfE f_{!!}\to\bfE f_\ast$
of functors
$\BEC(\I\CC_{M_\infty})\to\BEC(\I\CC_{N_\infty})$
and it is an isomorphism if $f$ is proper.
Moreover we have outer-hom functors
$\rihom^\rmE(K_1, K_2)$,
$\rhom^\rmE(K_1, K_2) := \alpha_{M_\infty}\rihom^\rmE(K_1, K_2)$,
$\rHom^\rmE(K_1, K_2) := \bfR\Gamma\big(M; \rhom^\rmE(K_1, K_2)\big)$
with values in $\BDC(\I\CC_{M_\infty})$, $\BDC(\CC_M)$ and $\BDC(\CC)$, respectively. 
Here, $\BDC(\CC)$ is the derived category of $\CC$-vector spaces.

For $F\in\BDC(\I\CC_{M_\infty})$ and $K\in\BEC(\I\CC_{M_\infty})$ the objects 
\begin{align*}
\pi^{-1}F\otimes K & :=\Q_{M_\infty}(\pi^{-1}F\otimes \bfL_{M_\infty}^\rmE K),\\
\rihom(\pi^{-1}F, K) & :=\Q_{M_\infty}\big(\rihom(\pi^{-1}F, \bfR_{M_\infty}^\rmE K)\big). 
\end{align*}
in $\BEC(\I\CC_{M_\infty})$ are well defined. 
We set $$\CC_{M_\infty}^\rmE := \Q_{M_\infty} 
\Bigl(``\underset{a\to +\infty}{\varinjlim}"\ \CC_{\{t\geq a\}}
\Bigr)\in\BEC(\I\CC_{M_\infty}).$$
Note that there exists an isomorphism
$\CC_{M_\infty}^\rmE\simeq\bfE j^{-1}\CC_{\che{M}}^\rmE$
in $\BEC(\I\CC_{M_\infty})$.
Then we have a natural embedding
$e_{M_\infty} \colon \BDC(\I\CC_{M_\infty}) \to \BEC(\I\CC_{M_\infty})$
defined by \[e_{M_\infty}(F) :=  \CC_{M_\infty}^\rmE\otimes\pi^{-1}F.\]
Note also that for a morphism $f\colon M_\infty\to N_\infty$ of bordered spaces
and objects $F\in\BDC(\I\CC_{M_\infty})$, $G\in\BDC(\I\CC_{N_\infty})$
we obtain 
\begin{align*}
\bfE f_{!!}(e_{M_\infty}F)
&\simeq
e_{N_\infty}(\bfR f_{!!}F),\\
\bfE f^{-1}(e_{N_\infty}G)
&\simeq
e_{M_\infty}(f^{-1}G),\\
\bfE f^{!}(e_{N_\infty}G)
&\simeq
e_{M_\infty}(f^{!}G)
\end{align*}
by using \cite[Proposition 2.18]{KS16-2}.
In particular we have an isomorphism in $\BEC(\I\CC_{M_\infty})$
\[\bfE j^{-1}(e_{\che{M}}F)\simeq e_{M_\infty}(j^{-1}F)\]
for $F\in\BDC(\I\CC_{\che{M}})$.
Let us define $\omega_{M_\infty}^\rmE := e_{M_\infty}(\omega_M)\in\BEC(\I\CC_{M_\infty})$
where $\omega_M\in\BDC(\CC_{M_\infty}) ( = \BDC(\CC_M))$ is the dualizing complex,
see \cite[Definition 3.1.16]{KS90} for the details.
Then we have the Verdier duality functor
$\rmD_{M_\infty}^\rmE \colon\BEC(\I\CC_{M_\infty})^{\op}\to\BEC(\I\CC_{M_\infty}) $
 for enhanced ind-sheaves on bordered spaces which is defined by
 \[\rmD_{M_\infty}^\rmE(K) := \Prihom(K, \omega_{M_\infty}^\rmE).\]
 Note that we have an isomorphism in $\BEC(\I\CC_{M_\infty})$
\[\rmD_{M_\infty}^\rmE(\bfE j^{-1}K)\simeq\bfE j^{-1}(\rmD_{\che{M}}^\rmE(K))\]
for any $K\in\BEC(\I\CC_{\che{M}})$.
Note also that there exists an isomorphism in $\BEC(\I\CC_{M_\infty})$
\[\rmD_{M_\infty}^\rmE(e_{M_\infty}\SF)
\simeq
e_{M_\infty}(\rmD_M\SF)\]
for any $\SF\in\BDC(\CC_M)$.
Let $i_0 \colon M_\infty\to M_\infty\times\RR_\infty$ be the inclusion map of bordered spaces
induced by $x\mapsto (x, 0)$.
We set
\[\sh_{M_\infty} := \alpha_{M_\infty}\circ i_0^!\circ \bfR_{M_\infty}^{\rmE} \colon
\BEC(\I\CC_{M_\infty}) \to \BDC(\CC_{M}) \]
and call it the sheafification functor for enhanced ind-sheaves on bordered spaces.
Note that we have an isomorphism in $\BDC(\CC_{M})$
\[\sh_{M_\infty}(\bfE j^{-1}K)\simeq j^{-1}(\sh_{\che{M}}(K))\]
for $K\in\BEC(\I\CC_{\che{M}})$.
Note also that there exists an isomorphism $\SF\simto \sh_{M_\infty}(e_{M_\infty}(\SF))$
for $\SF\in\BDC(\CC_M)$.

For a continuous function $\varphi \colon U\to \RR$ defined on an open subset $U\subset M$,
we set the exponential enhanced ind-sheaf by 
\[\EE_{U|M_\infty}^\varphi := 
\CC_{M_\infty}^\rmE\Potimes
\Q_{M_\infty}\big(\CC_{\{t+\varphi\geq0\}}\big)
, \]
where $\{t+\varphi\geq0\}$ stands for 
$\{(x, t)\in \che{M}\times\var{\RR}\ |\ t\in\RR, x\in U, 
t+\varphi(x)\geq0\}$.

\subsection{Enhanced Ind-Sheaves on Bordered Spaces II}\label{subsec2.5}
The aim of this subsection is to prepare some auxiliary results
on enhanced ind-sheaves on bordered spaces which will be used in \S \ref{sec3}.
In particular we shall prove that for any smooth algebraic variety $X$
the triangulated category $\BEC(\I\CC_{(X^\an, \tl{X}^\an)})$
does not depend on the choice of $\tl{X}$.
Although this is known by experts, we give two proofs.

Let $M_\infty = (M, \che{N})$, $N_\infty = (N, \che{N})$ be two bordered spaces.
\begin{sublemma}\label{sublem2.1}
Let $f \colon M\to N$ be a continuous map and assume that $\che{M}$ and $\che{N}$ are compact.
Then the map $f$ induces a semi-proper morphism from $M_\infty$ to $N_\infty$.
\end{sublemma}
\begin{proof}
Let us denote by $\Gamma_f$ the graph of $f$
and by $\var{\Gamma}_f$ the closure of $\Gamma_f$ in $\che{M}\times\che{N}$.

First we shall prove that the map $f$ induces a morphism of bordered spaces.
Namely let us prove that 
the restriction of the first projection ${\rm pr}_1\colon \che{M}\times\che{N}\to\che{M}$
to $\var{\Gamma}_f$ is proper.
It is enough to show that
$\var{\Gamma}_f\cap{\rm pr}_1^{-1}(K)$ is compact
for any compact subset $K$ of $\che{M}$.
By the assumption,
$\che{M}\times\che{N}$ is compact and hence $\var{\Gamma}_f$ is also compact.
Moreover since $\che{N}$ is compact,
${\rm pr}_1^{-1}(K) = K\times \che{N}$ is also compact
for any compact subset $K$ of $\che{M}$.
Therefore $\var{\Gamma}_f\cap{\rm pr}_1^{-1}(K)$ is compact
for any compact subset $K$ of $\che{M}$.

In a similar way, we can prove that 
the restriction of the second projection ${\rm pr}_2\colon \che{M}\times\che{N}\to\che{N}$
to $\var{\Gamma}_f$ is proper.
Hence the morphism induced by the continuous map $f$ is semi-proper
(see \S \ref{subsec2.2} for definition).
\end{proof}

\begin{lemma}\label{lem2.2}
In the situation of Sublemma \ref{sublem2.1},
we assume that the continuous map $f$ is an isomorphism.
Then the morphism induced by the map $f$ is also isomorphism
between $M_\infty$ and $N_\infty$.
\end{lemma}
\begin{proof}
Let $g \colon N\to M$ be the inverse map of the isomorphism $f \colon M\simto N$.
By Sublemma \ref{sublem2.1},
we obtain morphisms $f_\infty \colon M_\infty\to N_\infty$ (resp.\ $g_\infty \colon N_\infty\to M_\infty$)
of bordered spaces which is induced by the continuous map $f \colon M\to N$ (resp.\ $g \colon N\to M$).
Then it is obvious $g_\infty\circ f_\infty = (g\circ f)_\infty = \id_{M_\infty}$
and $f_\infty\circ g_\infty = (f\circ g)_\infty = \id_{N_\infty}$,
where we used \cite[Lemma 3.2.3]{DK16}.
\end{proof}

The following result is proved by D'Agnolo and Kashiwara.
\begin{lemma}[{\cite[Lemma 2.7.6]{DK16-2}}]
Let $Z\subset M$ be a locally closed subset of $M$
and consider the morphism $i_{Z_\infty} \colon Z_\infty\to M_\infty$ of bordered spaces
given by the embedding $Z\hookrightarrow M$.
For $K\in\BEC(\I\CC_{M_\infty})$ we have isomorphisms in $\BEC(\I\CC_{M_\infty})$
\begin{align*}
\pi^{-1}\CC_Z\otimes K &\simeq \bfE i_{Z_\infty!!}\bfE i_{Z_\infty}^{-1}K,\\
\rihom(\pi^{-1}\CC_Z, K) &\simeq \bfE i_{Z_\infty\ast}\bfE i_{Z_\infty}^{!}K.
\end{align*}
\end{lemma}

In particular, for any $K\in \BEC(\I\CC_{M_\infty})$
we have $\pi^{-1}\CC_{M}\otimes K \simeq \bfE j_{!!}\bfE j^{-1}K$,
where $j \colon M_\infty\to \che{M}$ is the morphism of bordered spaces
given by the embedding $M\hookrightarrow \che{M}$.
Hence we have an equivalence of triangulated categories
\[\xymatrix@C=55pt{
\BEC(\I\CC_{M_\infty})\ar@<0.7ex>@{->}[r]^-{\bfE j_{!!}}\ar@{}[r]|-{\sim}
&
\{K\in\BEC(\I\CC_{\che{M}})\ |\ \pi^{-1}\CC_{M}\otimes K\simto K\}
\ar@<0.7ex>@{->}[l]^-{\bfE j^{-1}}.
}\]

We shall denote $j_{M_\infty} \colon M_\infty\to \che{M}$, $j_{N_\infty} \colon N_\infty\to \che{N}$
by $j_M, j_N$ for simplicity.
Then Sublemma \ref{sublem2.2} below follows from \cite[Lemma 3.3.12]{DK16}.
\begin{sublemma}\label{sublem2.2}
Let $f \colon M_\infty \to N_\infty $ be  the morphism of bordered spaces
associated with a continuous map $\che{f} \colon\che{M}\to \che{N}$ such that $\che{f}(M)\subset N$.
\begin{itemize}
\item[\rm(1)]
For any $K\in\BEC(\I\CC_{M_\infty})$ there exist isomorphisms in $\BEC(\I\CC_{N_\infty})$
\begin{align*}
\bfE f_{!!}K &\simeq \bfE j_N^{-1}\bfE \che{f}_{!!}\bfE j_{M!!}K,\\
\bfE f_{\ast}K &\simeq \bfE j_N^{-1}\bfE \che{f}_{\ast}\bfE j_{M\ast}K.
\end{align*}

\item[\rm(2)]
For any $L\in\BEC(\I\CC_{N_\infty})$ there exist isomorphisms in $\BEC(\I\CC_{M_\infty})$
\begin{align*}
\bfE f^{-1}L &\simeq \bfE j_M^{-1}\bfE \che{f}^{-1}\bfE j_{N!!}L
\simeq \bfE j_M^{-1}\bfE \che{f}^{-1}\bfE j_{N\ast}L,\\
\bfE f^{!}L &\simeq \bfE j_M^{-1}\bfE \che{f}^{!}\bfE j_{N!!}L
\simeq \bfE j_M^{-1}\bfE \che{f}^{!}\bfE j_{N\ast}L.
\end{align*}
\end{itemize}
\end{sublemma}

Remark that $\che{M}$ and $\che{N}$ are not necessary compact. 
Hence  we obtain:
\begin{lemma}\label{lem2.3}
In the situation of Sublemma \ref{sublem2.2}
(not necessarily $\che{M}$ and $\che{N}$ are compact),
we assume that the restriction $\che{f}|_M$ of $\che{f}$ to $M$ induces
an isomorphism $M \simto N$.
Then there exists an equivalence of triangulated categories
\[\xymatrix@C=55pt{
\BEC(\I\CC_{M_\infty})\ar@<0.7ex>@{->}[r]^-{\bfE f_{!!}}\ar@{}[r]|-{\sim}
&
\BEC(\I\CC_{N_\infty})
\ar@<0.7ex>@{->}[l]^-{\bfE f^{-1}}.
}\]
\end{lemma}
\begin{proof}
This follows from Sublemma \ref{sublem2.2} and 
the fact that the functor $\bfE\che{f}_{!!}$ (resp. $\bfE\che{f}^{-1}$)
is an isomorphism over $M$ (resp.\ $N$)
by the assumption $M\simto N$.
\end{proof}

At the end of this subsection, we shall apply the above results to our situation.
Let $X$ be a smooth algebraic variety over $\CC$ and denote by $X^\an$
the underlying complex manifold of $X$.
Then we can obtain a smooth complete algebraic variety $\tl{X}$ such that
$X\subset \tl{X}$ and $D := \tl{X}\setminus X$ is a normal crossing divisor of $\tl{X}$
by Hironaka's desingularization theorem \cite{Hiro} (see also \cite[Theorem 4.3]{Naga}).
Hence we obtain a bordered space $(X^\an, \tl{X}^\an)$ and
an equivalence of triangulated categories:
\[\xymatrix@C=55pt{
\BEC(\I\CC_{(X^\an, \tl{X}^\an)})\ar@<0.7ex>@{->}[r]^-{\bfE j_{!!}}\ar@{}[r]|-{\sim}
&
\{K\in\BEC(\I\CC_{\tl{X}^\an})\ |\ \pi^{-1}\CC_{X^\an}\otimes K\simto K\}
\ar@<0.7ex>@{->}[l]^-{\bfE j^{-1}}.
}\]

Let $\tl{X}_i\ (i = 1, 2)$ be smooth complete algebraic varieties over $\CC$
such that $X\subset \tl{X}_i \ (i = 1, 2)$,
then the identity map $\id_{X^\an}$ of $X^\an$ induces an isomorphism of bordered spaces
\[(X^\an, \tl{X}_1^\an)\simeq(X^\an, \tl{X}_2^\an)\]
by Lemma \ref{lem2.2}.
Hence we have an equivalence of triangulated categories
\[\BEC(\I\CC_{(X^\an, \tl{X}_1^\an)})\simeq\BEC(\I\CC_{(X^\an, \tl{X}_2^\an)}).\]

On the other hand, this equivalence can be proved by Lemma \ref{lem2.3} as follows.
\begin{proof}
We regard $X$ as a subset of $\tl{X}_1\times\tl{X}_2$
by the diagonal embedding $X\hookrightarrow \tl{X}_i\times\tl{X}_2$.
By Hironaka's desingularization theorem,
we obtain a blow-up of the closure $\var{X}$ in $\tl{X}_1\times \tl{X}_2$.
Namely there exist a smooth complete algebraic variety $\tl{X}$ over $\CC$ and
a projective map $\rho \colon \tl{X} \to \tl{X}_1\times \tl{X}_2$ such that
the restriction $\rho|_{\rho^{-1}(X)}$
induces an isomorphism $\rho^{-1}(X)\simto X$.
Now we can regard $X$ as an open subset of $\tl{X}$
 and hence consider a bordered space $(X^\an, \tl{X}^\an)$. 
Then we have two morphisms of bordered spaces
$$f_i \colon \big(X^\an, \tl{X}^\an\big)\to \big(X^\an, \tl{X}_i^\an\big)\ (i=1, 2 )$$
associated with smooth maps $f_i := {\rm pr}_i^\an\circ\rho^\an \colon \tl{X}^\an\to\tl{X}_i^\an$
such that the restrictions $(f_i)|_{X^\an}$
are equal to the identity map of $X^\an$,
where ${\rm pr}_i^\an \colon \tl{X}_1^\an\times\tl{X}_2^\an\to \tl{X}_i^\an$
is the $i$-th projection ($i=1, 2$) and
$\rho^\an \colon \tl{X}^\an\to \tl{X}_1^\an\times\tl{X}_2^\an$
is the corresponding morphism of $\rho$. 
Therefore by Lemma \ref{lem2.3} there exist equivalences of categories
\[\bfE (f_i)_{!!} \colon \BEC(\I\CC_{(X^\an, \tl{X}^\an)})\simto\BEC(\I\CC_{(X^\an, \tl{X}_i^\an)})\ (i=1,2 )\]
and hence the proof is completed.
\end{proof}
Therefore the bordered space $(X^\an, \tl{X}^\an)$ and 
the triangulated category $\BEC(\I\CC_{(X^\an, \tl{X}^\an)})$
is independent of the choice of a smooth complete variety $\tl{X}$.
 Hence we can write
$$X_\infty^\an := (X^\an, \tl{X}^\an)$$
and
$$\BEC(\I\CC_{X^\an_\infty}) := \BEC(\I\CC_{(X^\an, \tl{X}^\an)})$$
for a smooth algebraic variety $X$.

For a morphism $f \colon X\to Y$ of smooth algebraic varieties,
we obtain a semi-proper morphism of bordered spaces
from $X^\an_\infty$ to $Y^\an_\infty$ by Sublemma \ref{sublem2.1}.
We denote it by $f^\an_\infty \colon X^\an_\infty\to Y^\an_\infty$.
On the other hand,
since there exists a morphism of complete algebraic varieties $\tl{f} \colon \tl{X}\to\tl{Y}$
such that $\tl{f}|_X = f$,
we obtain a morphism of bordered spaces from $X^\an_\infty$ to $Y^\an_\infty$,
see \S \ref{subsec2.2} or \cite[\S 3.2]{DK16} for the details.
It is clear that this morphism is equal to $f_\infty$
and hence we can apply Sublemma \ref{sublem2.2} to $f_\infty$.
This fact will be used in the proof of Proposition \ref{prop-comm}.

\subsection{$\RR$-Constructible Enhanced Ind-Sheaves} 
We shall recall a notion of the $\RR$-constructability for enhanced ind-sheaves and results on it.
References are made to D'Agnolo-Kashiwara \cite{DK16, DK16-2}.

In this subsection,
we assume that a bordered space $M_\infty = (M, \che{M})$ is a subanalytic bordered space.
Namely, $\che{M}$ is a subanalytic space and $M$ is an open subanalytic subset of $\che{M}$.
See \cite[Definition 3.1.1]{DK16-2} for the details.

\begin{definition}[{\cite[Definition 3.1.2]{DK16-2}}]
We denote by $\BDC_{\RR-c}(\CC_{M_\infty})$
the full subcategory of $\BDC(\CC_{M_\infty})\ (=\BDC(\CC_{M}))$
consisting of objects $\SF$ satisfying
$\rmR j_{M_\infty!}\SF$ is an $\RR$-constructible sheaf on $\che{M}$.
We regard $\BDC_{\RR-c}(\CC_{M_\infty})$
as a full subcategory of $\BDC(\I\CC_{M_\infty})$.
\end{definition}

\begin{definition}[{\cite[Definition 3.3.1]{DK16-2}}]\label{def2.6}
We say that $K\in\BEC(\I\CC_{M_\infty})$ is $\RR$-constructible
if for any relatively compact subanalytic open subset $U\subset M$
there exists an isomorphism $\bfE i_{U_\infty}^{-1}K\simeq \CC_{U_\infty}^{\rmE}\Potimes \SF$
for some $\SF\in\BDC_{\RR-c}(\CC_{U_\infty\times\RR_\infty})$.
We denote by $\BEC_{\RR-c}(\I\CC_{M_\infty})$
the full triangulated subcategory of $\BEC(\I\CC_{M_\infty})$
consisting of $\RR$-constructible enhanced ind-sheaves.
\end{definition}

\begin{lemma}[{\cite[Lemma 3.3.2]{DK16-2}}]\label{lem2.7}
Let $K$ be an object of $\BEC(\I\CC_{M_\infty})$.
Then $K$ is $\RR$-constructible if and only if
$\bfE j_{M_\infty!!}K\in\BEC(\I\CC_{\che{M}})$ is $\RR$-constructible.
\end{lemma}

Note that the triangulated category $\BEC_{\RR-c}(\I\CC_{M_\infty})$
has the following standard t-structure
which is induced by the standard t-structure on $\BEC(\I\CC_{M_\infty})$:
\begin{align*}
\bfE^{\leq 0}_{\RR-c}(\I\CC_{M_\infty})
& := \bfE^{\leq 0}(\I\CC_{M_\infty})\cap \BEC_{\RR-c}(\I\CC_{M_\infty}),\\
\bfE^{\geq 0}_{\RR-c}(\I\CC_{M_\infty})
& := \bfE^{\geq 0}(\I\CC_{M_\infty})\cap \BEC_{\RR-c}(\I\CC_{M_\infty}).
\end{align*}
We set $\ZEC_{\RR-c}(\I\CC_{M_\infty}) :=
\bfE^{\leq 0}_{\RR-c}(\I\CC_{M_\infty})\cap\bfE^{\geq 0}_{\RR-c}(\I\CC_{M_\infty})$.

Recall that the Verdier duality functor for enhanced ind-sheaves on $M_\infty$ is defined by
\[\rmD_{M_\infty}^\rmE \colon \BEC(\I\CC_{M_\infty})^{\op}\to\BEC(\I\CC_{M_\infty}),
\hspace{5pt} K\mapsto \Prihom(K, \omega_{M_\infty}^{\rmE}),\]
see \S \ref{subsec2.4} for the details.
\begin{proposition}[{\cite[Proposition 3.3.3 (ii), (iii), (iv)]{DK16-2}}]\label{prop2.8}
Let $f \colon M_\infty \to N_\infty$ be a morphism of bordered spaces
and $K\in \BEC_{\RR-c}(\I\CC_{M_\infty})$, $L\in \BEC_{\RR-c}(\I\CC_{N_\infty})$.
Then we have$:$
\begin{itemize}
\item[\rm (1)]
$\rmD_{M_\infty}^\rmE(K)\in\BEC_{\RR-c}(\I\CC_{M_\infty})$
and $K\simto\rmD_{M_\infty}^\rmE\rmD_{M_\infty}^\rmE K$,

\item[\rm(2)]
$\bfE f^{-1}L,\ \bfE f^{!}L\in\BEC_{\RR-c}(\I\CC_{M_\infty})$
and $\bfE f^{!}L\simeq \rmD_{M_\infty}^\rmE\bfE f^{-1}\rmD_{N_\infty}^\rmE(L)$,

\item[\rm(3)]
if $f \colon M_\infty \to N_\infty$ is semi-proper,
$\bfE f_\ast K,\ \bfE f_{!!}K\in\BEC_{\RR-c}(\I\CC_{N_\infty})$
and $\bfE f_{\ast}K\simeq \rmD_{N_\infty}^\rmE\bfE f_{!!}\rmD_{M_\infty}^\rmE(K)$.
\end{itemize}
\end{proposition}

\subsection{$\D$-Modules}
In this subsection we recall some basic notions and results on $\SD$-modules. 
References are made to 
\cite{Bjo93},
\cite[\S\S 8, 9]{DK16},
\cite[\S 7]{KS01},
\cite[\S\S 3, 4, 7]{KS16} for analytic $\D$-modules,
to \cite{Be}, \cite{Bor}, \cite{HTT08} for algebraic ones.  

\subsubsection{Analytic $\D$-Modules}
In this paper,
we use bold letters for complex manifolds to avoid confusion with algebraic varieties.

For a complex manifold $\bfX$ we denote by $d_{\bfX}$ its complex dimension. 
Denote by $\SO_{\bfX}$ and $\D_{\bfX}$ the sheaves of holomorphic functions 
and holomorphic differential operators on $\bfX$, respectively. 
Let $\BDC(\D_{\bfX})$ be the bounded derived category of left $\D_{\bfX}$-modules. 
Moreover we denote by $\BDCcoh(\D_{\bfX})$,
$\BDChol(\D_{\bfX})$ and $\BDCrh(\D_{\bfX})$ the full triangulated subcategories
of $\BDC(\D_{\bfX})$ consisting of objects with analytic coherent,
analytic holonomic and analytic regular holonomic cohomologies, respectively.
For a morphism $f \colon \bfX\to \bfY$ of complex manifolds, 
denote by $\Dotimes$, $\Dboxtimes$, $\bfD f_\ast$, $\bfD f^\ast$, 
$\DD_{\bfX} \colon \BDCcoh(\SD_{\bfX})^{\op} \simto \BDCcoh(\SD_{\bfX})$  
the standard operations for analytic $\SD$-modules. 

For an analytic hypersurface $D$ in $\bfX$ we denote by $\SO_{\bfX}(\ast D)$ 
the sheaf of meromorphic functions on $\bfX$ with poles in $D$. 
Then for $\M\in\BDC(\D_{\bfX})$ we set 
$\M(\ast D) := \M\Dotimes\SO_{\bfX}(\ast D)$.
We say that a $\D_{\bfX}$-module is an analytic meromorphic connection on $\bfX$ along $D$
if it is isomorphic as an $\SO_{\bfX}$-module to a coherent $\SO_{\bfX}(\ast D)$-module.
We denote by $\Conn({\bfX}; D)$ the category of meromorphic connections along $D$.
Note that it is a full abelian subcategory of $\Modhol(\D_{\bfX})$.
Moreover, we set
\[\BDCmero(\D_{\bfX(D)}) :=\{\M\in\BDChol(\D_{\bfX})\
|\ \SH^i(\M)\in\Conn(\bfX; D) \mbox{ for any }i\in\ZZ \}.\]

The classical solution functor on $\bfX$ is defined by  
\begin{align*}
\Sol_{\bfX} &\colon \BDCcoh (\D_{\bfX})^{\op}\to\BDC(\CC_{\bfX}),
\hspace{10pt}\M \longmapsto \rhom_{\D_{\bfX}}(\M, \SO_{\bfX}).
\end{align*}
We denote by $\SO_{\bfX}^{\rmE}$ the enhanced ind-sheaf of tempered holomorphic functions
\cite[Definition 8.2.1]{DK16}
and by $\Sol_{\bfX}^{\rmE}$ the enhanced solution functor on $\bfX$:
\[
\Sol_{\bfX}^\rmE \colon \BDCcoh (\D_{\bfX})^{\op}\to\BEC(\I\CC_{\bfX}), 
\hspace{10pt} 
\M \longmapsto \rihom_{\D_{\bfX}}(\M, \SO_{\bfX}^\rmE) ,
\]
\cite[Definition 9.1.1]{DK16}.
Note that for $\M\in\BDCcoh(\D_{\bfX})$, 
we have an isomorphism
\[\sh_{\bfX}\big( \Sol_{\bfX}^{\rmE}(\M)\big)\simeq \Sol_{\bfX}(\M)\]
by \cite[Lemma 9.5.5]{DK16}.

Let us recall the results of \cite{DK16}.
We note that (3) of Theorem \ref{thm-DK} below was proved in \cite{DK16}
under the assumption that $\M$ has a globally good filtration.
However, any holonomic $\D$-module on $\bfX$ has a globally defined good filtration
by \cite{Mal94, Mal94-2, Mal96} (see also \cite[Theorem 4.3.4]{Sab11}).
\begin{theorem}[{\cite[\S 9.4]{DK16}}]\label{thm-DK}
\begin{enumerate}
\item[\rm{(1)}]
 For $\M\in\BDChol(\D_{\bfX})$ there exists an isomorphism in $\BEC(\I\CC_{\bfX})$
\[\Sol_{\bfX}^\rmE(\DD_{\bfX}\M)[2d_{\bfX}]
\simeq\rmD_{\bfX}^\rmE\Sol_{\bfX}^\rmE(\M).\]

\item[\rm{(2)}] 
Let $f \colon \bfX\to \bfY$ be a morphism of complex manifolds.
Then for $\N\in\BDChol(\D_{\bfY})$ there exists an isomorphism in $\BEC(\I\CC_{\bfX})$
\[\Sol_{\bfX}^\rmE({\bfD} f^\ast\N)\simeq\bfE f^{-1}\Sol_{\bfY}^\rmE(\N).\]

\item[\rm{(3)}] 
Let $f \colon \bfX\to \bfY$ be a proper morphism of complex manifolds.
For $\M\in\BDChol(\D_{\bfX})$ there exists an isomorphism in $\BEC(\I\CC_{\bfY})$
\[\Sol_{\bfY}^\rmE({\bfD} f_\ast\M)[d_{\bfY}]
\simeq\bfE f_\ast \Sol_{\bfX}^\rmE(\M )[d_{\bfX}].\]

\item[\rm{(4)}]
For $\M_1, \M_2\in\BDChol(\D_{\bfX})$,
there exists an isomorphism in $\BEC(\I\CC_{\bfX})$
\[\Sol_{\bfX}^\rmE(\M_1\Dotimes\M_2)\simeq
\Sol_{\bfX}^\rmE(\M_1)\Potimes \Sol_{\bfX}^\rmE(\M_2).\]

\item[\rm{(5)}]
Let $\M\in\BDChol(\D_{\bfX})$ and $D\subset \bfX$ be an analytic hypersurface,
then there exists an isomorphism in $\BEC(\I\CC_{\bfX})$
\[
\Sol_{\bfX}^\rmE\big(\M(\ast D)\big) \simeq
\pi^{-1}\CC_{\bfX\bs D}\otimes \Sol_{\bfX}^\rmE(\M). 
\]
\end{enumerate}
\end{theorem}

We also recall the following theorems \cite[Theorems 9.6.1 and 9.1.3]{DK16}.
\begin{theorem}\label{thm2.6}
\begin{itemize}
\item[\rm(1)]
The enhanced solution functor induces an embedding
\[ \Sol_{\bfX}^\rmE \colon \BDChol(\D_{\bfX})^{\rm op}
\hookrightarrow\BEC_{\RR-c}(\I\CC_{\bfX}).\]
Moreover for any $\M\in\BDChol(\D_{\bfX})$ there exists an isomorphism
\[\M\simto \RH_{\bfX}^{\rmE}\big(\Sol_{\bfX}^{\rmE}(\M)\big),\]  
where $\RH_{\bfX}^{\rmE}(K) := \rhom^{\rmE}(K, \SO_{\bfX}^{\rmE})$.

\item[\rm(2)]
For any $\M\in\BDCrh(\D_{\bfX})$ there exists an isomorphism
$$\Sol_{\bfX}^{\rmE}(\M)\simeq e_{\bfX}\big(\Sol_{\bfX}(\M)\big)$$
and hence we have a commutative diagram
\[\xymatrix@C=30pt@M=5pt{
\BDChol(\D_{\bfX})^{\op}\ar@{^{(}->}[r]^-{\Sol_{\bfX}^{\rmE}}
\ar@{}[rd]|{\rotatebox[origin=c]{180}{$\circlearrowright$}}
 & \BEC_{\RR-c}(\I\CC_{\bfX})\\
\BDCrh(\D_{\bfX})^{\op}\ar@{->}[r]_-{\Sol_{\bfX}}^-{\sim}\ar@{}[u]|-{\bigcup}
&\BDC_{\CC-c}(\CC_{\bfX}).\ar@{^{(}->}[u]_-{e_{\bfX}}
}\]
\end{itemize}
\end{theorem}
See \S \ref{sec2.8} for the essential image of the enhanced solution functor $\Sol_{\bfX}^\rmE$.

At the end of this subsection, let us recall the notion of $\M_{\reg}$.
We denote by $\D_{\bfX}^\infty$
the sheaf of rings of differential operators of infinite order on $\bfX$
and set $\M^\infty := \D_{\bfX}^\infty\otimes _{\D_{\bfX}}\M$.

\begin{proposition}[{\cite[Theorem 5.5.22]{Bjo93}},
{\cite[Theorem 5.2.1]{KK}} and {\cite[Proposition 5.7]{Kas84}}]\label{prop2.7}~\\
\vspace{-20pt}
\begin{itemize}
\item[\rm (1)]
Let $\M$ be an analytic holonomic $\D_{\bfX}$-module.
Then there exists a unique analytic regular holonomic $\D_{\bfX}$-module $\M_{\reg}$
such that
\begin{itemize}
\item[\rm (i)]
$\M_{\reg}^\infty\simeq \M^\infty$,

\item[\rm (ii)]
$\M_\reg$ contains every analytic regular holonomic $\D_{\bfX}$-submodule of $\M^\infty$,

\item[\rm (iii)]
$\Sol_{\bfX}(\M_\reg)\simeq \Sol_{\bfX}(\M)$.
\end{itemize}

\item[\rm (2)]
There exists an isomorphism
\[\M_{\reg}\simeq
\{s\in\M^\infty\ |\ \D_{\bfX}\cdot s\in\Modrh(\D_{\bfX})\}.\]
\end{itemize}
\end{proposition}
By this proposition, we obtain a functor
\[(\cdot)_{\reg} \colon \Modhol(\D_{\bfX})\to\Modrh(\D_{\bfX}),
\hspace{5pt} \M\mapsto\M_\reg.\]
We call it the regularization functor for analytic holonomic $\D$-modules.
Note that this is an exact functor.
Hence, it induces the functor between derived categories
$$(\cdot)_{\reg} \colon \BDChol(\D_{\bfX})\to\BDCrh(\D_{\bfX}).$$

\subsubsection{Algebraic $\D$-Modules}\label{sec-alg}
Let $X$ be a smooth algebraic variety over $\CC$ and denote by $d_X$ its complex dimension.
We shall denote by $\SO_X$ and $\SD_X$ the sheaves of regular functions
and algebraic differential operators on $X$, respectively. 
Let $\BDC(\D_X)$ be the bounded derived category of left $\D_X$-modules. 
Moreover we denote by $\BDCcoh(\D_X)$,
$\BDChol(\D_X)$ and $\BDCrh(\D_X)$ the full triangulated subcategories
of $\BDC(\D_X)$ consisting of objects with algebraic coherent,
algebraic holonomic and algebraic regular holonomic cohomologies, respectively.
For a morphism $f \colon X\to Y$ of smooth algebraic varieties, 
we denote by $\Dotimes$, $\Dboxtimes$, $\bfD f_\ast$, $\bfD f^\ast$ and $\DD_X$
the tensor product functor, the external tensor product functor,
the direct image functor, the inverse image functor
and the duality functor for $\D$-modules, respectively.
See e.g., \cite[\S 3]{HTT08} for the details.
In this paper, for convenience, we set
\begin{align*}
\bfD f_! := \DD_Y\circ\bfD f_\ast\circ\DD_X,\hspace{10pt}
\bfD f^\bigstar := \DD_X\circ\bfD f^\ast\circ\DD_Y.
\end{align*}
Remark that in \cite[Definition 3.2.13]{HTT08})
the functor $\DD_X\circ\bfD f^\ast(\cdot)[d_X-d_Y]\circ\DD_Y$ is denoted by $\bfD f^\bigstar$.
Note that these functors preserve the holonomicity.
Namely we have functors:
\begin{align*}
(\cdot) \Dotimes (\cdot) &\colon \BDChol(\D_X)\times \BDChol(\D_X)\to\BDChol(\D_X),\\
(\cdot) \Dboxtimes (\cdot) &\colon \BDChol(\D_X)\times \BDChol(\D_X)\to\BDChol(\D_X),\\
\DD_X &\colon \BDChol(\D_X)^{\op} \simto \BDChol(\D_X),\\
\bfD f^\ast &\colon \BDChol(\D_Y)\to\BDChol(\D_X),\\
\bfD f_\ast &\colon \BDChol(\D_X)\to\BDChol(\D_Y),\\
\bfD f^\bigstar &\colon \BDChol(\D_Y)\to\BDChol(\D_X),\\
\bfD f_! &\colon \BDChol(\D_X)\to\BDChol(\D_Y).
\end{align*}
See \cite[Propositions 3.2.1, 3.2.2, Theorem 3.2.3 and Corollary 3.2.4]{HTT08} for the details.

We denote by $X^\an$ the underlying complex manifold of $X$
and by $\tl{\iota} \colon (X^\an, \SO_{X^\an})\to(X, \SO_X)$ the morphism of ringed spaces.
Since there exists a morphism $\tl{\iota}^{-1}\SO_X\to\SO_{X^\an}$ of sheaves on $X^\an$,
we have a canonical morphism $\tl{\iota}^{-1}\D_X\to\D_{X^\an}$.
Then we set $$\M^\an := \D_{X^\an}\otimes_{\tl{\iota}^{-1}\D_X}\tl{\iota}^{-1}\M$$
for an algebraic $\D_X$-module $\M\in\Mod(\D_X)$
and obtain a functor
\[(\cdot)^\an \colon \Mod(\D_X)\to\Mod(\D_{X^\an}).\]
It is called the analytification functor on $X$.
Since the sheaf $\D_{X^\an}$ is faithfully flat over $\tl{\iota}^{-1}\D_X$,
the analytification functor is faithful and exact,
and hence we obtain
\[(\cdot)^\an \colon \BDC(\D_X)\to\BDC(\D_{X^\an}).\]
The analytification functor induces
\begin{align*}
(\cdot)^\an \colon \Modcoh(\D_X)\to\Modcoh(\D_{X^\an}),\hspace{25pt}
&(\cdot)^\an \colon \BDCcoh(\D_X)\to\BDCcoh(\D_{X^\an}),\\
(\cdot)^\an \colon \Modhol(\D_X)\to\Modhol(\D_{X^\an}),\hspace{25pt}
&(\cdot)^\an \colon \BDChol(\D_X)\to\BDChol(\D_{X^\an}).
\end{align*}
Moreover we have some functorial properties of the analytification functor.  
\begin{proposition}[{\cite[Propositions 4.7.1 and 4.7.2]{HTT08}}]\label{prop2.12}
\
\begin{itemize}
\item[\rm(1)]
For any $\M\in\BDCcoh(\D_X)$ 
we have $(\DD_X\M)^\an\simeq \DD_{X^\an}(\M^\an)$.

\item[\rm (2)]
Let $f \colon X\to Y$ be a morphism of smooth algebraic varieties and $\N\in\BDC(\D_Y)$.
Then we have $(\bfD f^{\ast}\N)^\an\simeq\bfD(f^\an)^{\ast}(\N^\an)$.

\item[\rm (3)]
Let $f \colon X\to Y$ be a morphism of smooth algebraic varieties and $\M\in\BDC(\D_X)$.
Then we have a canonical morphism
$(\bfD f_{\ast}\M)^\an\to\bfD(f^\an)_{\ast}(\M^\an)$.
If the morphism $f$ is proper and $\M\in\BDCcoh(\D_X)$ then we have
$(\bfD f_{\ast}\M)^\an\simto\bfD(f^\an)_{\ast}(\M^\an)$
\end{itemize}
\end{proposition}

The classical solution functor on $X$ is defined by  
\begin{align*}
\Sol_X &\colon \BDCcoh (\SD_X)^{\op}\to\BDC(\CC_{X^\an}),
\hspace{10pt} \Sol_X(\M) \colon= \Sol_{X^\an}(\M^\an).
\end{align*}
We denote by $\BDC_{\CC-c}(\CC_X)$
the triangulated category of algebraic $\CC$-constructible sheaves on $X^\an$.
It has a t-structure
$\big({}^p\bfD^{\leq0}_{\CC-c}(\CC_X), {}^p\bfD^{\geq0}_{\CC-c}(\CC_X)\big)$
which is called the perverse t-structure.
Let us denote by $$\Perv(\CC_X) :=
{}^p\bfD^{\leq0}_{\CC-c}(\CC_X)\cap {}^p\bfD^{\geq0}_{\CC-c}(\CC_X)$$ its heart
and call an object of $\Perv(\CC_X) $ an algebraic perverse sheaf on $X^\an$.

The last part of the following result is
an algebraic version of the regular Riemann--Hilbert correspondence.
See e.g., {\cite[Theorems 4.7.7, 7.2.2]{HTT08} for the details.
\begin{theorem}\label{thm-RH}
For any $\M\in\BDChol(\D_X)$ we have $\Sol_X(\M)\in\BDC_{\CC-c}(\CC_X)$.

Moreover there exists an equivalence of triangulated categories:
\[\Sol_X \colon \BDCrh (\D_X)^{\op}\simto\BDC_{\CC-c}(\CC_{X}).\]
\end{theorem}

The following result means that the classical solution functor is t-exact
with respect to the standard t-structure on $\BDChol(\D_X)$
and the perverse t-structure on $\BDC_{\CC-c}(\CC_X)$.
See e.g., the proof of \cite[Theorem 7.2.5]{HTT08} for the details.
\begin{theorem}\label{thm-perv}
For any $\M\in\BDChol(\D_X)$, we have
\begin{itemize}
\item[\rm(1)]
$\M\in\DChol^{\leq0}(\D_X)\Longleftrightarrow
\Sol_X(\M)[d_X]\in{}^p\bfD^{\geq0}_{\CC-c}(\CC_X)$,
\item[\rm(2)]
$\M\in\DChol^{\geq0}(\D_X)\Longleftrightarrow
\Sol_X(\M)[d_X]\in{}^p\bfD^{\leq0}_{\CC-c}(\CC_X)$.
\end{itemize}
Hence for any $\M\in\Modhol(\D_X)$,
we have $\Sol_X(\M)[d_X]\in\Perv(\CC_X)$.

Moreover the above equivalence induces an equivalence of abelian categories:
\[\Sol_X(\cdot)[d_X] \colon \Modrh (\D_X)^{\op}\simto\Perv(\CC_{X}).\]
\end{theorem}

By Theorems \ref{thm-RH} and \ref{thm-perv},
for any $\M\in\BDChol(\D_X)$
there exists a unique algebraic regular holonomic $\D_X$-modules $\M_\reg$
such that $\Sol_X(\M_\reg)\simeq\Sol_X(\M)$.
In fact, $\M_\reg$ is given by $\RH_X\big(\Sol_X(\M))$,
where $\RH_X \colon \BDC_{\CC-c}(\CC_X)^{\op}\simto \BDCrh(\D_X)$
is the inverse functor of $\Sol_X \colon \BDCrh (\D_X)^{\op}\simto\BDC_{\CC-c}(\CC_{X})$.
Hence we have a functor
\[(\cdot)_{\reg} \colon \BDChol(\D_{X})\to\BDCrh(\D_{X})\]
and a commutative diagram
\[\xymatrix@C=50pt@M=5pt{
\BDChol(\D_X)^{\op}\ar@{->}[r]^-{\Sol_X}
\ar@{->}[d]_-{(\cdot)_\reg}
 & \BDC_{\CC-c}(\CC_X)\ar@{->}[ld]^-{\RH_X}_-{\sim}\\
\BDCrh(\D_X)^{\op}.&
}\]
We call it the regularization functor for algebraic holonomic $\D$-modules.
By Theorem \ref{thm-perv},
this is a t-exact functor with respect to the standard t-stuructures.
Namely, we obtain
the functor \[(\cdot)_\reg \colon \Modhol(\D_X)\to\Modrh(\D_X).\]
between abelian categories.

At the end of this subsection, we shall recall algebraic meromorphic connections.
Let $D$ be a divisor of $X$,
and $j \colon X\setminus D\hookrightarrow X$ the natural embedding.
Then we set $\SO_X(\ast D) := j_\ast\SO_X$
and also set
$\SM(\ast D) := \SM\Dotimes\SO_X(\ast D)$
for $\SM\in\Mod(\SD_X)$.
Note that we have $\M(\ast D)\simeq \bfD j_\ast\bfD j^{\ast}\M$.
We say that a $\D_X$-module is an algebraic meromorphic connection along $D$
if it is isomorphic as an $\SO_X$-module to a coherent $\SO_X(\ast D)$-module.
We denote by $\Conn(X; D)$ the category of algebraic meromorphic connections along $D$.
Note that it is full abelian subcategory of $\Modhol(\D_X)$.
Moreover, we set
\[\BDCmero(\D_{X(D)}) :=\{\M\in\BDChol(\D_X)\
|\ \SH^i(\M)\in\Conn(X; D) \mbox{ for any }i\in\ZZ \}.\]

We say that
a Zariski locally finite partition $\{X_\alpha\}_{\alpha\in A}$ of $X$
by locally closed subvarieties $X_\alpha$ is an algebraic stratification of $X$
if for any $\alpha\in A$ $X_\alpha$ is smooth and there exists a subset $B\subset A$
such that $\var{X}_\alpha = \sqcup_{\beta\in B} X_{\beta}$.
The following result is known.
See e.g., \cite[Theorem 3.1]{HTT08} for the details, \cite[Lemma 3.24]{Ito19} for the analytic case.
\begin{lemma}
For any $\M\in\Modhol(\D_X)$
there exists an algebraic stratification $\{X_\alpha\}_{\alpha\in A}$
such that any cohomology of $\bfD i_{X_\alpha}^\ast(\M)$
is an integrable connection on $X_\alpha$ for each $\alpha\in A$.
\end{lemma}
Hence by this lemma we have:
\begin{lemma}\label{lem-str}
Let $\M$ be a holonomic $\D_X$-module.
Then there exists an algebraic stratification $\{X_\alpha\}_{\alpha\in A}$ of $X$
such that for any $\alpha\in A$ and
any complex blow-up $b_\alpha \colon \var{X}^\blow_\alpha \to X$ of $\var{X_\alpha}$
along $\partial X_\alpha := \var{X_\alpha}\setminus X_\alpha$ we have
$(\bfD b_\alpha^\ast \M)(\ast D_\alpha)
\in\BDCmero(\D_{\var{X}^\blow_\alpha(D_\alpha)})$,
where $D_\alpha := b_\alpha^{-1}(\partial X_\alpha)$
is a normal crossing divisor of $\var{X}^\blow_\alpha$.
\end{lemma}
This lemma will be used in the proof of Proposition \ref{prop3.4}.

The analytification functor $(\cdot)^\an : \Mod(\D_X)\to\Mod(\D_{X^\an})$ induces
\begin{align*}
(\cdot)^\an &\colon \Conn(X; D)\to\Conn(X^\an; D^\an),\\
(\cdot)^\an &\colon \BDCmero(\D_{X(D)})\to \BDCmero(\D_{X^\an(D^\an)})
\end{align*}
where we set $D^\an := \bfX^\an\setminus (\bfX\setminus D)^\an$.
We note that if $X$ is complete
there exists an equivalence of categories between
the abelian category $\Conn(X; D)$ and
the one of effective meromorphic connections on $X^\an$ along $D^\an$
by \cite[\S 5.3]{HTT08}.
However as a consequence of \cite[Theorem 4.2]{Mal04}
any analytic meromorphic connection is effective.
Hence we have:
\begin{lemma}[{\cite[(5.3.2)]{HTT08}}, {\cite{Mal04}}]\label{lem2.17}
If $X$ is complete, 
there exists an equivalence of abelian categories:
\[(\cdot)^\an \colon \Conn(X; D)\simto\Conn(X^\an; D^\an).\]
Moreover this induces an equivalence of triangulated categories:
\[(\cdot)^\an \colon \BDCmero(\D_{X(D)})\simto\BDCmero(\D_{X^\an(D^\an)}).\]
\end{lemma}

\subsection{Analytic $\CC$-Constructible Enhanced Ind-Sheaves}\label{sec2.8}
In this subsection, we shall recall some notions and results in \cite{Ito19}.

Let $\bfX$ be a complex manifold and 
$D \subset \bfX$ a normal crossing divisor in it. 
Let us take local coordinates 
$(u_1, \ldots, u_l, v_1, \ldots, v_{d_X-l})$ 
of $\bfX$ such that $D= \{ u_1 u_2 \cdots u_l=0 \}$
and set $$Y= \{ u_1=u_2= \cdots =u_l=0 \}.$$ 
We define a partial order $\leq$ on the 
set $\ZZ^l$ by 
\[ a \leq a^{\prime} \ \Longleftrightarrow 
\ a_i \leq a_i^{\prime} \ (1 \leq i \leq l).\] 
Then for a meromorphic function $\varphi\in\SO_{\bfX}(\ast D)$
on $\bfX$ along $D$ which has the Laurent expansion
\[ \varphi = \sum_{a \in \ZZ^l} c_a( \varphi )(v) \cdot 
u^a \ \in \SO_{\bfX}(\ast D) \]
with respect to $u_1, \ldots, u_{l}$,
where $c_a( \varphi )$ are holomorphic functions on $Y$, 
we define its order 
$\ord( \varphi ) \in \ZZ^l$ by the minimum 
\[ \min \Big( \{ a \in \ZZ^l \ | \
c_a( \varphi ) \not= 0 \} \cup \{ 0 \} \Big) \]
if it exists. 
For any $f\in\SO_{\bfX}(\ast D)/ \SO_{\bfX}$, we take any lift $\tl{f}$ to $\SO_{\bfX}(\ast D)$,
and we set $\ord(f) := \ord(\tl{f})$, if the right-hand side exists.
Note that it is independent of the choice of a lift $\tl{f}$.
If $\ord(f)\neq0$, $c_{\ord(f)}(\tl{f})$ is independent of the choice of a lift $\tl{f}$,
which is denoted by $c_{\ord(f)}(f)$.
\begin{definition}[{\cite[Definition 2.1.2]{Mochi11}}]\label{def2.10}
In the situation as above,
let us set $$Y= \{ u_1=u_2= \cdots =u_l=0 \}.$$
A finite subset $\calI \subset \SO_{\bfX}(\ast D)/ \SO_{\bfX}$
is called a good set of irregular values on $(\bfX,D)$,
if the following conditions are satisfied:
\begin{itemize}
\setlength{\itemsep}{-3pt}
\item[-]
For each element $f\in\calI$, $\ord(f)$ exists.
If $f\neq0$ in $\SO_{\bfX}(\ast D)/ \SO_{\bfX}$, $c_{\ord(f)}(f)$ is invertible on $Y$.
\item[-]
For two distinct $f, g\in\calI$, $\ord(f-g)$ exists and
$c_{\ord(f-g)}(f-g)$ is invertible on $Y$.
\item[-]
The set $\{\ord(f-g)\ |\ f, g\in\calI\}$ is totally ordered
with respect to the above partial order $\leq$ on $\ZZ^l$.
\end{itemize}
\end{definition}

\begin{definition}[{\cite[Definition 3.6]{Ito19}}]\label{def-normal}
We say that an $\RR$-constructible enhanced ind-sheaf
$K\in\ZEC_{\RR-c}(\I\CC_{\bfX})$ has a normal form along $D$ if 
\begin{itemize}
\setlength{\itemsep}{-3pt}
\item[(i)]
$\pi^{-1}\CC_{\bfX\setminus D}\otimes K\simto K$,

\item[(ii)]
for any $x\in \bfX\setminus D$ there exist an open neighborhood $U_x\subset \bfX\setminus D$
of $x$ and a non-negative integer $k$ such that
\[K|_{U_x}\simeq (\CC_{U_x}^{\rmE})^{\oplus k},\]

\item[(iii)]
for any $x\in D$ there exist an open neighborhood $U_x\subset \bfX$ of $x$,
a good set of irregular values $\{\varphi_i\}_i$ on $(U_x, D\cap U_x)$
and a finite sectorial open covering $\{U_{x, j}\}_j$ of $U_x\bs D$
such that
\[\pi^{-1}\CC_{U_{x, j}}\otimes K|_{U_x}\simeq
\bigoplus_i \EE_{U_{x, j} | U_x}^{\Re\varphi_i}
\hspace{10pt} \mbox{for any } j,\]
see the end of \S\ref{subsec2.4} for the definition of 
$\EE_{U_{x, j} | U_x}^{\Re\varphi_i}$.
\end{itemize}
\end{definition}

A ramification of $\bfX$ along a normal crossing divisor $D$ on a neighborhood $U$ 
of $x \in D$ is a finite map $r \colon U^{\rami}\to U$ of complex manifolds of the form
$z' \mapsto z=(z_1,z_2, \ldots, z_n)= 
 r(z') = (z'^{m_1}_1,\ldots, z'^{m_r}_r, z'_{r+1},\ldots,z'_n)$ 
for some $(m_1, \ldots, m_r)\in (\ZZ_{>0})^r$, where 
$(z'_1,\ldots, z'_n)$ is a local coordinate system of $U^{\rami}$ and 
$(z_1, \ldots, z_n)$ is the one of 
$U$ such that $D \cap U=\{z_1\cdots z_r=0\}$. 

\begin{definition}[{\cite[Definition 3.11]{Ito19}}]\label{def-quasi}
We say that an enhanced ind-sheaf
$K\in\ZEC(\I\CC_{\bfX})$ has a quasi-normal form along $D$
if it satisfies (i) and (ii) in Definition \ref{def-normal}, and if
for any $x\in D$ there exist an open neighborhood $U_x\subset \bfX$ of $x$
and a ramification $r_x \colon U_x^{\rami}\to U_x$ of $U_x$ along $D_x := U_x\cap D$
such that $\bfE r_x^{-1}(K|_{U_x})$ has a normal form along $D_x^{\rami}:= r_x^{-1}(D_x)$.
\end{definition}
Note that any enhanced ind-sheaf which has a quasi-normal form along $D$
is an $\RR$-constructible enhanced ind-sheaf on $\bfX$.
See \cite[Proposition 3.12]{Ito19} for the details. 

A modification of $\bfX$ with respect to an analytic hypersurface $H$
is a projective map $m \colon \bfX^{\modi}\to \bfX$ from a complex manifold $\bfX^{\modi}$ to $\bfX$ such that
$D^{\modi} := m^{-1}(H)$ is a normal crossing divisor of $\bfX^{\modi}$
and $m$ induces an isomorphism $\bfX^{\modi}\setminus D^{\modi}\simto \bfX\setminus H$.

\begin{definition}[{\cite[Definition 3.14]{Ito19}}]\label{def-modi}
We say that an enhanced ind-sheaf $K\in\ZEC(\I\CC_{\bfX})$
has a modified quasi-normal form along $H$
if it satisfies (i) and (ii) in Definition \ref{def-normal}, and if
for any $x\in H$ there exist an open neighborhood $U_x\subset \bfX$ of $x$
and a modification $m_x \colon U_x^{\modi}\to U_x$ of $U_x$ along $H_x := U_x\cap H$
such that $\bfE m_x^{-1}(K|_{U_x})$ has a quasi-normal form along $D_x^{\modi} := f_x^{-1}(H_x)$.
\end{definition}
Note that any enhanced ind-sheaf which has a modified quasi-normal form along $H$
is an $\RR$-constructible enhanced ind-sheaf on $\bfX$. 
See \cite[Proposition 3.15]{Ito19} for the details. 
Moreover we have:
\begin{lemma}[{\cite[Lemma 3.16]{Ito19}}]\label{lem-modi}
The enhanced solution functor $\Sol_{\bfX}^{\rmE}$ induces
an equivalence of abelian categories between
the full subcategory of $\ZEC_{\RR-c}(\I\CC_{\bfX})$ consisting of
objects which have a modified quasi-normal form along $H$
and the abelian category $\Conn(\bfX; H)$ of
meromorphic connections on $\bfX$ along $H$.
\end{lemma}
We denote by $\ZECmero(\I\CC_{\bfX(H)})$ the essential image of
$$\Sol_{\bfX}^{\rmE} \colon \Conn(\bfX; H)^{\op}\to\ZEC_{\RR-c}(\I\CC_{\bfX}).$$
This abelian category is nothing but the full subcategory of $\ZEC_{\RR-c}(\I\CC_{\bfX})$
consisting of enhanced ind-sheaves which have a modified quasi-normal form along $H$
by Lemma \ref{lem-modi}.
Moreover, we set
\begin{align*}
\BECmero(\I\CC_{\bfX(H)}) &:=\{K\in\BEC_{\RR-c}(\I\CC_{\bfX})\
|\ \SH^i(K)\in\ZECmero(\I\CC_{\bfX(H)}) \mbox{ for any }i\in\ZZ \}.
\end{align*}

Since the category $\BDCmero(\D_{\bfX(H)})$ is a full triangulated subcategory
of $\BDChol(\D_{\bfX})$
and the category $\BECmero(\I\CC_{\bfX(H)})$ is a full triangulated subcategory
of $\BEC_{\RR-c}(\I\CC_{\bfX})$, the following proposition is obvious
by induction on the length of a complex:
\begin{proposition}\label{prop2.23}
The enhanced solution functor $\Sol_{\bfX}^\rmE$ induces an equivalence of triangulated categories
$$\BDCmero(\D_{\bfX(H)})^{\op} \simto \BECmero(\I\CC_{\bfX(H)}),$$
and hence we obtain a commutative diagram
\[\xymatrix@C=30pt@M=3pt{
\BDCmero(\D_{\bfX(H)})^{\op}\ar@{->}[r]^\sim\ar@<1.0ex>@{}[r]^-{\Sol_{\bfX}^{\rmE}}
\ar@{}[rd]|{\rotatebox[origin=c]{180}{$\circlearrowright$}}
 & \BECmero(\I\CC_{\bfX(H)})\\
\Conn(\bfX; H)^{\op}\ar@{->}[r]_-{\Sol_{\bfX}^\rmE}^-{\sim}\ar@{}[u]|-{\bigcup}
&\ZECmero(\I\CC_{\bfX(H)}).\ar@{}[u]|-{\bigcup}
}\]
\end{proposition}

A complex analytic stratification of $\bfX$ is
a locally finite partition $\{\bfX_\alpha\}_{\alpha\in A}$ of $\bfX$
by locally closed analytic subsets $\bfX_\alpha$
such that for any $\alpha\in A$ $\bfX_\alpha$ is smooth,
$\var{\bfX}_\alpha$ and $\partial\bfX_\alpha := \var{\bfX}_\alpha\setminus \bfX_\alpha$
are complex analytic subsets 
and $\var{X}_\alpha = \sqcup_{\beta\in B} X_{\beta}$ for a subset $B\subset A$.
\begin{definition}[{\cite[Definition 3.19]{Ito19}}]\label{def-const}
We say that an enhanced ind-sheaf $K\in\ZEC(\I\CC_{\bfX})$ is $\CC$-constructible if
there exists a complex analytic stratification $\{\bfX_\alpha\}_\alpha$ of $\bfX$
such that
$$\pi^{-1}\CC_{\var{X}^{\blow}_\alpha\setminus D_\alpha}\otimes \bfE b_\alpha^{-1}K$$
has a modified quasi-normal form along $D_\alpha$ for any $\alpha$,
where $b_\alpha \colon \var{\bfX}^{\blow}_\alpha \to X$ is a complex blow-up of $\var{\bfX_\alpha}$
along $\partial \bfX_\alpha = \var{\bfX_\alpha}\setminus \bfX_\alpha$ and
$D_\alpha := b_\alpha^{-1}(\partial \bfX_\alpha)$.
Namely $\var{\bfX}^{\blow} _\alpha$ is a complex manifold,
$D_\alpha$ is a normal crossing divisor of $\var{\bfX}^{\blow} _\alpha$
and $b_\alpha$ is a projective map
which induces an isomorphism $\var{\bfX}^{\blow} _\alpha\setminus D_\alpha\simto \bfX_\alpha$
and satisfies $b_\alpha\big(\var{\bfX}^{\blow} _\alpha\big)=\var{\bfX_\alpha}$.

We call such a family $\{\bfX_\alpha\}_{\alpha\in A}$
a complex analytic stratification adapted to $K$.
\end{definition}

\begin{remark}\label{rem2.25}
Definition \ref{def-const} does not depend on the choice of a complex blow-up $b_\alpha$
by \cite[Sublemma 3.22]{Ito19}.
\end{remark}

We denote by $\ZEC_{\CC-c}(\I\CC_{\bfX})$ the full subcategory of $\ZEC(\I\CC_{\bfX})$
whose objects are $\CC$-constructible
and set
\[\BEC_{\CC-c}(\I\CC_{\bfX}) := \{K\in\BEC(\I\CC_{\bfX})\
|\ \SH^i(K)\in\ZEC_{\CC-c}(\I\CC_{\bfX}) \mbox{ for any }i\in\ZZ \}\subset \BEC(\I\CC_{\bfX}).\]
Note that the category $\ZEC_{\CC-c}(\I\CC_{\bfX})$ is
the full abelian subcategory of $\ZEC_{\RR-c}(\I\CC_{\bfX})$.
Hence the category $\BEC_{\CC-c}(\I\CC_{\bfX})$ is
a full triangulated subcategory of $\BEC_{\RR-c}(\I\CC_{\bfX})$.
See \cite[Proposition 3.21]{Ito19} for the details.

\begin{theorem}[{\cite[Proposition 3.25, Theorem 3.26]{Ito19}}]\label{thm-const}
For any $\M\in\BDChol(\D_{\bfX})$, the enhanced solution complex $\Sol_{\bfX}^\rmE(\M)$
of $\M$ is a $\CC$-constructible enhanced ind-sheaf.
On the other hand,
for any $\CC$-constructible enhanced ind-sheaf $K\in\BEC_{\CC-c}(\I\CC_{\bfX})$,
there exists $\M\in\BDChol(\D_{\bfX})$
such that $$K\simto \Sol_{\bfX}^{\rmE}(\M).$$
Therefore we obtain an equivalence of triangulated categories
\[\xymatrix@C=50pt@R=7pt{
\BDChol(\D_{\bfX})^{\op}\ar@<0.7ex>@{->}[r]^-{\Sol_{\bfX}^{\rmE}}\ar@{}[r]|-{\sim}
&
\BEC_{\CC-c}(\I\CC_{\bfX})\ar@<0.7ex>@{->}[l]^-{\RH_{\bfX}^{\rmE}}}.\]
\end{theorem}
Furthermore we have commutative diagrams:
\[\xymatrix@C=30pt@M=5pt{
\BDChol(\D_{\bfX})^{\op}\ar@{->}[r]^\sim\ar@<1.0ex>@{}[r]^-{\Sol_{\bfX}^{\rmE}}
\ar@{}[rd]|{\rotatebox[origin=c]{180}{$\circlearrowright$}}
 & \BEC_{\CC-c}(\I\CC_{\bfX})\\
\BDCrh(\D_{\bfX})^{\op}\ar@{->}[r]_-{\Sol_{\bfX}}^-{\sim}\ar@{}[u]|-{\bigcup}
&\BDC_{\CC-c}(\CC_{\bfX}),
\ar@{^{(}->}[u]_-{e_{\bfX}}}
\hspace{25pt}
\xymatrix@C=30pt@M=5pt{
\BDChol(\D_{\bfX})^{\op}\ar@{->}[r]^\sim\ar@<1.0ex>@{}[r]^-{\Sol_{\bfX}^{\rmE}}
\ar@{->}[d]_-{(\cdot)_\reg}\ar@{}[rd]|{\rotatebox[origin=c]{180}{$\circlearrowright$}}
 & \BEC_{\CC-c}(\I\CC_{\bfX})\ar@{->}[d]^-{\sh_{\bfX}}\\
\BDCrh(\D_{\bfX})^{\op}\ar@{->}[r]_-{\Sol_{\bfX}}^-{\sim}&\BDC_{\CC-c}(\CC_{\bfX}).
}\]
See \cite[Corollaries 3.27, 3.28]{Ito19} for the details.

We set
\begin{align*}
{}^p\bfE^{\leq0}_{\CC-c}(\I\CC_{\bfX}) &:=
\{K\in\BEC_{\CC-c}(\I\CC_{\bfX})\ |\ \sh_{\bfX}(K)\in{}^p\bfD^{\leq0}_{\CC-c}(\CC_{\bfX})\},\\
{}^p\bfE^{\geq0}_{\CC-c}(\I\CC_{\bfX}) &:=
\{K\in\BEC_{\CC-c}(\I\CC_{\bfX})\ |\ \rmD_{\bfX}^{\rmE}(K)\in{}^p\bfE^{\leq0}_{\CC-c}(\I\CC_{\bfX})\}\\
&=
\{K\in\BEC_{\CC-c}(\I\CC_{\bfX})\ |\ \sh_{\bfX}(K)\in{}^p\bfD^{\geq0}_{\CC-c}(\CC_{\bfX})\},
\end{align*}
where the pair $\big({}^p\bfD^{\leq0}_{\CC-c}(\CC_{\bfX}), {}^p\bfD^{\geq0}_{\CC-c}(\CC_{\bfX})\big)$
is the perverse t-structure on $\BDC_{\CC-c}(\CC_{\bfX})$.

\begin{theorem}[{\cite[Theorem 4.4]{Ito19}}]\label{thm2.28}
For any $\M\in\BDChol(\D_{\bfX})$,
we have
\begin{itemize}
\item[\rm(1)]
$\M\in\DChol^{\leq0}(\D_{\bfX})\Longleftrightarrow
\Sol_{\bfX}^{\rmE}(\M)[d_{\bfX}]\in{}^p\bfE^{\geq0}_{\CC-c}(\I\CC_{\bfX})$,
\item[\rm(2)]
$\M\in\DChol^{\geq0}(\D_{\bfX})\Longleftrightarrow
\Sol_{\bfX}^{\rmE}(\M)[d_{\bfX}]\in{}^p\bfE^{\leq0}_{\CC-c}(\I\CC_{\bfX})$.
\end{itemize}

Therefore, the pair $\big({}^p\bfE^{\leq0}_{\CC-c}(\I\CC_{\bfX}),
{}^p\bfE^{\geq0}_{\CC-c}(\I\CC_{\bfX})\big)$
is a t-structure on $\BEC_{\CC-c}(\I\CC_{\bfX})$
and
its heart $$\Perv(\I\CC_{\bfX}) :=
{}^p\bfE^{\leq0}_{\CC-c}(\I\CC_{\bfX})\cap{}^p\bfE^{\geq0}_{\CC-c}(\I\CC_{\bfX})$$ is equivalent to the abelian category $\Modhol(\D_{\bfX})$
of holonomic $\D_{\bfX}$-modules.
\end{theorem}

Let us recall that
there exists a generalized t-structure 
$\big({}^{\frac{1}{2}}\bfE_{\RR-c}^{\leq c}(\I\CC_\bfX),
{}^{\frac{1}{2}}\bfE_{\RR-c}^{\geq c}(\I\CC_\bfX)\big)_{c\in\RR}$
on $\BEC_{\RR-c}(\I\CC_{\bfX})$
by \cite[Theorem 3.5.2 (i)]{DK16-2}.
Note that we have
\begin{align*}
{}^p\bfE_{\CC-c}^{\leq 0}(\I\CC_{\bfX}) &=
{}^{\frac{1}{2}}\bfE_{\RR-c}^{\leq 0}(\I\CC_{\bfX})\cap
\BEC_{\CC-c}(\I\CC_{\bfX}),\\
{}^p\bfE_{\CC-c}^{\geq 0}(\I\CC_{\bfX})
&={}^{\frac{1}{2}}\bfE_{\RR-c}^{\geq 0}(\I\CC_{\bfX})\cap
\BEC_{\CC-c}(\I\CC_{\bfX}),
\end{align*}
see \cite[Corollary 4.5]{Ito19} for the details.

\section{Main Results}\label{sec3}
In this section,
we define algebraic $\CC$-constructible enhanced ind-sheaves and prove that
the triangulated category of them is equivalent to the one of algebraic holonomic $\D$-modules.

\subsection{The Condition (AC)}\label{sec-AC}
The similar results of this section for the analytic case is proved in \cite[\S3.5]{Ito19}.

Let $X$ be a smooth algebraic variety over $\CC$ and
denote by $X^\an$ the underlying complex analytic manifold of $X$.
Recall that an algebraic stratification of $X$ is
a Zariski locally finite partition $\{X_\alpha\}_{\alpha\in A}$ of $X$
by locally closed subvarieties $X_\alpha$
such that for any $\alpha\in A$ $X_\alpha$ is smooth and
$\var{X}_\alpha = \sqcup_{\beta\in B} X_{\beta}$ for a subset $B\subset A$.
Moreover an algebraic stratification $\{X_\alpha\}_{\alpha\in A}$ of $X$ induces
a complex analytic stratification $\{X^\an_\alpha\}_{\alpha\in A}$ of $X^\an$.

\begin{definition}\label{def-AC}
We say that an enhanced ind-sheaf $K\in\ZEC(\I\CC_{X^\an})$ satisfies the condition 
$\AC$
if there exists an algebraic stratification $\{X_\alpha\}_\alpha$ of $X$ such that
n$$\pi^{-1}\CC_{(\var{X}^{\blow}_\alpha)^\an \setminus D^\an_\alpha}
\otimes \bfE (b^\an_\alpha)^{-1}K$$
has a modified quasi-normal form along $D^\an_\alpha$ for any $\alpha$,
where $b_\alpha \colon \var{X}^\blow_\alpha \to X$ is a blow-up of $\var{X_\alpha}$
along $\partial X_\alpha := \var{X_\alpha}\setminus X_\alpha$,
$D_\alpha := b_\alpha^{-1}(\partial X_\alpha)$
and $D_\alpha^\an := \big(\var{X}_\alpha^\blow\big)^\an\setminus
\big(\var{X}_\alpha^\blow\setminus D_\alpha\big)^\an$.
Namely $\var{X}^\blow_\alpha$ is a smooth algebraic variety over $\CC$,
$D_\alpha$ is a normal crossing divisor of $\var{X}^\blow_\alpha$
and $b_\alpha$ is a projective map
which induces an isomorphism $\var{X}^\blow_\alpha\setminus D_\alpha\simto X_\alpha$
and satisfies $b_\alpha\big(\var{X}^\blow_\alpha\big)=\var{X_\alpha}$.

We call such a family $\{X_\alpha\}_{\alpha\in A}$ an algebraic stratification adapted to $K$.
\end{definition}

We denote by $\ZEC_{\CC-c}(\I\CC_X)$ the full subcategory of $\ZEC(\I\CC_{X^\an})$
whose objects satisfy the condition 
$\AC$.
Note that $\ZEC_{\CC-c}(\I\CC_X)$ is the full subcategory of
the abelian category $\ZEC_{\CC-c}(\I\CC_{X^\an})$ of $\CC$-constructible enhanced ind-sheaves on $X^\an$, see Definition \ref{def-const}.
Moreover we set
\[\BEC_{\CC-c}(\I\CC_X) := \{K\in\BEC(\I\CC_{X^\an})\
|\ \SH^i(K)\in\ZEC_{\CC-c}(\I\CC_X) \mbox{ for any }i\in\ZZ \}
\subset \BEC_{\CC-c}(\I\CC_{X^\an}).\]

\begin{remark}\label{rem3.2}
Definition \ref{def-AC} does not depend on the choice of
the complex blow-up $b^\an_\alpha$ by Remark \ref{rem2.25}.
In particular, Definition \ref{def-AC} does not depend on a choice
of a blow-up $b_\alpha$.
Hence we obtain the following properties: 
\begin{itemize}
\item[(1)]
By Hironaka's desingularization theorem \cite{Hiro} (see also \cite[Theorem 4.3]{Naga}),
we can take $\var{X}^\blow_\alpha$ in Definition \ref{def-AC}
as a smooth complete algebraic variety.

\item[(2)]
Let $\{X_\alpha\}_{\alpha\in A}$ be an algebraic stratification of $X$
adapted to $K\in\ZEC_{\CC-c}(\I\CC_X)$.
Then any algebraic stratification of $X$ which is finer than the one $\{X_\alpha\}_{\alpha\in A}$
is also adapted to $K$,
see \cite[Sublemma 3.22]{Ito19} for analytic case.

\item[(3)]
For any $K, L\in\ZEC_{\CC-c}(\I\CC_X)$,
there exists a common algebraic stratification $\{X_\alpha\}_\alpha$ adapted to $K$ and $L$
with a common blow-up of $\var{X_\alpha}$ along $\partial X_\alpha$,
 \cite[Lemma 3.23]{Ito19} for analytic case.
\end{itemize}
\end{remark}

\begin{proposition}
The category $\ZEC_{\CC-c}(\I\CC_X)$ is
the full abelian subcategory of $\ZEC_{\CC-c}(\I\CC_{X^\an})$.
Hence the category $\BEC_{\CC-c}(\I\CC_X)$ is
a full triangulated subcategory of $\BEC_{\CC-c}(\I\CC_{X^\an})$.
\end{proposition}

\begin{proof}
It is enough to show that the kernel and the cokernel of a morphism $\Phi \colon K\to L$
of objects of $\ZEC_{\CC-c}(\I\CC_X)$ also satisfy the condition 
$\AC$.
By Remark \ref{rem3.2} (2) and Lemma \ref{lem-modi},
we can take a common algebraic stratification $\{X_\alpha\}_\alpha$ adapted to $K$ and $L$
with a common blow-up $b_\alpha : \var{X}^\blow_\alpha\to X$ of $\var{X_\alpha}$
along $\partial X_\alpha$
such that there exist analytic meromorphic connections $\M_\alpha $, $\N_\alpha$
on $\big(\var{X}^\blow_\alpha\big)^\an$ along $D^\an_\alpha$ satisfying the following isomorphisms
\begin{align*}
\pi^{-1}\CC_{(\var{X}^\blow_\alpha)^\an\setminus D^\an_\alpha}
\otimes \bfE (b^\an_\alpha)^{-1} K
&\simto \Sol_{(\var{X}^\blow_\alpha)^\an}^{\rmE}(\M_\alpha),\\
\pi^{-1}\CC_{(\var{X}^\blow_\alpha)^\an\setminus D^\an_\alpha}
\otimes \bfE (b^\an_\alpha)^{-1}L
&\simto \Sol_{(\var{X}^\blow_\alpha)^\an}^{\rmE}(\N_\alpha),
\end{align*}
where we set $D_\alpha = b_\alpha^{-1}(\partial X_\alpha)$.
Let $$\varphi_\alpha :=
\RH_{(\var{X}^\blow_\alpha)^\an}^\rmE(
\pi^{-1}\CC_{(\var{X}^\blow_\alpha)^\an\bs D^\an_\alpha}\otimes \bfE (b^\an_\alpha)^{-1}\Phi)
\colon \N_\alpha\to \M_\alpha$$ be the morphism of meromorphic connections
induced by the morphism $\Phi \colon K\to L$.
Since the category of meromorphic connections is abelian,
the cokernel $\Coker\varphi_\alpha$ of $\varphi_\alpha$
is a meromorphic connection on $\big(\var{X}^\blow_\alpha\big)^\an$ along $D^\an_\alpha$.
Moreover, we have $$\Ker\big(\Sol_{(\var{X}^\blow_\alpha)^\an}^\rmE(\varphi_\alpha)\big)
\simeq \Sol_{(\var{X}^\blow_\alpha)^\an}^{\rmE}(\Coker\varphi_\alpha)$$
because there exists an equivalence of abelian categories
$$\Sol_{(\var{X}^\blow_\alpha)^\an}^\rmE\colon
\Conn\big(\big(\var{X}^\blow_\alpha\big)^\an; D^\an_\alpha\big)^{\op}
\simto\ZECmero(\I\CC_{(\var{X}^\blow_\alpha)^\an(D^\an_\alpha)})$$
by Lemma \ref{lem-modi}.
Then we obtain a commutative diagram
\[\hspace{-33pt}\xymatrix@C=13pt@R=20pt{
0\ar@{->}[r] &
\pi^{-1}\CC_{(\var{X}^\blow_\alpha)^\an\setminus D^\an_\alpha}
\otimes\bfE (b^\an_\alpha)^{-1}(\Ker\Phi)\ar@{->}[r]\ar@{.>}[d]_{{}^\exists}^-\wr
& \pi^{-1}\CC_{(\var{X}^\blow_\alpha)^\an\setminus D^\an_\alpha}
\otimes\bfE (b^\an_\alpha)^{-1}K \ar@{->}[r]\ar@{->}[d]^-\wr
& \pi^{-1}\CC_{(\var{X}^\blow_\alpha)^\an\setminus D^\an_\alpha}
\otimes\bfE (b^\an_\alpha)^{-1}L \ar@{->}[d]^-\wr\\
0\ar@{->}[r] & 
\Ker\big(\Sol_{(\var{X}^\blow_\alpha)^\an}^\rmE(\varphi_\alpha)\big)
\ar@{->}[r]\ar@{-}[d]^-\wr
& \Sol_{(\var{X}^\blow_\alpha)^\an}^{\rmE}(\M_\alpha)
\ar@{->}[r]_-{\Sol_{(\var{X}^\blow_\alpha)^\an}^{\rmE}(\varphi_\alpha)}
& \Sol_{(\var{X}^\blow_\alpha)^\an}^{\rmE}(\N_\alpha).\\
{} & \Sol_{(\var{X}^\blow_\alpha)^\an}^{\rmE}(\Coker\varphi_\alpha) & {} & {}
}\]
Therefore, we have
$\pi^{-1}\CC_{(\var{X}^\blow_\alpha)^\an\setminus D^\an_\alpha}
\otimes\bfE (b^\an_\alpha)^{-1}(\Ker\Phi)
\in\ZECmero(\I\CC_{(\var{X}^\blow_\alpha)^\an(D^\an_\alpha)})$.
Namely it has a modified quasi-normal form along $D^\an_\alpha$
and hence $\Ker\Phi$ satisfies the condition 
$\AC$. 
Similarly we can show that $\Coker\Phi$ satisfies the condition 
$\AC$.
\end{proof}

For $\M\in\BDChol(\D_X)$, we set
\begin{align*}
\Sol_X^\rmE(\M) &:= \Sol_{X^\an}^{\rmE}(\M^\an)\in\BEC(\I\CC_{X^\an}),\\
\DR_X^\rmE(\M) &:= \DR_{X^\an}^{\rmE}(\M^\an)\in\BEC(\I\CC_{X^\an}).
\end{align*}
By Theorem \ref{thm-const},
we have $\Sol_X^\rmE(\M)\in\BEC_{\CC-c}(\I\CC_{X^\an})$
for any $\M\in\BDChol(\D_X)$. 
Moreover we obtain the following assertion:
\begin{proposition}\label{prop3.4}
For any $\M\in\BDChol(\D_X)$
the enhanced solution complex $\Sol_X^{\rmE}(\M)$ of $\M$
is an object of $\BEC_{\CC-c}(\I\CC_X)$.
\end{proposition}
\begin{proof}
Since the category $\BDChol(\D_X)$ is a full triangulated subcategory of $\BDC(\D_X)$
and the category $\BEC_{\CC-c}(\I\CC_X)$ is a full triangulated subcategory
of $\BEC_{\CC-c}(\I\CC_{X^\an})$, 
it is enough to show the assertion in the case of $\M\in\Modhol(\D_X)$ 
by induction on the length of a complex.

Let $\M\in\Modhol(\D_X)$ and we put $K := \Sol_X^\rmE(\M)\in\BEC_{\CC-c}(\I\CC_{X^\an})$.
By Lemma \ref{lem-str}, 
there exist an algebraic stratification $\{X_\alpha\}_{\alpha\in A}$ of $X$
and a blow-up $b_\alpha \colon \var{X}^\blow_\alpha \to X$ of $\var{X_\alpha}$
along $\partial X_\alpha$ for each $\alpha\in A$
such that $(\bfD b_\alpha^\ast \M)(\ast D_\alpha)
\in\BDCmero(\D_{\var{X}^\blow_\alpha(D_\alpha)})$,
where $D_\alpha := b_\alpha^{-1}(\partial X_\alpha)$
is a normal crossing divisor.
Then we have
$$\pi^{-1}\CC_{(\var{X}^\blow_\alpha)^\an\setminus D^\an_\alpha}
\otimes \bfE (b^\an_\alpha)^{-1}K
\simeq
\Sol_{\var{X}^\blow_\alpha}^\rmE\big((\bfD b_\alpha^\ast \M)(\ast D_\alpha)\big)
\in\BECmero(\I\CC_{(\var{X}^\blow_\alpha)^\an(D^\an_\alpha)})$$
for any $\alpha\in A$,
where we used 
$$\BDCmero(\D_{\var{X}^\blow_\alpha(D_\alpha)})\simto
\BDCmero(\D_{(\var{X}^\blow_\alpha)^\an(D_\alpha^\an)})\simto
\BECmero(\I\CC_{(\var{X}^\blow_\alpha)^\an(D^\an_\alpha)}),$$
see Lemma \ref{lem2.17} and Proposition \ref{prop2.23} for the details.
Since the functor
$\pi^{-1}\CC_{(\var{X}^\blow_\alpha)^\an\setminus D^\an_\alpha}\otimes
\bfE (b^\an_\alpha)^{-1}(\cdot)$ is t-exact with respect to the standard t-structure
(see \cite[Proposition 2.7.3 (iv) and Lemma 2.7.5 (i)]{DK16-2})
we have 
\begin{align*}
\pi^{-1}\CC_{(\var{X}^\blow_\alpha)^\an\setminus D^\an_\alpha}
\otimes \bfE (b^\an_\alpha)^{-1}(\SH^iK)
\simeq
\SH^i(\pi^{-1}\CC_{(\var{X}^\blow_\alpha)^\an\setminus D^\an_\alpha}
\otimes \bfE (b^\an_\alpha)^{-1}K)
\in\ZECmero(\I\CC_{(\var{X}^\blow_\alpha)^\an(D^\an_\alpha)})
\end{align*}
for any $i\in\ZZ$.
Namely the enhanced ind-sheaf
$\pi^{-1}\CC_{(\var{X}^\blow_\alpha)^\an\setminus D^\an_\alpha}
\otimes \bfE (b^\an_\alpha)^{-1}(\SH^iK)
\in\ZEC(\I\CC_{(\var{X}^\blow_\alpha)^\an})$ has a modified quasi-normal form
along $D^\an_\alpha$.
Therefore any cohomology of $K$ satisfies the condition 
$\AC$ and
hence $K\in\BEC_{\CC-c}(\I\CC_X)$.
\end{proof}

On the other hand, we have
\begin{proposition}\label{prop3.5}
For any $K\in\BEC_{\CC-c}(\I\CC_{X})$,
there exists $\M\in\BDChol(\D_X)$
such that $$K\simto \Sol_{X}^{\rmE}(\M).$$
\end{proposition} 
\begin{proof}
By induction on the length of a complex,
it is enough to show in the case of $K\in\ZEC_{\CC-c}(\I\CC_X)$.
Let $\{X_\alpha\}_{\alpha\in A}$ be an algebraic stratification of $X$ adapted to $K$
and we put $$ Y_k := \bigsqcup_{\dim X_\alpha\leq k}X_\alpha,
\hspace{17pt}
S_k := Y_k\setminus Y_{k-1} = \bigsqcup_{\dim X_\alpha = k}X_\alpha
\hspace{17pt}
\mbox{for any } k=0, 1, \ldots, d_X.$$
Then  $Y_0=S_0$ and $X=Y_{d_X}$.
Furthermore, there exists a distinguished triangle
\[\pi^{-1}\CC_{S^\an_k}\otimes K\to
\pi^{-1}\CC_{Y^\an_k}\otimes K \to
\pi^{-1}\CC_{Y^\an_{k-1}}\otimes K\xrightarrow{+1}.\]
Hence, by induction on $k$,
it is enough to show that 
$\pi^{-1}\CC_{S^\an_k}\otimes K\in\Sol_X^{\rmE}(\BDChol(\D_X))$ for any $k$.

Let $S_i$ be decomposed into $Z_1\sqcup\cdots \sqcup Z_{m_i}$
with some strata $Z_1, \ldots, Z_{m_i}\in\{X_\alpha\}_{\alpha\in A}$.
If $m_i=1$, by Lemma \ref{lem3.9} below
we have $\pi^{-1}\CC_{S^\an_i}\otimes K\in\Sol_X^{\rmE}(\BDChol(\D_X))$.
We shall prove the case of $m_i\geq2$.
In this case there exists a distinguished triangle
\[\pi^{-1}\CC_{Z^\an_1}\otimes K\to
\pi^{-1}\CC_{Z^\an_1\sqcup\cdots \sqcup Z^\an_j}\otimes K \to
\pi^{-1}\CC_{Z^\an_2\sqcup\cdots \sqcup Z^\an_j}\otimes K\xrightarrow{+1}\]
for any $j=2, \ldots, m_i$.
Hence, by induction on $j$
it is enough to show that 
$\pi^{-1}\CC_{Z^\an_1}\otimes K\in\Sol_X^{\rmE}(\BDChol(\D_X))$.
However, it follows from Lemma \ref{lem3.9} below.
\end{proof}

\begin{lemma}\label{lem3.9}
For $K\in\ZEC_{\CC-c}(\I\CC_{X})$
there exists an algebraic stratification $\{X_\alpha\}_\alpha$ of $X$ such that 
\[\pi^{-1}\CC_{X^\an_\alpha}\otimes K\in\Sol_X^\rmE\big(\BDChol(\D_X)\big).\]
\end{lemma}

\begin{proof}
By Sublemma \ref{sublem3.8} below,
there exist an algebraic stratification $\{X_\alpha\}_\alpha$ of $X$
and algebraic meromorphic connections $\M_\alpha$ on $\var{X}^\blow_\alpha$ along $D_\alpha$
such that
$$\pi^{-1}\CC_{(\var{X}^\blow_\alpha)^\an\setminus D^\an_\alpha}
\otimes \bfE (b^\an_\alpha)^{-1}K
\simeq\Sol_{\var{X}^\blow_\alpha}^{\rmE}(\M_\alpha)
\hspace{20pt} (\mbox{for any $\alpha$}),$$
where $b_\alpha \colon \var{X}^\blow_\alpha \to X$ is a blow-up of $\var{X_\alpha}$
along $\partial X_\alpha$ and
$D_\alpha := b_\alpha^{-1}(\partial X_\alpha)$.
By applying the direct image functor $\bfE b^\an_{\alpha !!}$
we obtain an isomorphism in $\BEC(\I\CC_{X^\an})$
$$\pi^{-1}\CC_{X^\an_\alpha}\otimes K\simto
\Sol_X^{\rmE}(\bfD b_{\alpha\ast}\N_\alpha)[d_X-d_{X_\alpha}].$$
Here we used Theorem \ref{thm-DK} (3) and Proposition \ref{prop2.12} (3).
Moreover, by \cite[Theorem 3.2.3 (i)]{HTT08},
we have $\bfD b_{\alpha\ast}\N_\alpha\in\BDChol(\D_X)$
and hence the proof is completed.
\end{proof}

\begin{sublemma}\label{sublem3.8}
Let $K$ be an object of $\ZEC(\I\CC_{X^\an})$.
Then the following two conditions are equivalent:
\begin{itemize}
\item[\rm(1)]
$K$ satisfies the condition 
$\AC$.

\item[\rm(2)]
There exist an algebraic stratification $\{X_\alpha\}_\alpha$ of $X$
and algebraic meromorphic connections $\M_\alpha$ on $\var{X}^\blow_\alpha$ along $D_\alpha$
such that for any $\alpha$
$$\pi^{-1}\CC_{(\var{X}^\blow_\alpha)^\an\setminus D^\an_\alpha}
\otimes \bfE (b^\an_\alpha)^{-1}K
\simeq\Sol_{\var{X}^\blow_\alpha}^{\rmE}(\M_\alpha),$$
where $b_\alpha \colon \var{X}^\blow_\alpha \to X$ is a blow-up of $\var{X_\alpha}$
along $\partial X_\alpha$ and
$D_\alpha := b_\alpha^{-1}(\partial X_\alpha)$.
\end{itemize}
\end{sublemma}
\begin{proof}
By Remark \ref{rem3.2} (1),
this follows from Lemmas \ref{lem2.17}, \ref{lem-modi}.
\end{proof}

Hence we obtain an essential surjective functor
\[\Sol_X^\rmE \colon \BDChol(\D_X)\to\BDC_{\CC-c}(\I\CC_X).\]
This is not fully faithful in general.

However if $X$ is complete, then this is fully faithful.
\begin{theorem}\label{thm-main-1}
Let $X$ be a smooth complete algebraic variety over $\CC$.
Then there exists an equivalence of triangulated categories
\[\Sol_X^{\rmE} \colon \BDChol(\D_X)^{\op}\simto \BEC_{\CC-c}(\I\CC_X).\]
\end{theorem}
\begin{proof}
It is enough to show
\[\Hom_{\BDChol(\D_X)}(\M, \N)\simto
\Hom_{\BEC_{\CC-c}(\I\CC_X)}(\Sol_X^{\rmE}(\N), \Sol_X^{\rmE}(\M))\]
for any $\M, \N\in\BDChol(\D_X)$.
By \cite[Lemma 4.5.14]{DK16},
this follows by taking the $0$-th cohomology in Lemma \ref{lem3.6} below.
\end{proof}

\begin{lemma}\label{lem3.6}
Let $X$ be a smooth complete algebraic variety over $\CC$.
For any $\M, \N\in\BDChol(\D_X)$,
there exists an isomorphism in $\BDC(\CC)$
\[\rHom_{\D_X}(\M, \N)\simto\rHom^\rmE(\Sol_X^{\rmE}(\N), \Sol_X^{\rmE}(\M)).\]
\end{lemma}

\begin{proof}
Let us denote by $p_{X^\an} \colon X^\an\to \{\pt\}$ the map from $X^\an$ to the set of one point.
Recall that the right-hand side is isomorphic to 
\[\bfR\Gamma(X^\an; \rhom^\rmE(\Sol_X^{\rmE}(\N), \Sol_X^{\rmE}(\M)))
\simeq
\bfR p_{X^\an\ast}\rhom^\rmE(\Sol_X^{\rmE}(\N), \Sol_X^{\rmE}(\M)).\]

By $\Sol_X^\rmE(\N)\in\BEC_{\CC-c}(\I\CC_{X})\subset\BEC_{\RR-c}(\I\CC_{X^\an})$
and the fact that any $\RR$-constructible enhanced ind-sheaf is stable (see \cite[\S 4.9]{DK16}), 
we have 
\begin{align*}
\rhom^\rmE\big(\Sol_X^{\rmE}(\N), \Sol_X^{\rmE}(\M)\big)
&\simeq
\rhom^\rmE\big(\CC_{X^\an}^\rmE\Potimes\Sol_X^{\rmE}(\N), \Sol_X^{\rmE}(\M)\big)\\
&\simeq
\rhom^\rmE\big(\CC_{X^\an}^\rmE, \Prihom\big(\Sol_X^{\rmE}(\N), \Sol_X^{\rmE}(\M)\big)\big),
\end{align*}
where in the second isomorphism
by the enhanced hom-tensor adjunction \cite[Lemma 4.5.15]{DK16}.
Furthermore we have isomorphisms in $\BEC(\I\CC_{X^\an})$
\begin{align*}
\Prihom\big(\Sol_X^{\rmE}(\N), \Sol_X^{\rmE}(\M)\big)
&\simeq
\rmD_{X^\an}^{\rmE}\big(
\Sol_X^{\rmE}(\N)\Potimes \rmD_{X^\an}^{\rmE}\Sol_X^{\rmE}(\M)\big)\\
&\simeq
\rmD_{X^\an}^{\rmE}\big(\Sol_X^{\rmE}(\N)\Potimes \Sol_X^{\rmE}(\DD_X\M)[2d_X]\big)\\
&\simeq
\rmD_{X^\an}^{\rmE}\big(\Sol_X^{\rmE}\big(\N\Dotimes\DD_X\M\big)[2d_X]\big)\\
&\simeq
\DR_{X}^{\rmE}(\N\Dotimes\DD_X\M)[-d_X],
\end{align*}
where in the first (resp.\ second, third, forth) isomorphism
we used \cite[Proposition 4.9.13 (2)]{DK16}
(resp.\ Theorem \ref{thm-DK} (1), Theorem \ref{thm-DK} (4),
\cite[Corollary 9.4.9]{DK16}).
By \cite[Theorem 9.1.2 (iii)]{DK16}, Proposition \ref{prop2.12}(3)
and the fact that the map $p_{X^\an}$ is proper
there exists an isomorphism in $\BEC(\I\CC_{\{\pt\}})$
\[\bfE p_{X^\an\ast}\DR_X^{\rmE}(\N\Dotimes\DD_X\M)\simeq
\DR_{\{\pt\}}^{\rmE}\big(\bfD p_{X\ast}(\N\Dotimes\DD_X\M)\big).\]
Hence we have isomorphisms in $\BDC(\CC)$
\begin{align*}
\bfR p_{X^\an\ast}\rhom^\rmE\big(\Sol_X^{\rmE}(\N), \Sol_X^{\rmE}(\M)\big)
&\simeq
\bfR p_{X^\an\ast}\rhom^\rmE\big(\CC_{X^\an}^\rmE,
\Prihom\big(\Sol_X^{\rmE}(\N), \Sol_X^{\rmE}(\M)\big)\big)\\
&\simeq
\bfR p_{X^\an\ast}\rhom^\rmE\big(\CC_{X^\an}^\rmE,
\DR_X^{\rmE}(\N\Dotimes\DD_X\M)[-d_X]\big)\\
&\simeq
\bfR p_{X^\an\ast}\rhom^\rmE\big(\bfE p_{X^\an}^{-1}\CC_{\{\pt\}}^\rmE,
\DR_X^{\rmE}(\N\Dotimes\DD_X\M)[-d_X]\big)\\
&\simeq
\rhom^\rmE\big(\CC_{\{\pt\}}^\rmE,
\bfE p_{{X^\an}\ast}\DR_X^{\rmE}(\N\Dotimes\DD_X\M)[-d_X]\big)\\
&\simeq
\rhom^\rmE\big(\CC_{\{\pt\}}^\rmE,
\DR_{\{\pt\}}^{\rmE}\big(\bfD p_{X\ast}(\N\Dotimes\DD_X\M)\big)[-d_X]\big)\\
&\simeq
\bfD p_{X\ast}(\N\Dotimes\DD_X\M)[-d_X],
\end{align*}
where in the forth isomorphism we used \cite[Lemma 4.5.17]{DK16}
and in the last isomorphism we used the fact that
$\rhom^\rmE\big(\CC_{\{\pt\}}^\rmE, \DR_{\{\pt\}}^{\rmE}(\cdot)\big)
\simeq\rhom^\rmE\big(\CC_{\{\pt\}}^\rmE, e_{\{\pt\}}(\cdot)\big)
\simeq \id$ (e.g., see \cite[Example 3.5.9]{DK16-2} for $\DR_{\{\pt\}}^\rmE=e_{\{\pt\}}$
and see the proof of \cite[Proposition 4.7.15]{DK16} for the second isomorphism).
On the other hand, 
we have an isomorphism in $\BDC(\CC)$
\[\bfD p_{X\ast}(\N\Dotimes\DD_X\M)[-d_X]
\simeq
\rHom_{\D_X}(\M, \N)\]
by \cite[Corollary 2.6.15]{HTT08}.
Hence the proof is completed.
\end{proof}

We obtain the following corollary, although this is known by experts.
\begin{corollary}
Let $X$ be a smooth complete algebraic variety over $\CC$.
Then the analytification functor $(\cdot)^\an \colon \Mod(\D_X)\to\Mod(\D_{X^\an})$
induces a fully faithful embedding
\[(\cdot)^\an \colon \Modhol(\D_X)\hookrightarrow\Modhol(\D_{X^\an}).\]  
Moreover we obtain a fully faithful embedding
\[(\cdot)^\an \colon \BDChol(\D_X)\hookrightarrow\BDChol(\D_{X^\an}).\]  
\end{corollary}
\begin{proof}
Since the proof of the first assertion is similar, we only prove the last one.
For $\M, \N\in\BDChol(\D_X)$, we have isomorphisms
\begin{align*}
\Hom_{\BDChol(\D_X)}(\M, \N)
&\simeq
\Hom_{\BEC_{\CC-c}(\I\CC_X)}\big(\Sol_X^{\rmE}(\N), \Sol_X^{\rmE}(\M)\big)\\
&\simeq
\Hom_{\BEC_{\CC-c}(\I\CC_X)}
\big(\Sol_{X^\an}^{\rmE}(\N^\an), \Sol_{X^\an}^{\rmE}(\M^\an)\big)\\
&\simeq
\Hom_{\BDChol(\D_{X^\an})}\big(\M^\an, \N^\an\big),
\end{align*}
where in the first (resp.\ last) isomorphism
we used Theorem \ref{thm-main-1} (resp.\ Theorem \ref{thm2.6} (1)).
Therefore the proof is completed.

\end{proof}

\subsection{The General Case}
In this subsection we consider the general case. 
Thanks to Hironaka's desingularization theorem \cite{Hiro} (see also \cite[Theorem 4.3]{Naga}),
for any smooth algebraic variety $X$ over $\CC$
we can take a smooth complete algebraic variety $\tl{X}$ such that $X\subset \tl{X}$
and $D := \tl{X}\setminus X$ is a normal crossing divisor of $\tl{X}$.

Let us consider a bordered space $X^\an_\infty = (X^\an, \tl{X}^\an)$
and the triangulated category $\BEC(\I\CC_{X^\an_\infty})$ of enhanced ind-sheaves
on $X^\an_\infty$.
Recall that $\BEC(\I\CC_{X^\an_\infty})$ does not depend on the choice of $\tl{X}$
and there exists an equivalence of triangulated categories
\[\xymatrix@C=55pt{
\BEC(\I\CC_{X^\an_\infty})\ar@<0.7ex>@{->}[r]^-{\bfE j_{!!}}\ar@{}[r]|-{\sim}
&
\{K\in\BEC(\I\CC_{\tl{X}^\an})\ |\ \pi^{-1}\CC_{X^\an}\otimes K\simto K\}
\ar@<0.7ex>@{->}[l]^-{\bfE j^{-1}},
}\]
where we denote by $j\colon X^\an_\infty \to \tl{X}^\an$ the morphism of bordered spaces 
given by the open embedding $X\hookrightarrow \tl{X}$ for simplicity,
see \S \ref{subsec2.5} for the details.

We shall denote the open embedding $X\hookrightarrow \tl{X}$ by the same symbol $j$
and set 
\[\Sol_{X_\infty}^\rmE(\M) := 
\bfE j^{-1}\Sol_{\tl{X}}^\rmE(\bfD j_\ast\M)
\in\BEC(\I\CC_{X^\an_\infty})\]
for any $\M\in\BDC(\D_X)$.
By Theorem \ref{thm-DK} (5) and the fact that for any $\M\in\BDChol(\D_X)$
there exist isomorphisms
$$(\bfD j_\ast\M)^\an(\ast D^\an)\simeq
\big((\bfD j_\ast\M)(\ast D)\big)^\an\simeq(\bfD j_\ast\M)^\an,$$
 in $\BDChol(\D_{\tl{X}^\an})$,
we have
$$\Sol_{\tl{X}}^\rmE(\bfD j_\ast\M)\in
\{K\in\BEC(\I\CC_{\tl{X}^\an})\ |\ \pi^{-1}\CC_{X^\an}\otimes K\simto K\}$$
 for any $\M\in\BDChol(\D_X)$.
Furthermore since $\bfD j_\ast\M\simeq(\bfD j_!\M)(\ast D)$ for any $\M\in\BDChol(\D_X)$,
we have
\[\Sol_{X_\infty}^\rmE(\M) \simeq 
\bfE j^{-1}\Sol_{\tl{X}}^\rmE(\bfD j_!\M)
\mbox{ in $\BEC(\I\CC_{X^\an_\infty})$}\]
for any $\M\in\BDChol(\D_X)$.

Moreover we obtain some functorial properties of
the enhanced solution functor $\Sol_{X_\infty}^\rmE$ on $X_\infty$,
see Proposition \ref{prop-comm} below. 

\begin{definition}\label{def3.11}
We say that an enhanced ind-sheaf $K\in\BEC(\I\CC_{X^\an_\infty})$ is
algebraic $\CC$-constructible on $X_\infty^\an$
if $\bfE j_{!!}K \in\BEC(\I\CC_{\tl{X}^\an})$ is an object of $\BEC_{\CC-c}(\I\CC_{\tl{X}})$.

We denote by $\BEC_{\CC-c}(\I\CC_{X_\infty})$
the full triangulated subcategory of $\BEC(\I\CC_{X^\an_\infty})$
consisting of algebraic $\CC$-constructible enhanced ind-sheaves on $X_\infty^\an$.
\end{definition}
By Lemma \ref{lem2.7} and
the fact that the triangulated category $\BEC_{\CC-c}(\I\CC_{X^\an})$ is
a full triangulated subcategory of $\BEC_{\RR-c}(\I\CC_{X^\an})$,
the one $\BEC_{\CC-c}(\I\CC_{X^\an_\infty})$ is also a full triangulated subcategory
of $\BEC_{\RR-c}(\I\CC_{X^\an_\infty})$.

Then we obtain the first main theorem of this paper.
\begin{theorem}\label{thm-main-2}
Let $X$ be a smooth algebraic variety over $\CC$.
For any $\M\in\BDChol(\D_X)$, the enhanced solution complex $\Sol_{X_\infty}^\rmE(\M)$
of $\M$ is an algebraic $\CC$-constructible enhanced ind-sheaf.
On the other hand, 
for any algebraic $\CC$-constructible enhanced ind-sheaf $K\in\BEC_{\CC-c}(\I\CC_{X_\infty})$,
there exists $\M\in\BDChol(\D_X)$
such that $$K\simto \Sol_{X_\infty}^{\rmE}(\M).$$
Moreover, we obtain an equivalence of triangulated categories
\[\Sol_{X_\infty}^{\rmE} \colon \BDChol(\D_X)^{\op}\simto \BEC_{\CC-c}(\I\CC_{X_\infty}).\]
\end{theorem}

\begin{proof}
For any $\M\in\BDChol(\D_X)$, we have 
$\Sol_{\tl{X}}^\rmE(\bfD j_\ast\M)\in\BEC_{\CC-c}(\I\CC_{\tl{X}})$
by Proposition \ref{prop3.4}.
Then there exist isomorphisms
\begin{align*}
\bfE j_{!!}\Sol_{X_\infty}^\rmE(\M) 
&\simeq
\bfE j_{!!}\bfE j^{-1}\Sol_{\tl{X}}^\rmE(\bfD j_\ast\M)\\
&\simeq
\pi^{-1}\CC_{X^\an}\otimes\Sol_{\tl{X}}^\rmE(\bfD j_\ast\M)\\
&\simeq
\Sol_{\tl{X}}^\rmE(\bfD j_\ast\M),
\end{align*}
where in the last isomorphism we used the fact 
$$\Sol_{\tl{X}}^\rmE(\bfD j_\ast\M)\in
\{K\in\BEC(\I\CC_{\tl{X}^\an})\ |\ \pi^{-1}\CC_{X^\an}\otimes K\simto K\}.$$
Hence, by the algebraic irregular Riemann--Hilbert correspondence
in the complete case (Theorem \ref{thm-main-1}),
we obtain 
$\Sol_{X_\infty}^{\rmE}(\M)\in\BEC_{\CC-c}(\I\CC_{X_\infty})$
for any $\M\in\BDChol(\D_X)$.
Moreover by Theorem \ref{thm-main-1} and
the fact that the direct image functor $\bfD j_\ast$
of the open embedding $j \colon X\hookrightarrow \tl{X}$ is fully faithful,
we obtain a fully faithful embedding
\[\BDChol(\D_X)\hookrightarrow
\{K\in\BEC_{\CC-c}(\I\CC_{\tl{X}^\an})\ |\ \pi^{-1}\CC_{X^\an}\otimes K\simto K\},
\hspace{20pt} \M\mapsto \Sol_{\tl{X}}^\rmE(\bfD j_\ast\M).\]
Hence the enhanced solution functor $\Sol_{X_\infty}^\rmE$ induces a fully faithful embedding
\[\Sol_{X_\infty}^{\rmE} \colon \BDChol(\D_X)^{\op}\hookrightarrow
 \BEC_{\CC-c}(\I\CC_{X_\infty})\]
 by the definition of 
 $\Sol_{X_\infty}^\rmE := \bfE j^{-1}\Sol_{\tl{X}}^\rmE\big(\bfD j_\ast(\cdot)\big)$.
 
 On the other hand, let $K\in\BEC_{\CC-c}(\I\CC_{X_\infty})$
 then \[\bfE j_{!!}K\in
 \{K\in\BEC_{\CC-c}(\I\CC_{\tl{X}^\an})\ |\ \pi^{-1}\CC_{X^\an}\otimes K\simto K\}\]
 by the definition of the algebraic $\CC$-constructability for enhanced ind-sheaves
 on $X_\infty^\an$.
Hence there exists an object $\N$ of $\BDChol(\D_{\tl{X}})$ such that 
 $\bfE j_{!!}K\simeq\Sol_{\tl{X}}^\rmE(\N)$
 by the algebraic irregular Riemann--Hilbert correspondence in the complete case (proved)
 in Theorem \ref{thm-main-1}.
 Moreover since $\pi^{-1}\CC_{X^\an}\otimes \bfE j_{!!}K\simto \bfE j_{!!}K$
 we have
 $$\Sol_{\tl{X}}^\rmE(\N)\simeq\pi^{-1}\CC_{X^\an}\otimes\Sol_{\tl{X}}^\rmE(\N)
 \simeq\Sol_{\tl{X}}^\rmE\big(\N(\ast D)\big)
 \simeq
 \Sol_{\tl{X}}^\rmE(\bfD j_\ast\bfD j^{\ast}\N).$$
We set $\M := \bfD j^{\ast}\N\in\BDChol(\D_X)$
then there exists an isomorphisms in $\BEC(\I\CC_{X^\an_\infty})$
\begin{align*}
\Sol_{X_\infty}^\rmE(\M)
&\simeq
\bfE j^{-1}\Sol_{\tl{X}}^\rmE(\bfD j_\ast\M)\\
&\simeq
\bfE j^{-1}\Sol_{\tl{X}}^\rmE(\N)\\
&\simeq
\bfE j^{-1}\bfE j_{!!}K \simeq K.
\end{align*}
Therefore the proof is completed.
\end{proof}

At the end of this subsection, we shall prove that 
the algebraic $\CC$-constructability is closed under many operations. 
This follows from some functorial properties of
the enhanced solution functor $\Sol_{X_\infty}^\rmE$ on $X_\infty$.
See \S \ref{sec-alg} for the notations of operations of $\D_X$-modules.  

\begin{proposition}\label{prop-comm}
\begin{enumerate}
\item[\rm{(1)}]
 For $\SM\in\BDC_{\rm hol}(\SD_X)$ there is an isomorphism in $\BEC(\I\CC_{X_\infty^\an})$
\[\rmD_{X_\infty^\an}^\rmE\Sol_{X_\infty}^\rmE(\SM)
\simeq
\Sol_{X_\infty}^\rmE(\DD_X\SM)[2d_X].\]

\item[\rm{(2)}] 
Let $f \colon X\to Y$ be a morphism of smooth algebraic varieties.
Then for $\SN\in\BDC_{\rm hol}(\SD_Y)$ there are isomorphisms in $\BEC(\I\CC_{X_\infty^\an})$
\begin{align*}
\bfE(f^\an_\infty)^{-1}\Sol_{Y_\infty}^\rmE(\SN)
&\simeq
\Sol_{X_\infty}^\rmE({\bfD} f^\ast\SN),\\
\bfE(f^\an_\infty)^{!}\Sol_{Y_\infty}^\rmE(\SN)
&\simeq
\Sol_{X_\infty}^\rmE({\bfD} f^\bigstar\SN).
\end{align*}

\item[\rm{(3)}] 
Let $f \colon X\to Y$ be a morphism of smooth algebraic varieties.
For $\SM\in\BDC_{\rm hol}(\SD_X)$ there are isomorphisms in $\BEC(\I\CC_{Y^\an_\infty})$
\begin{align*}
\bfE f^\an_{\infty\ast} \Sol_{X_\infty}^\rmE(\SM )[d_X]
&\simeq
\Sol_{Y_\infty}^\rmE({\bfD} f_!\SM)[d_Y],\\
\bfE f^\an_{\infty!!} \Sol_{X_\infty}^\rmE(\SM )[d_X]
&\simeq
\Sol_{Y_\infty}^\rmE({\bfD} f_\ast\SM)[d_Y].
\end{align*}

\item[\rm{(4)}]
For $\SM_1, \SM_2\in\BDC_{\rm hol}(\SD_X)$,
there exists an isomorphism in $\BEC(\I\CC_{X^\an_\infty})$
\[\Sol_{X_\infty}^\rmE(\SM_1)\Potimes \Sol_{X_\infty}^\rmE(\SM_2)
\simeq\Sol_{X_\infty}^\rmE(\SM_1\Dotimes\SM_2).\]

Moreover for any $\M\in\BDChol(\D_X)$ and any $\N\in\BDChol(\D_Y)$,
there exists an isomorphism in $\BEC(\I\CC_{X^\an_\infty\times Y^\an_\infty})$
\[\Sol_{X_\infty}^\rmE(\M)\Pboxtimes \Sol_{Y_\infty}^\rmE(\N)
\simeq\Sol_{X_\infty\times Y_\infty}^\rmE(\M\Dboxtimes\N).\]
\end{enumerate}
\end{proposition}

\begin{proof}
(1) Let us recall that there exists an isomorphism
$\DD_{\tl{X}}\bfD j_\ast\M\simeq\bfD j_!\DD_X\M$
in $\BDChol(\D_{\tl{X}})$.
Hence we have isomorphisms
\begin{align*}
\rmD_{X^\an_\infty}^\rmE\Sol_{X_\infty}^\rmE(\SM)
&\simeq
\rmD_{X^\an_\infty}^\rmE\bfE j^{-1}\Sol_{\tl{X}}^\rmE(\bfD j_\ast\M)\\
&\simeq
\bfE j^{-1}\rmD_{\tl{X}^\an}^\rmE\Sol_{\tl{X}}^\rmE(\bfD j_\ast\M)\\
&\simeq
\bfE j^{-1}\Sol_{\tl{X}}^\rmE(\DD_{\tl{X}}\bfD j_\ast\M)[2d_X]\\
&\simeq
\bfE j^{-1}\Sol_{\tl{X}}^\rmE(\bfD j_!\DD_{X}\M)[2d_X]\\
&\simeq
\Sol_{X_\infty}^\rmE(\DD_X\SM)[2d_X],
\end{align*}
where in the second (resp.\ third) isomorphism
we used Proposition \ref{prop2.8} (2)
(resp.\ Theorem \ref{thm-DK} (1) and Proposition \ref{prop2.12}).

(2) By Proposition \ref{prop2.8} (2) for $L\in\BEC_{\RR-c}(\I\CC_{Y^\an_\infty})$
we have an isomorphism in $\BEC_{\RR-c}(\I\CC_{X^\an_\infty})$
$$\bfE(f^\an_\infty)^!L\simeq
\rmD_{X^\an_\infty}^\rmE\bfE(f^\an_\infty)^{-1}\rmD_{Y^\an_\infty}^\rmE(L).$$
Hence it is enough to show that the first part of (2).
Let us recall that by Sublemma \ref{sublem2.2} (2)
for any $L\in\BEC(\I\CC_{Y^\an_\infty})$
there exists an isomorphism in $\BEC(\I\CC_{X^\an_\infty})$
\[\bfE(f^\an_\infty)^{-1}(L)
\simeq
\bfE j_X^{-1}\bfE (\tl{f}^\an)^{-1}\bfE j_{Y!!}(L)\]
where $\tl{f}^\an \colon \tl{X}^\an\to \tl{Y}^\an$ is a morphism of complex manifolds
induced by $f \colon X\to Y$ and
$j_X\colon X^\an_\infty \to \tl{X}^\an$ (resp.\ $j_Y\colon Y^\an_\infty \to \tl{Y}^\an$)
is a morphism of bordered spaces given
by the open embedding $X^\an\hookrightarrow \tl{X}^\an$
(resp.\ $Y^\an\hookrightarrow \tl{Y}^\an$).
Hence we have isomorphisms
\begin{align*}
\bfE(f^\an_\infty)^{-1}\Sol_{Y_\infty}^\rmE(\SN)
&\simeq
\bfE j_X^{-1}\bfE (\tl{f}^\an)^{-1}\bfE j_{Y!!}\Sol_{Y_\infty}^\rmE(\SN)\\
&\simeq
\bfE j_X^{-1}\bfE (\tl{f}^\an)^{-1}\Sol_{\tl{Y}}^\rmE(\bfD j_{Y\ast}\N)\\
&\simeq
\bfE j_X^{-1}\Sol_{\tl{X}}^\rmE(\bfD j_{X\ast}\bfD f^\ast\N)\\
&\simeq
\Sol_{X_\infty}^\rmE({\bfD} f^\ast\SN)
\end{align*}
where in the third isomorphism we used Theorem \ref{thm-DK} (2)
and Proposition \ref{prop2.12} (2).

(3) By Proposition \ref{prop2.8} (3) and
the fact that the morphism $f^\an_\infty \colon X^\an_\infty\to Y^\an_\infty$
of bordered spaces is semi-proper (see \S \ref{subsec2.2} for the definition),
for $K\in\BEC_{\RR-c}(\I\CC_{X^\an_\infty})$
we have an isomorphism in $\BEC(\I\CC_{Y^\an_\infty})$
$$\bfE f^\an_{\infty\ast}K\simeq
\rmD_{Y^\an_\infty}^\rmE\bfE f^\an_{\infty !!}\rmD_{X^\an_\infty}^\rmE(K).$$
Hence it is enough to prove that the second part of (3).
Let us recall that by  Sublemma \ref{sublem2.2} (1)
for any $K\in\BEC(\I\CC_{X^\an_\infty})$
there exists an isomorphism in $\BEC(\I\CC_{Y^\an_\infty})$
\[\bfE f^\an_{\infty !!}(K)
\simeq
\bfE j_Y^{-1}\bfE (\tl{f}^\an)_{!!}\bfE j_{X!!}(K),\]
where $\tl{f}^\an \colon \tl{X}^\an\to \tl{Y}^\an$ is a morphism of complex manifolds
induced by $f \colon X\to Y$and
$j_X\colon X^\an_\infty \to \tl{X}^\an$ (resp.\ $j_Y\colon Y^\an_\infty \to \tl{Y}^\an$)
is the morphism of bordered spaces given
by the open embedding $X^\an\hookrightarrow \tl{X}^\an$
(resp.\ $Y^\an\hookrightarrow \tl{Y}^\an$).
Hence we have isomorphisms in $\BEC(\I\CC_{Y^\an_\infty})$
\begin{align*}
\bfE f^\an_{\infty !!}\Sol_{X_\infty}^\rmE(\M)
&\simeq
\bfE j_Y^{-1}\bfE (\tl{f}^\an)_{!!}\bfE j_{X!!}\Sol_{X_\infty}^\rmE(\M)\\
&\simeq
\bfE j_Y^{-1}\bfE (\tl{f}^\an)_{!!}\Sol_{\tl{X}}^\rmE(\bfD j_{X\ast}\M)\\
&\simeq
\bfE j_Y^{-1}\Sol_{\tl{Y}}^\rmE(\bfD\tl{f}_\ast\bfD j_{X\ast}\M)[d_Y-d_X]\\
&\simeq
\bfE j_Y^{-1}\Sol_{\tl{Y}}^\rmE(\bfD j_{Y\ast}\bfD f_\ast\M)[d_Y-d_X]\\
&\simeq
\Sol_{Y_\infty}^\rmE({\bfD} f_\ast\M)[d_Y-d_X]
\end{align*}
where in the third isomorphism we used Theorem \ref{thm-DK} (3)
and Proposition \ref{prop2.12} (3).

(4) By (2) it is enough to show that the first part of (4).
Recall that there exists an isomorphism 
$\bfD j_\ast\M_1\Dotimes\bfD j_\ast\M_2
\simeq\bfD j_\ast(\M_1\Dotimes\M_2)$
in $\BDChol(\D_{\tl{X}})$ by \cite[Corollary 1.7.5]{HTT08}.
Hence we have isomorphisms in $\BEC(\I\CC_{X^\an_\infty})$
\begin{align*}
\Sol_{X_\infty}^\rmE(\M_1\Dotimes\M_2)
&\simeq
\bfE j^{-1}\Sol_{\tl{X}}^\rmE\big(\bfD j_\ast(\M_1\Dotimes\M_2)\big)\\
&\simeq
\bfE j^{-1}\Sol_{\tl{X}}^\rmE\big(\bfD j_\ast\M_1\Dotimes\bfD j_\ast\M_2\big)\\
&\simeq
\bfE j^{-1}\Sol_{\tl{X}}^\rmE(\bfD j_\ast\M_1))\Potimes
\bfE j^{-1}\Sol_{\tl{X}}^\rmE(\bfD j_\ast\M_2)\\
&\simeq \Sol_{X_\infty}^\rmE(\SM_1)
\Potimes \Sol_{X_\infty}^\rmE(\SM_2)
\end{align*}
where in the third isomorphism we used Theorem \ref{thm-DK} (2) and (4).
\end{proof}

\begin{corollary}\label{cor3.14}
Let $f \colon X\to Y$ be a morphism of smooth algebraic varieties and
$K, K_1, K_2\in\BEC_{\CC-c}(\I\CC_{X_\infty})$, $L\in\BEC_{\CC-c}(\I\CC_{Y_\infty})$.
Then we have
\begin{itemize}
\item[\rm(1)]
$\rmD_{X^\an_\infty}^{\rmE}(K)\in\BEC_{\CC-c}(\I\CC_{X_\infty})$ and 
$K\simto \rmD_{X^\an_\infty}^{\rmE}\rmD_{X^\an_\infty}^{\rmE}K$,

\item[\rm(2)]
$K_1\Potimes K_2$, $\Prihom(K_1, K_2)$ and $K\Pboxtimes L$ are algebraic $\CC$-constructible,

\item[\rm(3)]
$\bfE(f^\an_\infty)^{-1}L$ and $\bfE(f^\an_\infty)^!L$ are algebraic $\CC$-constructible,

\item[\rm(4)]
$\bfE f^\an_{\infty!!}K$ and $\bfE f^\an_{\infty\ast}K$ are algebraic $\CC$-constructible.
\end{itemize}
\end{corollary}

\begin{proof}
Since the proofs of these assertions in the corollary are similar,
we only prove (1).
By Theorem \ref{thm-main-2}, there exists $\M\in\BDChol(\D_X)$ such that
$K\simeq\Sol_{X_\infty}^\rmE(\M)$.
Then by Proposition \ref{prop-comm} (1) we obtain
\[\rmD_{X^\an_\infty}^\rmE(K)
\simeq
\rmD_{X^\an_\infty}^\rmE\Sol_{X_\infty}^\rmE(\SM)
\simeq
\Sol_{X_\infty}^\rmE(\DD_X\SM)[2d_X]
\simeq
\Sol_{X_\infty}^\rmE\big(\DD_X(\SM[2d_X])\big).\]
Hence by Theorem \ref{thm-main-2} and the fact that $\DD_X(\M[2d_X])\in\BDChol(\D_X)$
we have $\rmD_{X^\an_\infty}^\rmE(K)\in\BEC_{\CC-c}(\I\CC_{X_\infty})$.
Moreover we have the second part of (1) by Proposition \ref{prop2.8} and the fact that
any $\CC$-constructible enhanced ind-shaef is $\RR$-constructible.
\end{proof}

Let us recall that for any $\SF\in\BDC(\CC_{X^\an})$ we have
\[e_{X^\an_\infty}(\SF) \simeq \bfE j^{-1}(e_{\tl{X}^\an}(\bfR j^\an_!(\SF))),\]
see \S \ref{subsec2.4} for the details.
\begin{proposition}\label{prop-embed}
The functor $e_{X^\an_\infty} \colon \BDC(\CC_{X^\an}) \hookrightarrow \BEC(\I\CC_{X^\an_\infty})$
induces an embedding
$$\BDC_{\CC-c}(\CC_X)\hookrightarrow\BEC_{\CC-c}(\I\CC_{X_\infty})$$
and we have a commutative diagram
\[\xymatrix@C=50pt@M=5pt{
\BDChol(\D_X)^{\op}\ar@{->}[r]^\sim\ar@<1.0ex>@{}[r]^-{\Sol_{X_\infty}^{\rmE}}
\ar@{}[rd]|{\rotatebox[origin=c]{180}{$\circlearrowright$}}
 & \BEC_{\CC-c}(\I\CC_{X_\infty})\\
\BDCrh(\D_X)^{\op}\ar@{->}[r]_-{\Sol_X}^-{\sim}\ar@{}[u]|-{\bigcup}
&\BDC_{\CC-c}(\CC_X).
\ar@{^{(}->}[u]_-{e_{X^\an_\infty}}
}\]
\end{proposition}
\begin{proof}
It is enough to show the last part.
Let $\M$ be an object of $\BDCrh(\D_X)$.
Then $(\bfD j_\ast\M)^\an\in\BDCrh(\D_{\tl{X}^\an})$
by the definition of algebraic regular holonomic.
Hence we have isomorphisms in $\BEC(\I\CC_{\tl{X}^\an})$
\[\Sol_{\tl{X}}^\rmE(\bfD j_\ast\M)
\simeq
\Sol_{\tl{X}^\an}^\rmE\big((\bfD j_\ast\M)^\an\big)
\simeq
e_{\tl{X}^\an}\big(\Sol_{\tl{X}^\an}\big((\bfD j_\ast\M)^\an\big)\big)\]
by Theorem \ref{thm2.6} (2).
On the other hand, we have an isomorphism in $\BDC(\CC_{\tl{X}^\an})$
\[\Sol_{\tl{X}^\an}\big((\bfD j_\ast\M)^\an\big)
\simeq
\bfR j^\an_!\big(\Sol_{X^\an}(\M^\an)\big)
\simeq
\bfR j^\an_!\big(\Sol_{X}(\M)\big),\]
see e.g., \cite[Theorem 7.1.1]{HTT08}.
Hence there exist isomorphisms in $\BEC(\I\CC_{X^\an_\infty})$
\begin{align*}
\Sol_{X_\infty}^\rmE(\M)
&\simeq
\bfE j^{-1}\Sol_{\tl{X}}^\rmE(\bfD j_\ast\M)\\
&\simeq
\bfE j^{-1}e_{\tl{X}^\an}\big(\Sol_{\tl{X}^\an}\big((\bfD j_\ast\M)^\an\big)\\
&\simeq
\bfE j^{-1}e_{\tl{X}^\an}\bfR j^\an_!\big(\Sol_{X}(\M)\big)\\
&\simeq
e_{X^\an_\infty}\big(\Sol_{X}(\M)\big).
\end{align*}
\end{proof}

Let us recall that for any $K\in\BEC(\I\CC_{X^\an_\infty})$ we have
\[\sh_{X^\an_\infty}(K) \simeq (j^\an)^{-1}(\sh_{\tl{X}^\an}(\bfE j_{!!}(K))),\]
see \S \ref{subsec2.4} for the details.

\begin{lemma}\label{lem3.16}
For any $\M\in\BDChol(\D_X)$
there exists an isomorphism in $\BDC(\CC_{X^\an})$
\[\sh_{X^\an_\infty}(\Sol_{X_\infty}^\rmE(\M))
\simeq\Sol_{X}(\M).\]
\end{lemma}
\begin{proof}
By the definition of $\Sol_{X_\infty}^\rmE$ and $\sh_{X^\an_\infty}$
we have isomorphisms in $\BDC(\CC_{X^\an})$
\begin{align*}
\sh_{X^\an_\infty}(\Sol_{X_\infty}^\rmE(\M))
&\simeq
(j^\an)^{-1}\sh_{\tl{X}^\an}\bfE j_{!!}\bfE j^{-1}\Sol_{\tl{X}}^\rmE(\bfD j_\ast\M)\\
&\simeq
(j^\an)^{-1}\sh_{\tl{X}^\an}\Sol_{\tl{X}}^\rmE(\bfD j_\ast\M)\\
&\simeq
(j^\an)^{-1}\Sol_{\tl{X}}(\bfD j_\ast\M)\\
&\simeq
\Sol_{X}(\bfD j^\ast\bfD j_\ast\M)\\
&\simeq
\Sol_{X}(\M)
\end{align*}
where in the third isomorphism we used \cite[Lemma 9.5.5]{DK16}.
\end{proof}

\begin{proposition}\label{prop-sh}
The functor $\sh_{X^\an_\infty} \colon \BEC(\I\CC_{X^\an_\infty})\to\BDC(\CC_{X^\an})$ induces
$$\BEC_{\CC-c}(\I\CC_{X_\infty})\to\BDC_{\CC-c}(\CC_X)$$
and hence we have a commutative diagram
\[\xymatrix@C=50pt@M=5pt{
\BDChol(\D_X)^{\op}\ar@{->}[r]^\sim\ar@<1.0ex>@{}[r]^-{\Sol_{X_\infty}^{\rmE}}
\ar@{->}[d]_-{(\cdot)_\reg}\ar@{}[rd]|{\rotatebox[origin=c]{180}{$\circlearrowright$}}
 & \BEC_{\CC-c}(\I\CC_{X_\infty})\ar@{->}[d]^-{\sh_{X^\an_\infty}}\\
\BDCrh(\D_X)^{\op}\ar@{->}[r]_-{\Sol_X}^-{\sim}&\BDC_{\CC-c}(\CC_X).
}\]
\end{proposition}
\begin{proof}
The first part follows from Theorem \ref{thm-main-2} and Lemma \ref{lem3.16}.
Moreover since there exists an isomorphism $\Sol_X(\M)\simeq\Sol_X(\M_\reg)$
in $\BDC(\CC_{X^\an})$ for any $\M\in\BDChol(\D_X)$ (see \S \ref{sec-alg} for the details),
we obtain the commutative diagram.
\end{proof}

\subsection{Perverse t-Structure}
In this subsection, we define a t-structure
on the triangulated category $\BEC_{\CC-c}(\I\CC_{X_\infty})$
of algebraic $\CC$-constructible enhanced ind-sheaves on $X_\infty^\an$
and prove that its heart is equivalent to the abelian category $\Modhol(\D_X)$
of algebraic holonomic $\D_X$-modules.
The similar results of this section for the analytic case is proved in \cite[\S4]{Ito19}.

We denote by $\rmD_{X^\an} \colon\BDC(\CC_{X^\an})^{\op}\to\BDC(\CC_{X^\an})$
the Verdier dual functor for sheaves,
see \cite[\S 3]{KS90} for the definition.
The sheafification functor $\sh_{X^\an_\infty} \colon \BEC_{\CC-c}(\I\CC_{X_\infty})\to \BDC_{\CC-c}(\CC_{X})$ commutes
with the duality functor as follows.
\begin{lemma}\label{lem3.18}
For any $K\in\BEC_{\CC-c}(\I\CC_{X_\infty})$
there exists an isomorphism in $\BDC(\CC_{X^\an})$
$$\sh_{X^\an_\infty}\big(\rmD_{X^\an_\infty}^{\rmE}(K)\big)
\simeq\rmD_{X^\an}\big(\sh_{X^\an_\infty}(K)\big).$$
\end{lemma}

\begin{proof}
Let $K\in\BEC_{\CC-c}(\I\CC_{X_\infty})$.
Then there exists an object $\M\in\BDChol(\D_X)$ such that
$K\simeq\Sol_{X_\infty}^{\rmE}(\M)$
by the algebraic irregular Riemann--Hilbert correspondence (proved)
in Theorem \ref{thm-main-2}. 
Therefore we have isomorphisms in $\BDC(\CC_{X^\an})$
\begin{align*}
\sh_{X^\an_\infty}\big(\rmD_{X^\an_\infty}^{\rmE}(K)\big)
&\simeq
\sh_{X^\an_\infty}\big(\rmD_{X^\an_\infty}^{\rmE}(\Sol_{X_\infty}^{\rmE}(\M))\\
&\simeq
\sh_{X^\an_\infty}\big(\Sol_{X^\an_\infty}^{\rmE}\big(\DD_X\M\big)[2d_X]\big)\\
&\simeq
\Sol_X\big(\DD_X\M\big)[2d_X]\\
&\simeq
\rmD_{X^\an}\big(\Sol_X(\M)\big)\\
&\simeq
\rmD_{X^\an}\big(\sh_{X^\an_\infty}\Sol_{X_\infty}^\rmE(\M)\big)\\
&\simeq
\rmD_{X^\an}\big(\sh_{X^\an_\infty}(K)\big)
\end{align*}
where in the second isomorphism
we used Proposition \ref{prop-comm} (1),
in the third and fifth ones we used Lemma \ref{lem3.16}
and in the forth one we used the isomorphism
$\Sol_X\big(\DD_X(\M)\big)[2d_X]
\simeq
\rmD_{X^\an}\big(\Sol_X(\M)\big)$ for $\M\in\BDChol(\D_X)$
(see e.g., \cite[Proposition 4.7.9]{HTT08}).
\end{proof}

Let us recall the definition of algebraic perverse sheaves on $X^\an$.
We consider the following full subcategories of $\BDC_{\CC-c}(\CC_X)$:
\begin{align*}
{}^p\bfD^{\leq0}_{\CC-c}(\CC_X) & :=
\{\SF\in\BDC_{\CC-c}(\CC_X)\ |\
\dim(\supp\SH^i\SF)\leq -i \hspace{7pt}
\mbox{for any } i\in\ZZ \},\\
{}^p\bfD^{\geq0}_{\CC-c}(\CC_X) & :=
\{\SF\in\BDC_{\CC-c}(\CC_X)\ |\ \rmD_{X^\an}(\SF)\in{}^p\bfD^{\leq0}_{\CC-c}(\CC_X)\}\\
&\ =
\{\SF\in\BDC_{\CC-c}(\CC_X)\ |\
\dim(\supp\SH^i\rmD_{X^\an}(\SF))\leq -i \hspace{7pt}
\mbox{for any } i\in\ZZ \}.
\end{align*}
Then the pair $\big({}^p\bfD^{\leq0}_{\CC-c}(\CC_X), {}^p\bfD^{\geq0}_{\CC-c}(\CC_X)\big)$
is a t-structure on $\BDC_{\CC-c}(\CC_X)$ by \cite{BBD}, see also \cite[Theorem 8.1.27]{HTT08}.
We denote by $\Perv(\CC_X)$ the heart of its t-structure
and call an object of $\Perv(\CC_X)$ an algebraic perverse sheaf on $X^\an$.

\begin{definition}
We define full subcategories of $\BEC_{\CC-c}(\I\CC_{X_\infty})$ by
\begin{align*}
{}^p\bfE^{\leq0}_{\CC-c}(\I\CC_{X_\infty}) &:=
\{K\in\BEC_{\CC-c}(\I\CC_{X_\infty})\ |\
\sh_{X^\an_\infty}(K)\in{}^p\bfD^{\leq0}_{\CC-c}(\CC_X)\},\\
{}^p\bfE^{\geq0}_{\CC-c}(\I\CC_{X_\infty}) &:=
\{K\in\BEC_{\CC-c}(\I\CC_{X_\infty})\ |\ 
\rmD_{X^\an_\infty}^{\rmE}(K)\in{}^p\bfE^{\leq0}_{\CC-c}(\I\CC_{X_\infty})\}\\
&\ =
\{K\in\BEC_{\CC-c}(\I\CC_{X_\infty})\ |\
\sh_{X^\an_\infty}(K)\in{}^p\bfD^{\geq0}_{\CC-c}(\CC_X)\}
\hspace{20pt}(\mbox{ by Lemma \ref{lem3.18}}).
\end{align*}
\end{definition}

Then we obtain the second main theorem of this paper.
\begin{theorem}\label{thm-main-3}
Let $X$ be a smooth algebraic variety over $\CC$ and $\M\in\BDChol(\D_X)$.
Then we have
\begin{itemize}
\item[\rm(1)]
$\M\in\DChol^{\leq0}(\D_X)\Longleftrightarrow
\Sol_{X_\infty}^{\rmE}(\M)[d_X]\in{}^p\bfE^{\geq0}_{\CC-c}(\I\CC_{X_\infty})$,
\item[\rm(2)]
$\M\in\DChol^{\geq0}(\D_X)\Longleftrightarrow
\Sol_{X_\infty}^{\rmE}(\M)[d_X]\in{}^p\bfE^{\leq0}_{\CC-c}(\I\CC_{X_\infty})$.
\end{itemize}

Therefore, the pair $\big({}^p\bfE^{\leq0}_{\CC-c}(\I\CC_{X_\infty}),
{}^p\bfE^{\geq0}_{\CC-c}(\I\CC_{X_\infty})\big)$
is a t-structure on $\BEC_{\CC-c}(\I\CC_{X_\infty})$
and
its heart $$\Perv(\I\CC_{X_\infty}) :=
{}^p\bfE^{\leq0}_{\CC-c}(\I\CC_{X_\infty})\cap{}^p\bfE^{\geq0}_{\CC-c}(\I\CC_{X_\infty})$$
is equivalent to the abelian category $\Modhol(\D_X)$ of holonomic $\D_X$-modules on $X$.
Namely we obtain an equivalence of abelian categories
\[\Sol_{X_\infty}^\rmE(\cdot)[d_X] \colon \Modhol(\D_X)^{\op}\simto\Perv(\I\CC_{X_\infty}).\]
\end{theorem}
\begin{proof}
It is enough to show the first part.
Let $\M$ be an object of $\DChol^{\leq0}(\D_X)$.
Then we have $\Sol_X(\M)[d_X]\in{}^p\bfD^{\geq0}_{\CC-c}(\CC_X)$
by Theorem \ref{thm-perv} (1).
Furthermore by Lemma \ref{lem3.16}
there exists an isomorphism in $\BDC(\CC_{X^\an})$
\[\sh_{X^\an_\infty}(\Sol_{X_\infty}^\rmE(\M))
\simeq\Sol_{X}(\M).\]
Hence we obtain
\[\Sol_{X_\infty}^{\rmE}(\M)[d_X]\in{}^p\bfE^{\geq0}_{\CC-c}(\I\CC_{X_\infty}).\]
On the other hand, let $\M$ be an object of $\BDChol(\D_X)$ which satisfies
$\Sol_{X_\infty}^{\rmE}(\M)[d_X]\in{}^p\bfE^{\geq0}_{\CC-c}(\I\CC_{X_\infty})$.
Then we have 
\[\Sol_X(\M)[d_X]\simeq\sh_{X^\an_\infty}\Sol_{X_\infty}^{\rmE}(\M)[d_X]
\in{}^p\bfD^{\geq0}_{\CC-c}(\CC_X)\]
and hence $\M\in\DChol^{\leq0}(\D_X)$ by Theorem \ref{thm-perv} (1).
\end{proof}

\begin{definition}
We say that
$K\in\BEC_{\CC-c}(\I\CC_{X_\infty^\an})$ is an algebraic enhanced perverse ind-sheaf
on $X^\an_\infty$ if $K\in\Perv(\I\CC_{X_\infty}) :=
{}^p\bfE^{\leq0}_{\CC-c}(\I\CC_{X_\infty})\cap{}^p\bfE^{\geq0}_{\CC-c}(\I\CC_{X_\infty})$.
\end{definition}

By the definition of the t-structure
$\big({}^p\bfE^{\leq0}_{\CC-c}(\I\CC_{X_\infty}),
{}^p\bfE^{\geq0}_{\CC-c}(\I\CC_{X_\infty})\big)$,
we have
\begin{align*}
\rmD_{X^\an_\infty}^\rmE(K)\in{}^p\bfE^{\geq0}_{\CC-c}(\I\CC_{X_\infty}) &
\mbox{ for } K\in{}^p\bfE^{\leq0}_{\CC-c}(\I\CC_{X_\infty}),\\
\rmD_{X^\an_\infty}^\rmE(K)\in{}^p\bfE^{\leq0}_{\CC-c}(\I\CC_{X_\infty}) &
\mbox{ for } K\in{}^p\bfE^{\geq0}_{\CC-c}(\I\CC_{X_\infty}).
\end{align*}

Thus, the duality functor $\rmD_{X^\an_\infty}^\rmE$ induces an equivalence of abelian categories
$$\rmD_{X^\an_\infty}^{\rmE} \colon
\Perv(\I\CC_{X_\infty})^{\op}\simto\Perv(\I\CC_{X_\infty}).$$
By $\id\simto\sh_{X^\an_\infty} \circ e_{X^\an_\infty}$ and Proposition \ref{prop-embed}
we obtain:
\begin{proposition}\label{prop3.22}
\begin{itemize}
\item[\rm(1)]
For any $\SF\in{}^p\bfD_{\CC-c}^{\leq0}(\CC_X)$,
we have $e_{X^\an_\infty}(\SF)\in{}^p\bfE_{\CC-c}^{\leq0}(\I\CC_{X_\infty})$.
\item[\rm(2)]
For any $\SF\in{}^p\bfD_{\CC-c}^{\geq0}(\CC_X)$,
we have $e_{X^\an_\infty}(\SF)\in{}^p\bfE_{\CC-c}^{\geq0}(\I\CC_{X_\infty})$.
\item[\rm(3)] 
The functor $e_{X^\an_\infty} \colon \BDC_{\CC-c}(\CC_X)\hookrightarrow\BEC_{\CC-c}(\I\CC_{X_\infty})$
induces an embedding
\[\Perv(\CC_X)\hookrightarrow\Perv(\I\CC_{X_\infty}).\]
Moreover we obtain a commutative diagram
\[\xymatrix@C=60pt@M=5pt{
\Modhol(\D_X)^{\op}\ar@{->}[r]^\sim\ar@<1.0ex>@{}[r]^-{\Sol_{X_\infty}^{\rmE}(\cdot)[d_X]}
\ar@{}[rd]|{\rotatebox[origin=c]{180}{$\circlearrowright$}}
 & \Perv(\I\CC_{X_\infty})\\
\Modrh(\D_X)^{\op}\ar@{->}[r]_-{\Sol_X(\cdot)[d_X]}^-{\sim}\ar@{}[u]|-{\bigcup}
&\Perv(\CC_X).
\ar@{^{(}->}[u]_-{e_{X^\an_\infty}}
}\]
\end{itemize}
\end{proposition}

By the definition of the t-structure
$\big({}^p\bfE^{\leq0}_{\CC-c}(\I\CC_{X_\infty}),
{}^p\bfE^{\geq0}_{\CC-c}(\I\CC_{X_\infty})\big)$,
the sheafification functor $\sh_{X^\an_\infty}$ induces a functor
\[\Perv(\I\CC_{X_\infty})\to\Perv(\CC_{X}).\]
Moreover by Proposition \ref{prop-sh} we obtain a commutative diagram
\[\xymatrix@C=60pt@M=5pt{
\Modhol(\D_X)^{\op}\ar@{->}[r]^\sim\ar@<1.0ex>@{}[r]^-{\Sol_{X_\infty}^{\rmE}(\cdot)[d_X]}
\ar@{->}[d]_-{(\cdot)_\reg}\ar@{}[rd]|{\rotatebox[origin=c]{180}{$\circlearrowright$}}
 & \Perv(\I\CC_{X_\infty})\ar@{->}[d]^-{\sh_{X^\an_\infty}}\\
\Modrh(\D_X)^{\op}\ar@{->}[r]_-{\Sol_X(\cdot)[d_X]}^-{\sim}&\Perv(\CC_X).
}\]

Recall that there exists a generalized t-structure
$\big({}^{\frac{1}{2}}\bfE_{\RR-c}^{\leq c}(\I\CC_{X^\an_\infty}),
{}^{\frac{1}{2}}\bfE_{\RR-c}^{\geq c}(\I\CC_{X^\an_\infty})\big)_{c\in\RR}$
on $\BEC_{\RR-c}(\I\CC_{X^\an_\infty})$ by \cite[Theorem 3.5.2 (i)]{DK16-2}.
Then the pair $\big({}^p\bfE^{\leq0}_{\CC-c}(\I\CC_{X_\infty}),
{}^p\bfE^{\geq0}_{\CC-c}(\I\CC_{X_\infty})\big)$
is related to this as follows:

\begin{proposition}\label{prop3.23}
We have
\begin{align*}
{}^p\bfE_{\CC-c}^{\leq 0}(\I\CC_{X_\infty}) &=
{}^{\frac{1}{2}}\bfE_{\RR-c}^{\leq 0}(\I\CC_{X^\an_\infty})\cap
\BEC_{\CC-c}(\I\CC_{X_\infty}),\\
{}^p\bfE_{\CC-c}^{\geq 0}(\I\CC_{X_\infty})
&={}^{\frac{1}{2}}\bfE_{\RR-c}^{\geq 0}(\I\CC_{X^\an_\infty})\cap
\BEC_{\CC-c}(\I\CC_{X_\infty}).
\end{align*}

\end{proposition}

\begin{proof}
By Corollary \ref{cor3.14} (1) and the facts
\begin{align*}
{}^p\bfE^{\geq0}_{\CC-c}(\I\CC_{X_\infty}) &=
\{K\in\BEC_{\CC-c}(\I\CC_{X_\infty})\ |\ 
\rmD_{X^\an_\infty}^{\rmE}(K)\in{}^p\bfE^{\leq0}_{\CC-c}(\I\CC_{X_\infty})\},\\
{}^{\frac{1}{2}}\bfE^{\geq0}_{\RR-c}(\I\CC_{X^\an_\infty}) &=
\{K\in\BEC_{\RR-c}(\I\CC_{X^\an_\infty})\ |\ 
\rmD_{X^\an_\infty}^{\rmE}(K)\in{}^{\frac{1}{2}}\bfE^{\leq0}_{\RR-c}(\I\CC_{X^\an_\infty})\},
\end{align*}
it is enough to show the first part.
Let us recall that for any $K\in\BEC_{\RR-c}(\I\CC_{X^\an_\infty})$, we have
\[
K\in{}^{\frac{1}{2}}\bfE_{\RR-c}^{\leq0}(\I\CC_{X^\an_\infty})
\Longleftrightarrow
\bfE j_{X^\an_\infty!!}K\in{}^{\frac{1}{2}}\bfE_{\RR-c}^{\leq0}(\I\CC_{\tl{X}^\an})
\]
by \cite[Lemma 3.3.2, Proposition 3.5.6 (i), (iv)]{DK16-2}.
Hence the first part follows from Lemma \ref{lem3.24} (1) and Sublemma \ref{sublem3.25} (1) below.
\end{proof}

\begin{lemma}\label{lem3.24}
For any $K\in\BEC_{\CC-c}(\I\CC_{X_\infty})$,
we have
\begin{itemize}
\item[\rm(1)]
$K\in{}^p\bfE_{\CC-c}^{\leq0}(\I\CC_{X_\infty})
\Longleftrightarrow
\bfE j_{X^\an_\infty!!}K\in{}^p\bfE_{\CC-c}^{\leq0}(\I\CC_{\tl{X}})$,
\item[\rm(2)]
$K\in{}^p\bfE_{\CC-c}^{\geq0}(\I\CC_{X_\infty})
\Longleftrightarrow
\bfE j_{X^\an_\infty!!}K\in{}^p\bfE_{\CC-c}^{\geq0}(\I\CC_{\tl{X}})$.
\end{itemize}
\end{lemma}

\begin{proof}
Since the proof of (2) is similar, we only prove (1).

First, we assume $K\in{}^p\bfE_{\CC-c}^{\leq0}(\I\CC_{X_\infty})$.
Then there exists an object $\M\in \DChol^{\geq0}(\D_X)$ such that $K\simeq\Sol_{X_\infty}^\rmE(\M)[d_X]$
by Theorems \ref{thm-main-2} and \ref{thm-main-3} (2).
By the definition of the functor $\Sol_{X_\infty}^\rmE$,
we have $\bfE j_{X^\an_\infty!!}K\simeq\Sol_{\tl{X}}^\rmE(\bfD j_\ast\M)[d_X]$.
Furthremore since the canonical embedding $j\colon X\hookrightarrow\tl{X}$ is affine
we have $\bfD j_\ast\M\in \DChol^{\geq0}(\D_{\tl{X}})$,
and hence we have
\[\bfE j_{X^\an_\infty!!}K\simeq\Sol_{\tl{X}}^\rmE(\bfD j_\ast\M)[d_X]
\in{}^p\bfE_{\CC-c}^{\leq0}(\I\CC_{\tl{X}})\]
by Theorem \ref{thm-main-3} (2).

We assume $\bfE j_{X^\an_\infty!!}K\in{}^p\bfE_{\CC-c}^{\leq0}(\I\CC_{\tl{X}})$.
By the definition of the full subcategory ${}^p\bfE_{\CC-c}^{\leq0}(\I\CC_{\tl{X}})$,
we obtain $\sh_{\tl{X}^\an}\big(\bfE j_{X^\an_\infty!!}K\big)\in{}^p\bfD_{\CC-c}^{\leq0}(\CC_{\tl{X}})$.
Since the functor $$(j^\an)^{-1}\colon \BDC_{\CC-c}(\CC_{\tl{X}})\to\BDC_{\CC-c}(\CC_X)$$
is t-exact with respect to the perverse t-structures,
we have $$(j^\an)^{-1}\big(\sh_{\tl{X}^\an}\big(\bfE j_{X^\an_\infty!!}K\big)\big)
\in{}^p\bfD_{\CC-c}^{\leq0}(\CC_{X}).$$
Recall that for any $K\in\BEC(\I\CC_{X^\an_\infty})$ we have
\[\sh_{X^\an_\infty}(K) \simeq (j^\an)^{-1}(\sh_{\tl{X}^\an}(\bfE j_{!!}(K))).\]
Therefore $\sh_{X^\an_\infty}(K)\in{}^p\bfD_{\CC-c}^{\leq0}(\CC_X)$
and hence $K\in{}^p\bfE_{\CC-c}^{\leq0}(\I\CC_{X_\infty})$.
\end{proof}

Let us recall that the triangulated category $\BEC_{\CC-c}(\I\CC_{X})$
is the full triangulated subcategory of $\BEC_{\CC-c}(\I\CC_{X^\an})$,
see \S\S \ref{sec2.8} and \ref{sec-AC} for the details.
\begin{sublemma}\label{sublem3.25}
Let $X$ be a smooth complete algebraic variety over $\CC$.
Then we have
\begin{itemize}
\item[\rm(1)]
${}^p\bfE_{\CC-c}^{\leq 0}(\I\CC_{X}) = 
{}^p\bfE_{\CC-c}^{\leq 0}(\I\CC_{X^\an})\cap
\BEC_{\CC-c}(\I\CC_{X}) =
{}^{\frac{1}{2}}\bfE_{\RR-c}^{\leq 0}(\I\CC_{X^\an})\cap
\BEC_{\CC-c}(\I\CC_{X})$,
\item[\rm(2)]
${}^p\bfE_{\CC-c}^{\geq 0}(\I\CC_{X})= 
{}^p\bfE_{\CC-c}^{\geq 0}(\I\CC_{X^\an})\cap
\BEC_{\CC-c}(\I\CC_{X}) =
{}^{\frac{1}{2}}\bfE_{\RR-c}^{\geq 0}(\I\CC_{X^\an})\cap
\BEC_{\CC-c}(\I\CC_{X})$.
\end{itemize}
\end{sublemma}

\begin{proof}
Since the proof of (2) is similar, we only prove (1).
Let us recall that
\[{}^p\bfE_{\CC-c}^{\leq 0}(\I\CC_{X^\an}) =
{}^{\frac{1}{2}}\bfE_{\RR-c}^{\leq 0}(\I\CC_{X^\an})\cap
\BEC_{\CC-c}(\I\CC_{X^\an}),\]
see \S \ref{sec2.8} or \cite[Corollary 4.5]{Ito19} for the details.
Hence it is enough to prove
$${}^p\bfE_{\CC-c}^{\leq 0}(\I\CC_{X}) = 
{}^p\bfE_{\CC-c}^{\leq 0}(\I\CC_{X^\an})\cap
\BEC_{\CC-c}(\I\CC_{X}).$$

Let $K$ be an object of ${}^p\bfE_{\CC-c}^{\leq 0}(\I\CC_{X})$.
By Proposition \ref{prop3.5} and Theorem \ref{thm-main-3} (2),
there exists an object $\M$ of $\DChol^{\geq0}(\D_X)$
such that $$K\simeq \Sol_X^\rmE(\M)[d_X] \big(:= \Sol_{X^\an}^\rmE(\M^\an)[d_X]\big).$$
Since the analytification functor $(\cdot)^\an\colon\Mod(\D_X)\to\Mod(\D_{X^\an})$ is exact
we have $\M^\an\in\DChol^{\geq0}(\D_{X^\an})$,
and hence we have $\Sol_{X^\an}^\rmE(\M^\an)[d_X]\in{}^p\bfE_{\CC-c}^{\leq 0}(\I\CC_{X^\an})$
by Theorem \ref{thm2.28} (2).
Therefore we obtain $$K\in{}^p\bfE_{\CC-c}^{\leq 0}(\I\CC_{X^\an})\cap
\BEC_{\CC-c}(\I\CC_{X}).$$

Let $K$ be an object of ${}^p\bfE_{\CC-c}^{\leq 0}(\I\CC_{X^\an})\cap\BEC_{\CC-c}(\I\CC_{X})$.
Since $K\in\BEC_{\CC-c}(\I\CC_{X})$ there exists an object $\M\in\BDChol(\D_X)$ such that 
$$K\simeq\Sol_X^\rmE(\M)[d_X] \big(:= \Sol_{X^\an}^\rmE(\M^\an)[d_X]\big)$$
by Proposition \ref{prop3.5}.
Since $K\in{}^p\bfE_{\CC-c}^{\leq 0}(\I\CC_{X^\an})$
we have $\M^\an\in\DChol^{\geq0}(\D_{X^\an})$
by Theorem \ref{thm2.28} (2),
and hence we obtain $\M\in\DChol^{\geq0}(\D_{X})$
because the analytification functor $(\cdot)^\an\colon\Mod(\D_X)\to\Mod(\D_{X^\an})$ is exact and faithful.
Therefore we have $$K\simeq\Sol_X^\rmE(\M)[d_X]\in{}^p\bfE_{\CC-c}^{\leq 0}(\I\CC_{X})$$
by Theorem \ref{thm-main-3} (2).
\end{proof}

Thanks to \cite[Proposition 3.5.6]{DK16-2},
Proposition \ref{prop3.26} follows from Corollary \ref{cor3.14} (3), (4) and Proposition \ref{prop3.23}.

\begin{proposition}\label{prop3.26}
Let $f\colon X\to Y$ be a morphism of smooth algebraic varieties.
We assume that there exists a non-negative integer $d\in \ZZ_{\geq0}$
such that $\dim f^{-1}(y)\leq d$ for any $y\in Y$.
\begin{itemize}
\item[\rm(1)]
For any $K\in{}^p\bfE_{\CC-c}^{\leq0}(\I\CC_{X_\infty})$
we have $\bfE f^\an_{\infty!!} K\in{}^p\bfE_{\CC-c}^{\leq d}(\I\CC_{Y_\infty})$.
\item[\rm(2)]
For any $K\in{}^p\bfE_{\CC-c}^{\geq0}(\I\CC_{X_\infty})$
we have $\bfE f^\an_{\infty\ast} K\in{}^p\bfE_{\CC-c}^{\geq -d}(\I\CC_{Y_\infty})$.
\item[\rm(3)]
For any $L\in{}^p\bfE_{\CC-c}^{\leq0}(\I\CC_{Y_\infty})$
we have $\bfE (f_\infty^\an)^{-1} L\in{}^p\bfE_{\CC-c}^{\leq d}(\I\CC_{X_\infty})$.
\item[\rm(4)]
For any $L\in{}^p\bfE_{\CC-c}^{\geq0}(\I\CC_{Y_\infty})$
we have $\bfE (f_\infty^\an)^! L\in{}^p\bfE_{\CC-c}^{\geq -d}(\I\CC_{X_\infty})$.
\end{itemize}
\end{proposition}

\begin{proof}
Since the proofs of these assertions in the proposition is similar,
we only prove the assertion (1).

Let $K$ be an object of ${}^p\bfE_{\CC-c}^{\leq0}(\I\CC_{X_\infty})$.
Then we have 
$K\in{}^{\frac{1}{2}}\bfE_{\RR-c}^{\leq 0}(\I\CC_{X^\an_\infty})\cap\BEC_{\CC-c}(\I\CC_{X_\infty})$
by Proposition \ref{prop3.23}.
By \cite[Proposition 3.5.6]{DK16-2} (iv) we have
$\bfE f^\an_{\infty!!} K\in
{}^{\frac{1}{2}}\bfE_{\RR-c}^{\leq d}(\I\CC_{Y^\an_\infty})$
and by Corollary \ref{cor3.14} (4)
we have $\bfE f^\an_{\infty!!} K\in\BEC_{\CC-c}(\I\CC_{Y_\infty})$.
Hence we have 
\[\bfE f^\an_{\infty!!} K\in
{}^{\frac{1}{2}}\bfE_{\RR-c}^{\leq d}(\I\CC_{Y^\an_\infty})\cap\BEC_{\CC-c}(\I\CC_{Y_\infty})
= {}^p\bfE_{\CC-c}^{\leq d}(\I\CC_{Y_\infty}).\]
\end{proof}

\begin{corollary}\label{cor3.27}
Let $X$ be a smooth algebraic variety over $\CC$ and $Z$ a locally closed smooth subvariety of $X$.
We denote by $i_{Z^\an_\infty}\colon Z^\an_\infty\to X^\an_\infty$
the morphism of bordered spaces induced by the natural embedding $Z\hookrightarrow X$.
\begin{itemize}
\item[\rm(1)]
For any $K\in{}^p\bfE_{\CC-c}^{\leq0}(\I\CC_{Z_\infty})$
we have $\bfE i_{Z^\an_\infty!!} K\in{}^p\bfE_{\CC-c}^{\leq 0}(\I\CC_{X_\infty})$.
\item[\rm(2)]
For any $K\in{}^p\bfE_{\CC-c}^{\geq0}(\I\CC_{Z_\infty})$
we have $\bfE i_{Z^\an_\infty\ast} K\in{}^p\bfE_{\CC-c}^{\geq 0}(\I\CC_{X_\infty})$.
\item[\rm(3)]
For any $L\in{}^p\bfE_{\CC-c}^{\leq0}(\I\CC_{X_\infty})$
we have $\bfE i_{Z^\an_\infty}^{-1} L\in{}^p\bfE_{\CC-c}^{\leq 0}(\I\CC_{Z_\infty})$.
\item[\rm(4)]
For any $L\in{}^p\bfE_{\CC-c}^{\geq0}(\I\CC_{X_\infty})$
we have $\bfE i_{Z^\an_\infty}^! L\in{}^p\bfE_{\CC-c}^{\geq 0}(\I\CC_{Z_\infty})$.
\end{itemize}

In particular, if $Z$ is open $($resp.\ closed$)$,
then the functor 
$$\bfE i_{Z^\an_\infty}^{-1}\simeq\bfE i_{Z^\an_\infty}^{!}\hspace{10pt}
(\mbox{resp.}\ \bfE i_{Z^\an_\infty!!}\simeq \bfE i_{Z^\an_\infty\ast})$$
is t-exact with respect to the perverse t-structures.
\end{corollary}

\begin{remark}\label{rem3.28}
Let $X$ be a smooth algebraic variety over $\CC$ and $Z$ a locally closed smooth subvariety of $X$.
We assume that the natural embedding $i_Z\colon Z\hookrightarrow X$ is affine.

Then for a holonomic $\D_Z$-module $\M$,
we have $\SH^i(\bfD i_{Z\ast}\M) = \SH^i(\bfD i_{Z!}\M) = 0$ for any $i\neq0$.
Moreover we obtain exact functors
\[\bfD i_{Z\ast}, \bfD i_{Z!} \colon \Modhol(\D_Z)\to\Modhol(\D_X).\]

Hence by Proposition \ref{prop-comm} (3) and Theorem \ref{thm-main-3},
we obtain exact functors
\[\bfE i_{Z^\an_\infty\ast}, \bfE i_{Z^\an_\infty!!} \colon \Perv(\I\CC_{Z_\infty})\to\Perv(\I\CC_{X_\infty}).\]
Note that there exists a canonical morphism of functors
$\Perv(\I\CC_{Z_\infty})\to\Perv(\I\CC_{X_\infty})$
\[\bfE i_{Z^\an_\infty!!}\to\bfE i_{Z^\an_\infty\ast}\]
and it is an isomorphism if $Z$ is closed. 
\end{remark}

\begin{notation}
For a functor $\mathscr{F}\colon\BEC_{\CC-c}(\I\CC_{X_\infty})\to\BEC_{\CC-c}(\I\CC_{Y_\infty})$,
we set $${}^p\mathscr{F} := {}^p\SH^0\circ \mathscr{F} \colon\Perv(\I\CC_{X_\infty})\to\Perv(\I\CC_{Y_\infty}),$$
where ${}^p\SH^0$ is the $0$-th cohomology functor with respect to the perverse t-structures.

For example, for a morphism $f\colon X\to Y$ of smooth algebraic varieties,
we set
\begin{align*}
{}^p\bfE (f^\an_{\infty})^{-1} &:= {}^p\SH^0\circ\bfE (f^\an_{\infty})^{-1}\colon
\Perv(\I\CC_{Y_\infty})\to\Perv(\I\CC_{X_\infty}),\\
{}^p\bfE (f^\an_{\infty})^{!} &:= {}^p\SH^0\circ\bfE (f^\an_{\infty})^{!}\colon
\Perv(\I\CC_{Y_\infty})\to\Perv(\I\CC_{X_\infty}),\\
{}^p\bfE f^\an_{\infty\ast}&:= {}^p\SH^0\circ\bfE f^\an_{\infty\ast}\colon
\Perv(\I\CC_{X_\infty})\to\Perv(\I\CC_{Y_\infty}),\\
{}^p\bfE f^\an_{\infty!!}&:= {}^p\SH^0\circ\bfE f^\an_{\infty!!}\colon
\Perv(\I\CC_{X_\infty})\to\Perv(\I\CC_{Y_\infty}).
\end{align*} 
\end{notation}

In this paper, for an object $K\in\BEC(\I\CC_{X_\infty^\an})$,
let us define the support of $K$
by the complement of the union of open subsets $U^\an$ of $X^\an$ such that
$K|_{U_\infty^\an} := \bfE i_{U^\an_\infty}^{-1}K\simeq0$
and denote it by $\supp(K)$.
Namely, we set
$$\supp(K) := \Big(\bigcup_{U^\an\underset{\mbox{\tiny open}}{\subset} X^\an,\ K|_{U^\an_\infty} = 0}U^\an\Big)^c \hspace{7pt}
\subset X^\an.$$
Note that we have 
\[\bigcup_{U^\an\underset{\mbox{\tiny open}}{\subset} X^\an,\ K|_{U^\an_\infty} = 0}U^\an
 \hspace{7pt} = \hspace{7pt}
 \bigcup_{V^\an\underset{\mbox{\tiny open}}{\subset} X^\an,\ K|_{V^\an} = 0}V^\an.\]
Moreover, for a closed smooth subvariety $Z$ of $X$, we set 
$$\Perv_{Z}(\I\CC_{X_\infty}) :=  \{K\in\Perv(\I\CC_{X_\infty})\ |\ \supp(K)\subset Z^\an\}.$$

\begin{proposition}\label{prop3.29}
Let $X$ be a smooth algebraic variety over $\CC$ and $Z$ a closed smooth subvariety of $X$.
Then we have an equivalence of abelian categories$:$
\[\xymatrix@C=75pt{
\Perv_{Z}(\I\CC_{X_\infty})\ar@<0.7ex>@{->}[r]^-{{}^p\bfE i_{Z^\an_\infty}^{-1}}\ar@{}[r]|-{\sim}
&
\Perv(\I\CC_{Z_\infty})
\ar@<0.7ex>@{->}[l]^-{\bfE i_{Z^\an_\infty!!}}.
}\]
Furthermore for any $K\in\Perv_Z(\I\CC_{X_\infty})$
there exists an isomorphism in $\Perv(\I\CC_{Z_\infty})$
$${}^p\bfE i_{Z^\an_\infty}^{-1}K\simeq{}^p\bfE i_{Z^\an_\infty}^{!}K.$$

\end{proposition}

\begin{proof}
Let $L$ be an object of $\Perv(\I\CC_{Z_\infty})$.
Recall that the functor $\bfE i_{Z^\an_\infty!!}$ is t-exact with respect to the perverse t-structures
by Corollary \ref{cor3.27}, and hence we have $\bfE i_{Z^\an_\infty!!}L\in\Perv(\I\CC_{X_\infty})$.
Furthermore since $\bfE i_{X^\an_\infty\setminus Z^\an_\infty}^{-1}\bfE i_{Z^\an_\infty!!}L\simeq 0$,
we obtain $\supp(\bfE i_{Z^\an_\infty!!}L)\subset Z^\an$.
This implies that $$\bfE i_{Z^\an_\infty!!}L\in\Perv_{Z}(\I\CC_{X_\infty}).$$
Note that there exists an isomorphism
$\bfE i_{Z^\an_\infty}^{-1}\bfE i_{Z^\an_\infty!!}L\simto L$ in $\BEC(\I\CC_{Z_\infty^\an})$.
Thus we have $${}^p\bfE i_{Z^\an_\infty}^{-1}\bfE i_{Z^\an_\infty!!}L\simto L.$$

Let $K$ be an object of $\Perv_Z(\I\CC_{X_\infty})$.
Then we have $$\pi^{-1}\CC_{(X\setminus Z)^\an}\otimes K\simeq
\bfE i_{(X\setminus Z)^\an_\infty!!}\bfE i_{(X\setminus Z)^\an_\infty}^{-1}K\simeq0,$$
where in the first isomorphism we used \cite[Lemma 2.7.6]{DK16-2}.
Thus we obtain $K\simto\pi^{-1}\CC_{Z^\an}\otimes K$
because there exists a distinguished triangle
$$\pi^{-1}\CC_{X^\an\setminus Z^\an}\otimes K\to K\to\pi^{-1}\CC_{Z^\an}\otimes K\xrightarrow{+1}$$
in $\BEC(\I\CC_{X_\infty^\an})$.
Furthermore there exists an isomorphism
$\bfE i_{Z^\an_\infty!!}\bfE i_{Z^\an_\infty}^{-1}K\simeq\pi^{-1}\CC_{Z^\an}\otimes K$
in $\BEC(\I\CC_{X^\an_\infty})$ by \cite[Lemma 2.7.6]{DK16-2},
and hence we have $K\simto\bfE i_{Z^\an_\infty!!}\bfE i_{Z^\an_\infty}^{-1}K$.
Since the functor $\bfE i_{Z^\an_\infty!!}$ is t-exact with respect to the perverse t-structures,
we obtain $$K\simto\bfE i_{Z^\an_\infty!!}{}^p\bfE i_{Z^\an_\infty}^{-1}K.$$
Therefore the proof of the first part is completed.

For any $K\in\Perv_Z(\I\CC_{X_\infty})$,
there exists an object $L\in\Perv(\I\CC_{Z_\infty})$ such that
$K\simeq\bfE i_{Z^\an_\infty!!}L\simeq\bfE i_{Z^\an_\infty\ast}L$.
Hence we have isomorphisms in $\Perv(\I\CC_{Z_\infty})$:
$${}^p\bfE i_{Z^\an_\infty}^{-1}K
\simeq
{}^p\bfE i_{Z^\an_\infty}^{-1}\bfE i_{Z^\an_\infty\ast}L
\simeq L
\simeq
{}^p\bfE i_{Z^\an_\infty}^{!}\bfE i_{Z^\an_\infty!!}L
\simeq{}^p\bfE i_{Z^\an_\infty}^{!}K.$$
Hence the proof of the second part is completed.
\end{proof}

\begin{remark}\label{rem3.31}
By using Theorem \ref{thm-main-3},
Proposition \ref{prop3.29} also follows from the Kashiwara's equivalence,
that is an equivalence of abelian categories below:
\[\xymatrix@C=85pt{
\Modhol^Z(\D_X)\ar@<0.7ex>@{->}[r]^-{\SH^0\bfD i_{Z}^{-1}(\cdot)[d_Z-d_X]}\ar@{}[r]|-{\sim}
&
\Modhol(\D_Z)
\ar@<0.7ex>@{->}[l]^-{\bfD i_{Z\ast}},
}\]
where $\Modhol^Z(\D_X)$ is the full subcategory of $\Modhol(\D_X)$ consisting of $\D_X$-modules
whose support is contained in $Z$.
See e.g., \cite[Theorem 1.6.1]{HTT08} for general cases of this equivalence of categories.
\end{remark}

\subsection{Minimal Extensions}\label{sec-minimal}
In this subsection we shall consider simple enhanced perverse ind-sheaves on a smooth algebraic variety,
and a counter part of minimal extensions of algebraic holonomic $\D$-modules.

First let us recall the definition and properties of minimal extensions of algebraic holonomic $\D$-modules,
see e.g., \cite[\S 3.4]{HTT08} for the details.
Let $X$ be a smooth algebraic variety over $\CC$ and $Z$ a locally closed smooth subvariety of $X$.
We assume that the natural embedding $i_Z\colon Z\hookrightarrow X$ is affine.
Then we have exact functors
\[\bfD i_{Z\ast}, \bfD i_{Z!} \colon \Modhol(\D_Z)\to\Modhol(\D_X).\]
Note that there exists a canonical morphism of functors $\Modhol(\D_Z)\to\Modhol(\D_X)$
\[\bfD i_{Z!}\to \bfD i_{Z\ast}.\]
See e.g., \cite[Theorem 3.2.16]{HTT08} for the details.
In particular we have a canonical morphism of holonomic $\D_X$-modules
\[\bfD i_{Z!}\M\to \bfD i_{Z\ast}\M\]
for any $\M\in\Modhol(\D_Z)$.

\begin{definition}
In the situation as above,
for any $\M\in\Modhol(\D_Z)$,
we call the image of the canonical morphism $\bfD i_{Z!}\M\to \bfD i_{Z\ast}\M$
the minimal extension of $\M$ along $Z$,
and denote it by $\SL(Z; \M)$.
\end{definition}

Note that the minimal extensions are holonomic.
Moreover we have a functor
\[\SL(Z; \cdot)\colon \Modhol(\D_Z)\to\Modhol(\D_X).\]
Note also that
the minimal extension functor $\SL(Z; \cdot)$ commutes with the duality functor of $\D$-modules.
Namely for any holonomic $\D_Z$-module $\M$,
there exists an isomorphism of holonomic $\D_X$-modules
\[\DD_X(\SL(Z; \M))\simeq\SL(Z; \DD_Z\M).\]

Recall that a non-zero holonomic $\D_X$-module is called simple
if it contains no holonomic $\D_X$-submodules other than $\M$ or $0$,
and $\D_X$-module $\SL$ is called integrable connection if it is locally free of finite rank over $\SO_X$.
Note that integrable connections on $X$ are holonomic $\D_X$-modules.
\begin{theorem}[{\cite[Theorem 3.4.2]{HTT08}}]\label{thm-simple-D}
\begin{itemize}
\item[\rm(1)]
Let $Z$ be a locally closed smooth connected subvariety of smooth algebraic variety $X$
and $\M$ a simple holonomic $\D_Z$-module.
We assume that the natural embedding $i_Z\colon Z\hookrightarrow X$ is affine.

Then the minimal extension $\SL(Z; \M)$ of $\M$ along $Z$ is also simple,
and it is characterized as the unique simple submodule $($resp.\ unique simple quotient module$)$
of $\bfD i_{Z\ast}\M$ $($resp.\ $\bfD i_{Z!}\M)$.

\item[\rm(2)]
For any simple $\D_X$-module $\M$, there exist a locally closed smooth connected subvariety $Z$ of $X$
and a simple integrable connection $\SL$ on $Z$ such that
the natural embedding $i_Z\colon Z\hookrightarrow X$ is affine 
and $\M\simeq \SL(Z; \SL)$.

\item[\rm(3)]
Let $(Z, \SL)$ be as in $(1)$ and $(Z', \SL')$ be another such pair. 
Then we have $\SL(Z; \SL)\simeq \SL(Z'; \SL')$ if and only if
$\var{Z} = \var{Z'}$ and there exists an open dense subset $U$ of $Z\cap Z'$ such that $\SL|_U\simeq\SL'|_U$.
\end{itemize}
\end{theorem}

Thanks to this theorem, 
by Proposition \ref{prop-comm} (3) and Theorem \ref{thm-main-3},
simple objects in the abelian category of algebraic enhanced ind-sheaves can be described as follows.

\begin{definition}\label{def-simple}
Let $X$ be a smooth algebraic variety over $\CC$.
A non-zero algebraic enhanced perverse ind-sheaf $K\in\Perv(\I\CC_{X_\infty})$ is called simple
if it contains no subobjects in $\Perv(\I\CC_{X_\infty})$ other than $K$ or $0$.
\end{definition}

\begin{proposition}\label{prop3.35}
Let $X$ be a smooth algebraic variety over $\CC$.
For any simple algebraic perverse sheaf $\SF\in\Perv(\CC_X)$,
the natural embedding $e_{X^\an_\infty}(\SF)$ of $\SF$ is also simple.
\end{proposition}

\begin{proof}
Let $\SF\in\Perv(\CC_X)$ be a simple algebraic perverse sheaf on $X$
and $K\in\Perv(\I\CC_{X_\infty})$ a subobject of $e_{X^\an_\infty}(\SF)$.
Then there exists $\M\in\Modhol(\D_X)$ such that
$$K\simeq\Sol_{X_\infty}^\rmE(\M)[d_X]$$ by Theorem \ref{thm-main-3}.
Since the functor $\sh_{X^\an_\infty}\colon \Perv(\I\CC_{X_\infty})\to\Perv(\CC_X)$ is t-exact
with respect to the perverse t-structures,
we obtain $\sh_{X^\an_\infty}(K)\subset \sh_{X^\an_\infty}e_{X^\an_\infty}(\SF)$.
By Proposition \ref{prop-sh} and the fact that  
there exists an isomorphism $\SF\simto \sh_{X^\an_\infty}e_{X^\an_\infty}(\SF)$,
we obtain  $$\Sol_X(\M_\reg)[d_X]\simeq\sh_{X^\an_\infty}(K)\subset 
\sh_{X^\an_\infty}e_{X^\an_\infty}(\SF)\simeq\SF.$$
Since $\SF$ is simple, we obtain $\Sol_X(\M_\reg)[d_X]\simeq0$, and hence $\M_\reg\simeq0$.
This implies that $\M\simeq0$, 
thus we have $K\simeq0$.
The proof is completed.
\end{proof}

In this paper, 
we shall say that $K\in\ZEC(\I\CC_{X_\infty^\an})$ is an enhanced local system on $X_\infty$
if for any $x\in X$ there exist an open neighborhood $U\subset X$ of $x$ and a non-negative integer $k$
such that $K|_{U_\infty^\an}\simeq(\CC_{U^\an_\infty}^\rmE)^{\oplus k}.$
Note that for any enhanced local system $K$ on $X_\infty$ there exists an integrable connection $\SL$ on $X$
such that $K\simeq \Sol_{X_\infty}^\rmE(\SL)\in\BEC_{\CC-c}(\I\CC_{X_\infty})$.
In particular $K[d_X]$ is an algebraic enhanced perverse ind-sheaf on $X_\infty$. 

\begin{proposition}\label{prop3.37}
\begin{itemize}
\item[\rm(1)]
Let $Z$ be a locally closed smooth connected subvariety of smooth algebraic variety $X$
and $K$ a simple algebraic enhanced perverse ind-sheaf on $X_\infty$.
We assume that the natural embedding $i_Z\colon Z\hookrightarrow X$ is affine.

Then the image of the canonical morphism $\bfE i_{Z^\an_\infty!!}K\to\bfE i_{Z^\an_\infty\ast}K$ is also simple,
and it is characterized as the unique simple submodule $($resp.\ unique simple quotient module$)$
of $\bfE i_{Z^\an_\infty\ast}K$ $($resp.\ $\bfE i_{Z^\an_\infty!!}K)$.

\item[\rm(2)]
For any simple algebraic enhanced perverse ind-sheaf $K$ on $X_\infty$,
there exist a locally closed smooth connected subvariety $Z$ of $X$
and a simple enhanced local system $L$ on $Z_\infty$ such that
the natural embedding $i_Z\colon Z\hookrightarrow X$ is affine 
and $$K\simeq\Image\big(\bfE i_{Z^\an_\infty!!}L[d_Z]\to\bfE i_{Z^\an_\infty\ast}L[d_Z]\big).$$

\item[\rm(3)]
Let $(Z, L)$ be as in $(1)$ and $(Z', L')$ be another such pair. 
Then we have
$$\Image\big(\bfE i_{Z^\an_\infty!!}L[d_Z]\to\bfE i_{Z^\an_\infty\ast}L[d_Z]\big)\simeq
\Image\big(\bfE i_{Z^\an_\infty!!}L'[d_{Z'}]\to\bfE i_{Z^\an_\infty\ast}L'[d_{Z'}]\big)$$
if and only if
$\var{Z} = \var{Z'}$ and there exists an open dense subset $U$ of $Z\cap Z'$ such that
$L|_{U^\an_\infty}\simeq L'|_{U^\an_\infty}$.
\end{itemize}
\end{proposition}

\begin{proof}
For any $\M\in\Modhol(\D_Z)$, there exist isomorphisms in $\Perv(\I\CC_{X_\infty})$
\begin{align*}
\Sol_{X_\infty}^\rmE(\SL(Z; \M))[d_X]
&\simeq
\Sol_{X_\infty}^\rmE\big(\Image\big(\bfD i_{Z!}\M\to \bfD i_{Z\ast}\M\big)\big)[d_X]\\
&\simeq
\Image\big(\Sol_{X_\infty}^\rmE(\bfD i_{Z\ast}\M)[d_X]\to \Sol_{X_\infty}^\rmE(\bfD i_{Z!}\M)[d_X]\big)\\
&\simeq
\Image\big(\bfE i_{Z^\an_\infty!!}\Sol_{Z_\infty}^\rmE(\M)[d_Z]\to
\bfE i_{Z^\an_\infty\ast}\Sol_{Z_\infty}^\rmE(\M)[d_Z]\big),
\end{align*}
where in the second (resp.\ third) isomorphism
we used Theorems \ref{thm-main-3} (resp.\ Proposition \ref{prop-comm} (3)).
Therefore this theorem follows from Theorems \ref{thm-main-3} and \ref{thm-simple-D}.
\end{proof}

From now on, we shall consider the image of a canonical morphism
\[{}^p\bfE i_{Z^\an_\infty!!}K\to{}^p\bfE i_{Z^\an_\infty\ast}K\]
for a locally closed smooth subvariety $Z$ of $X$
(not necessarily the natural embedding $i_Z\colon Z\hookrightarrow X$ is affine)
and $K\in\Perv(\I\CC_{Z_\infty})$.
Remark that the canonical morphism $${}^p\bfE i_{Z^\an_\infty!!}K\to{}^p\bfE i_{Z^\an_\infty\ast}K$$
is induced by a canonical morphism of functors
$\Perv(\I\CC_{Z_\infty})\to\Perv(\I\CC_{X_\infty})$
\[{}^p\bfE i_{Z^\an_\infty!!}\to{}^p\bfE i_{Z^\an_\infty\ast}\]
and it is an isomorphism if $Z$ is closed. 

In this paper,
we shall define minimal extensions of algebraic enhanced perverse ind-sheaves as follows.

\begin{definition}\label{def-minimal}
Let $X$ be a smooth algebraic variety over $\CC$ and $Z$ a locally closed smooth subvariety of $X$
(not necessarily the natural embedding $i_Z\colon Z\hookrightarrow X$ is affine).

For any $K\in\Perv(\I\CC_{Z_\infty})$,
we call the image of the canonical morphism $${}^p\bfE i_{Z^\an_\infty!!}K\to{}^p\bfE i_{Z^\an_\infty\ast}K$$
the minimal extension of $K$ along $Z$,
and denote it by ${}^p\bfE i_{Z^\an_\infty!!\ast}K$.
\end{definition}

\begin{remark}\label{rem3.37}
\begin{itemize}
\item[(1)]
If $Z$ is open then we have $\big({}^p\bfE i_{Z^\an_\infty!!\ast}K\big)|_{Z^\an_\infty}\simeq K$
by the definition of minimal extensions along $Z$.
\item[(2)]
If $Z$ is closed then we have ${}^p\bfE i_{Z^\an_\infty!!\ast}K
\simeq\bfE i_{Z^\an_\infty\ast}K\simeq\bfE i_{Z^\an_\infty!!}K\in\Perv_Z(\I\CC_{X_\infty})$.
Namely, in the case when $Z$ is closed,
minimal extensions along $Z$ can be characterized by Proposition \ref{prop3.29}.

\item[(3)]
If the natural embedding $i_Z\colon Z\hookrightarrow X$ is affine,
then we have $${}^p\bfE i_{Z^\an_\infty!!\ast}(\cdot) \simeq
 \Image\big(\bfE i_{Z^\an_\infty!!}(\cdot)\to\bfE i_{Z^\an_\infty\ast}(\cdot)\big)$$
by the fact that the functors $\bfE i_{Z^\an_\infty\ast}$ and $\bfE i_{Z^\an_\infty!!}$ are t-exact
with respect to the perverse t-structures, see Remark \ref{rem3.28} for the details.
Moreover, for any $\M\in\Modhol(\D_Z)$ there exists an isomorphism in $\Perv(\I\CC_{X_\infty})$:
\[{}^p\bfE i_{Z^\an_\infty!!\ast}\Sol_{Z_\infty}^\rmE(\M)[d_Z]
\simeq
\Sol_{X_\infty}^\rmE(\SL(Z; \M))[d_X],
\]
see the proof of Proposition \ref{prop3.37} for the details.
This implies that the minimal extension functor commutes the enhanced solution functor.
\end{itemize}
\end{remark}

Since the category $\Perv(\I\CC_{X_\infty})$ is abelian,
the minimal extension ${}^p\bfE i_{Z^\an_\infty!!\ast}K$ of $K\in\Perv(\I\CC_{Z_\infty})$ along $Z$ is also
an algebraic enhanced perverse ind-sheaf on $X_\infty$.
Moreover we have a functor
\[{}^p\bfE i_{Z^\an_\infty!!\ast}\colon\Perv(\I\CC_{Z_\infty})\to\Perv(\I\CC_{X_\infty}).\]
The following proposition means that 
the minimal extension functor ${}^p\bfE i_{Z^\an_\infty!!\ast}$ commutes
with the duality functor of algebraic enhanced perverse ind-sheaves.

\begin{proposition}
In the situation as above,
there exists a commutative diagram$:$
\[\xymatrix@C=60pt@M=5pt{
\Perv(\I\CC_{Z_\infty})^\op\ar@{->}[r]^-{{}^p\bfE i_{Z^\an_\infty!!\ast}}
\ar@{->}[d]_-{\rmD^\rmE_{Z^\an_\infty}}^-\wr\ar@{}[rd]|{\rotatebox[origin=c]{180}{$\circlearrowright$}}
 & \Perv(\I\CC_{X_\infty})^\op\ar@{->}[d]^-{\rmD^\rmE_{Z^\an_\infty}}_-\wr\\
\Perv(\I\CC_{Z_\infty})\ar@{->}[r]_-{{}^p\bfE i_{Z^\an_\infty!!\ast}}
&\Perv(\I\CC_{X_\infty}).
}\]

Namely, for any $K\in\Perv(\I\CC_{Z_\infty})$,
there exists an isomorphism of algebraic enhanced perverse ind-sheaves$:$
\[\rmD_{X^\an_\infty}^\rmE({}^p\bfE i_{Z^\an_\infty!!\ast}K)
\simeq
{}^p\bfE i_{Z^\an_\infty!!\ast}(\rmD_{Z^\an_\infty}^\rmE K).\]
\end{proposition}

\begin{proof}
Recall that the duality functor $\rmD^\rmE$ of enhanced ind-sheaves is t-exact
with respect to the perverse t-structures.
Hence we have isomorphisms in $\Perv(\I\CC_{X_\infty})$
\begin{align*}
\rmD_{X^\an_\infty}^\rmE({}^p\bfE i_{Z^\an_\infty!!}K)
&\simeq
\rmD_{X^\an_\infty}^\rmE({}^p\SH^0(\bfE i_{Z^\an_\infty!!}K))\\
&\simeq
{}^p\SH^0(\rmD_{X^\an_\infty}^\rmE(\bfE i_{Z^\an_\infty!!}K))\\
&\simeq
{}^p\SH^0(\bfE i_{Z^\an_\infty\ast}\rmD_{Z^\an_\infty}^\rmE K)\\
&\simeq
{}^p\bfE i_{Z^\an_\infty\ast}\rmD_{Z^\an_\infty}^\rmE K,
\end{align*}
where in the third isomorphism we used Proposition \ref{prop2.8} (3).
In the similar way,
we have an isomorphism
$\rmD_{X^\an_\infty}^\rmE({}^p\bfE i_{Z^\an_\infty\ast}K)
\simeq
{}^p\bfE i_{Z^\an_\infty!!}\rmD_{Z^\an_\infty}^\rmE K$
in $\Perv(\I\CC_{X_\infty})$.

Therefore we obtain isomorphisms in $\Perv(\I\CC_{X_\infty})$:
\begin{align*}
\rmD_{X^\an_\infty}^\rmE({}^p\bfE i_{Z^\an_\infty!!\ast}K)
&\simeq
\rmD_{X^\an_\infty}^\rmE\big(\Image({}^p\bfE i_{Z^\an_\infty!!}K\to{}^p\bfE i_{Z^\an_\infty\ast}K)\big)\\
&\simeq
\Image\big(\rmD_{X^\an_\infty}^\rmE({}^p\bfE i_{Z^\an_\infty\ast}K)\to
\rmD_{X^\an_\infty}^\rmE({}^p\bfE i_{Z^\an_\infty!!}K)\big)\\
&\simeq
\Image\big({}^p\bfE i_{Z^\an_\infty!!}\rmD_{Z^\an_\infty}^\rmE K\to
{}^p\bfE i_{Z^\an_\infty\ast}\rmD_{Z^\an_\infty}^\rmE K\big)\\
&\simeq
{}^p\bfE i_{Z^\an_\infty!!\ast}\rmD_{Z^\an_\infty}^\rmE K.
\end{align*}
\end{proof}

Let us recall that the minimal extension ${}^p\bfR i_{Z!\ast}\F$ of a perverse sheaf $F$ along $Z$ is defined by
the image of a canonical morphism \[{}^p\bfR i_{Z^\an!}\F\to{}^p\bfR i_{Z^\an\ast}\F,\]
where ${}^p\bfR i_{Z^\an!} := {}^p\SH^0\circ\bfR i_{Z^\an!}$,
${}^p\bfR i_{Z^\an\ast} := {}^p\SH^0\circ\bfR i_{Z^\an\ast}$
and ${}^p\SH^0$ is the $0$-th cohomology functor with respect to the perverse t-structures.
See \cite{BBD} (also \cite[\S 8.2.2]{HTT08}) for the details.
The following proposition means that the natural embedding functor
commute with the minimal extension functor.

\begin{proposition}
In the situation as above,
there exists a commutative diagram$:$
\[\xymatrix@C=60pt@M=5pt{
\Perv(\I\CC_{Z_\infty})\ar@{->}[r]^-{{}^p\bfE i_{Z^\an_\infty!!\ast}}
\ar@{}[rd]|{\rotatebox[origin=c]{180}{$\circlearrowright$}}
 & \Perv(\I\CC_{X_\infty})\\
\Perv(\CC_{Z})\ar@{^{(}->}[u]^-{e_{Z^\an_\infty}}\ar@{->}[r]_-{{}^p\bfR i_{Z^\an!\ast}}
&\Perv(\CC_{X})\ar@{^{(}->}[u]_-{e_{X^\an_\infty}}.
}\]
Namely, for any $\SF\in\Perv(\CC_{Z})$,
there exists an isomorphism in $\Perv(\I\CC_{X_\infty}):$
\[e_{X^\an_\infty}({}^p\bfR i_{Z^\an!\ast}\SF)
\simeq
{}^p\bfE i_{Z^\an!!\ast}(e_{Z^\an_\infty} \SF).\]
\end{proposition}

\begin{proof}
Let $\SF$ be an object of $\Perv(\CC_Z)$.
Recall that
since the morphism $i_{Z^\an_\infty}\colon Z^\an_\infty\to X^\an_\infty$ of bordered spaces is semi-proper,
there exist isomorphisms in $\BEC_{\CC-c}(\I\CC_{X_\infty})$:
\begin{align*}
\bfE i_{Z^\an_\infty!!}(e_{Z^\an_\infty} \SF)
&\simeq
e_{X^\an_\infty}(\bfR i_{Z^\an!}\SF),\\
\bfE i_{Z^\an_\infty\ast}(e_{Z^\an_\infty} \SF)
&\simeq
e_{X^\an_\infty}(\bfR i_{Z^\an\ast}\SF)
\end{align*}
where the second isomorphism follows from the first isomorphism, Proposition \ref{prop2.8} (3) and
the fact that the duality functor commutes with the natural embedding functor.

Moreover since the natural embedding functor is t-exact
with respect to the perverse t-structures by Proposition \ref{prop3.22},
there exist isomorphisms in $\Perv(\I\CC_{X_\infty})$:
\begin{align*}
e_{X^\an_\infty}({}^p\bfR i_{Z^\an\ast}\SF)
&\simeq
{}^p\bfE i_{Z^\an\ast}(e_{Z^\an_\infty} \SF),\\
e_{X^\an_\infty}({}^p\bfR i_{Z^\an!}\SF)
&\simeq
{}^p\bfE i_{Z^\an!!}(e_{Z^\an_\infty} \SF).
\end{align*}
Hence we obtain isomorphisms in $\Perv(\I\CC_{X_\infty})$:
\begin{align*}
e_{X^\an_\infty}({}^p\bfR i_{Z^\an!\ast}\SF)
&\simeq
e_{X^\an_\infty}\big(\Image\big({}^p\bfR i_{Z^\an!}\SF\to{}^p\bfR i_{Z^\an\ast}\SF\big)\big)\\
&\simeq
\Image\big(e_{X^\an_\infty}\big({}^p\bfR i_{Z^\an!}\SF\big)\to
e_{X^\an_\infty}\big({}^p\bfR i_{Z^\an\ast}\SF\big)\big)\\
&\simeq
\Image\big({}^p\bfE i_{Z^\an_\infty!!}\big(e_{Z^\an_\infty}\SF\big)\to{}^p\bfE i_{Z^\an_\infty\ast}\big(e_{Z^\an_\infty}\SF\big)\big)\\
&\simeq
{}^p\bfE i_{Z^\an_\infty!!\ast}(e_{Z^\an_\infty} \SF).
\end{align*}
The proof is completed.
\end{proof}

The following lemma will be used in the proof of Theorem \ref{thm-main-4}.
\begin{lemma}\label{lem3.41}
Let $X$ be a smooth algebraic variety over $\CC$, $Z$ and $W$ locally closed smooth subvarieties of $X$
$($not necessarily the natural embeddings $Z\hookrightarrow X$ and $W\hookrightarrow X$ is affine$)$.
We assume $W\subset Z$ and we consider a commutative diagram $:$
\[\xymatrix@R=20pt@C=40pt{
Z^\an_\infty\ar@{->}[r]^-{i_{Z^\an_\infty}} & X^\an_\infty\\
W^\an_\infty\ar@{->}[ru]_-{i_{W^\an_\infty}}\ar@{->}[u]^-{k} 
},\]
where $i_{Z^\an_\infty}, i_{W^\an_\infty}$ and $k$ are the morphisms of bordered spaces
induced by the natural embeddings $Z\hookrightarrow X, W\hookrightarrow X$ and $W\hookrightarrow Z$, respectively.

Then for any $K\in\Perv(\I\CC_{W_\infty})$ we have 
\begin{itemize}
\item[\rm(1)]
${}^p\bfE i_{W^\an_\infty!!}K\simeq
{}^p\bfE i_{Z^\an_\infty!!}{}^p\bfE k_{!!}K$ and 
${}^p\bfE i_{W^\an_\infty\ast}K\simeq
{}^p\bfE i_{Z^\an_\infty\ast}{}^p\bfE k_{\ast}K$,
\item[\rm(2)]
${}^p\bfE i_{W^\an_\infty!!\ast}K\simeq
{}^p\bfE i_{Z^\an_\infty!!\ast}{}^p\bfE k_{!!\ast}K.$
\end{itemize}
\end{lemma}

\begin{proof}
Let $K$ be an object of $\Perv(\I\CC_{W_\infty})$.
\item[(1)]
Since the proof of the first assertion of (1) is similar,
we shall only prove the second one.
Recall that the functor $\bfE i_{Z^\an_\infty\ast}$ is left t-exact with respect to the perverse t-structures
by Corollary \ref{cor3.27} (2).
Hence we have an isomorphism in $\Perv(\I\CC_{X\infty})$
\[{}^p\SH^0\big(\bfE i_{Z^\an_\infty\ast}\bfE k_{\ast}K\big)
\simeq
{}^p\SH^0\big(\bfE i_{Z^\an_\infty\ast}{}^p\SH^0(\bfE k_{\ast}K)\big)\]
for any $K\in\Perv(\I\CC_{W_\infty})$,
see e.g., \cite[Proposition 8.1.15 (i)]{HTT08}.
Therefore we obtain isomorphisms in $\Perv(\I\CC_{X_\infty})$:
\begin{align*}
{}^p\bfE i_{W^\an_\infty\ast}K
&\simeq
{}^p\SH^0\bfE i_{W^\an_\infty\ast}K\\
&\simeq
{}^p\SH^0\big(\bfE i_{Z^\an_\infty\ast}\bfE k_\ast K\big)\\
&\simeq
{}^p\SH^0\big(\bfE i_{Z^\an_\infty\ast}{}^p\SH^0(\bfE k_{\ast}K)\big)\\
&\simeq
{}^p\bfE i_{Z^\an_\infty\ast}{}^p\bfE k_{\ast}K.
\end{align*}

\item[(2)]
Recall that there exist canonical morphisms:
\begin{align*}
{}^p\bfE i_{W^\an_\infty!!}K\twoheadrightarrow
&{}^p\bfE i_{W^\an_\infty!!\ast}K\hookrightarrow
{}^p\bfE i_{W^\an_\infty\ast}K,\\
{}^p\bfE k_{!!}K\twoheadrightarrow
&{}^p\bfE k_{!!\ast}K\hookrightarrow
{}^p\bfE k_{\ast}K,\\
{}^p\bfE i_{Z^\an_\infty!!}{}^p\bfE k_{!!\ast}K\twoheadrightarrow
&{}^p\bfE i_{Z^\an_\infty!!\ast}{}^p\bfE k_{!!\ast}K\hookrightarrow
{}^p\bfE i_{z^\an_\infty\ast}{}^p\bfE k_{!!\ast}K,
\end{align*}
where $\twoheadrightarrow$ (resp.\ $\hookrightarrow$) is an epimorphism (resp.\ a monomorphism)
in the abelian category of algebraic enhanced perverse ind-sheaves.
Since the functor $\bfE i_{Z^\an_\infty\ast}$ (resp.\ $\bfE i_{Z^\an_\infty!!}$) is left (resp.\ right) t-exact
with respect to the perverse t-structures,
the canonical morphism 
$${}^p\bfE i_{W^\an_\infty!!}K = {}^p\bfE i_{Z^\an_\infty!!}{}^p\bfE k_{!!}K
\to
{}^p\bfE i_{Z^\an_\infty\ast}{}^p\bfE k_{\ast}K = {}^p\bfE i_{W^\an_\infty\ast}K$$ can be decomposed
as follows:
\begin{align*}
{}^p\bfE i_{Z^\an_\infty!!}{}^p\bfE k_{!!}K
\twoheadrightarrow
{}^p\bfE i_{Z^\an_\infty!!}{}^p\bfE k_{!!\ast}K
\twoheadrightarrow
{}^p\bfE i_{Z^\an_\infty!!\ast}{}^p\bfE k_{!!\ast}K\hookrightarrow
{}^p\bfE i_{Z^\an_\infty\ast}{}^p\bfE k_{!!\ast}K
\hookrightarrow
{}^p\bfE i_{Z^\an_\infty\ast}{}^p\bfE k_{\ast}K.
\end{align*}
This implies that 
$$\Image({}^p\bfE i_{W^\an_\infty!!}K\to{}^p\bfE i_{W^\an_\infty\ast}K)
\simeq {}^p\bfE i_{Z^\an_\infty!!\ast}{}^p\bfE k_{!!\ast}K.$$
Hence the proof is completed.
\end{proof}

Recall that in the case when $Z$ is a closed smooth subvariety of $X$,
minimal extensions along $Z$ can be characterized
by Proposition \ref{prop3.29}, see also Remark \ref{rem3.37} (2). 
On the other hand, in the case when $Z$ is open whose complement is a smooth subvariety,
the minimal extensions along $Z$ can be characterized as follows.
Let $U$ be such an open subset of $X$ and set $W := X\setminus U$.
Namely $U$ is an open subset of $X$ and $W := X\setminus U$ is a closed smooth subvariety of $X$. 

\begin{proposition}\label{prop3.41}
In the situation as above,
the minimal extension ${}^p\bfE i_{UZ^\an_\infty!!\ast}K$ of $K\in\Perv(\I\CC_{U_\infty})$ along $U$
is characterized as the unique algebraic enhanced perverse ind-sheaf $L$ on $X_\infty$ satisfying the conditions
\begin{itemize}
\setlength{\itemsep}{-2pt}
\item[\rm(1)]
$\bfE i_{U^\an_\infty}^{-1}L\simeq K$,
\item[\rm(2)]
$\bfE i_{W^\an_\infty}^{-1}L\in\bfE^{\leq-1}_{\CC-c}(\I\CC_{W_\infty})$,
\item[\rm(3)]
$\bfE i_{W^\an_\infty}^{!}L\in\bfE^{\geq1}_{\CC-c}(\I\CC_{W_\infty})$.
\end{itemize}
\end{proposition}

\begin{proof}
Let $K$ be an object of $\Perv(\I\CC_{U_\infty})$ and 
set $L:={}^p\bfE i_{U^\an_\infty!!\ast}K\in\Perv(\I\CC_{X_\infty})$.
Then we have $\bfE i_{U^\an_\infty}^{-1}L\simeq K$ by the definition of $L$.
Since the functor $\bfE i_{W^\an_\infty}^{-1}$ (resp.\ $\bfE i_{W^\an_\infty}^{!}$)
is right (resp.\ left) t-exact with respect to the perverse t-structures by Corollary \ref{cor3.27} (3) (resp.\ (4)),
we obtain $\bfE i_{W^\an_\infty}^{-1}L\in\bfE^{\leq0}_{\CC-c}(\I\CC_{W_\infty})$
(resp.\ $\bfE i_{W^\an_\infty}^{!}L\in\bfE^{\geq0}_{\CC-c}(\I\CC_{W_\infty})$).
Recall that there exist canonical distinguished triangles in $\BEC(\I\CC_{X_\infty})$:
\begin{align*}
\bfE i_{U^\an_\infty!!}\bfE i_{U^\an_\infty}^{-1}L\to &L\to
\bfE i_{W^\an_\infty!!}\bfE i_{W^\an_\infty}^{-1}L\xrightarrow{+1},\\
\bfE i_{W^\an_\infty\ast}\bfE i_{W^\an_\infty}^{!}L\to &L\to
\bfE i_{U^\an_\infty\ast}\bfE i_{U^\an_\infty}^{!}L\xrightarrow{+1}
\end{align*}
by \cite[Lemma 2.7.7]{DK16-2}.
Then we obtain distinguished triangles in $\BEC(\I\CC_{X_\infty})$
\begin{align*}
\bfE i_{U^\an_\infty!!}K\to L\to
\bfE i_{W^\an_\infty!!}\bfE i_{W^\an_\infty}^{-1}L\xrightarrow{+1},\\
\bfE i_{W^\an_\infty\ast}\bfE i_{W^\an_\infty}^{!}L\to L\to
\bfE i_{U^\an_\infty\ast}K\xrightarrow{+1}.
\end{align*}
By taking the $0$-th cohomology, we obtain exact sequences in $\Perv(\I\CC_{X_\infty})$:
\begin{align*}
{}^p\bfE i_{U^\an_\infty!!}K\to L\to
&{}^p\SH^0\big(\bfE i_{W^\an_\infty!!}\bfE i_{W^\an_\infty}^{-1}L\big)
\to{}^p\SH^1\bfE i_{U^\an_\infty!!}K,\\
{}^p\SH^{-1}\bfE i_{U^\an_\infty\ast}K\to
&{}^p\SH^0\big(\bfE i_{W^\an_\infty\ast}{}^p\bfE i_{W^\an_\infty}^{!}L\big)\to L\to
{}^p\bfE i_{U^\an_\infty\ast}K.
\end{align*}
Moreover since the functor $\bfE i_{U^\an_\infty!!}$ (resp.\ $\bfE i_{U^\an_\infty\ast}$)
is right (resp.\ left) t-exact with respect to the perverse t-structures by Corollary \ref{cor3.27} (1) (resp.\ (2)),
we obtain ${}^p\SH^1\bfE i_{U^\an_\infty!!}K\simeq0$ and ${}^p\SH^{-1}\bfE i_{U^\an_\infty\ast}K\simeq0$.
Thus there exist exact sequences in $\Perv(\I\CC_{X_\infty})$:
\begin{align*}
{}^p\bfE i_{U^\an_\infty!!}K\to L\to
{}^p\SH^0\big(\bfE i_{W^\an_\infty!!}\bfE i_{W^\an_\infty}^{-1}L\big)\to0,\\
0\to
{}^p\SH^0\big(\bfE i_{W^\an_\infty\ast}{}^p\bfE i_{W^\an_\infty}^{!}L\big)\to L\to
{}^p\bfE i_{U^\an_\infty\ast}K.
\end{align*}
Since the morphism ${}^p\bfE i_{U^\an_\infty!!}K\to L$
(resp.\ $L\to{}^p\bfE i_{U^\an_\infty\ast}K$)
is an epimorphism (resp.\ a monomorphism),
we obtain ${}^p\SH^0\big(\bfE i_{W^\an_\infty!!}\bfE i_{W^\an_\infty}^{-1}L\big)\simeq0$
\big(resp.\ ${}^p\SH^0\big(\bfE i_{W^\an_\infty\ast}\bfE i_{W^\an_\infty}^{!}L\big)\simeq0$\big).
Therefore we have ${}^p\SH^0\big(\bfE i_{W^\an_\infty}^{-1}L\big)\simeq0$ and 
${}^p\SH^0\big(\bfE i_{W^\an_\infty}^{!}L\big)\simeq0$, and hence 
$$\bfE i_{W^\an_\infty}^{-1}L\in\bfE^{\leq-1}_{\CC-c}(\I\CC_{W_\infty}) \mbox{\ and\ } 
\bfE i_{W^\an_\infty}^{!}L\in\bfE^{\geq1}_{\CC-c}(\I\CC_{W_\infty}).$$

On the other hand, 
let $L$ be an object of $\Perv(\I\CC_{X_\infty})$ which satisfies the conditions (1), (2) and (3) as above.
Recall that there exist canonical morphisms
\[\bfE i_{U^\an_\infty!!}\bfE i_{U^\an_\infty}^{-1}L\to L\to\bfE i_{U^\an_\infty\ast}\bfE i_{U^\an_\infty}^{-1}L\]
in $\BEC_{\CC-c}(\I\CC_{X_\infty})$.
Since $L$ satisfies the condition (1) as above,
we have canonical morphisms $\bfE i_{U^\an_\infty!!}K\xrightarrow{} L\xrightarrow{} \bfE i_{U^\an_\infty\ast}K$
in $\BEC_{\CC-c}(\I\CC_{X_\infty})$, and hence we obtain morphisms
$${}^p\bfE i_{U^\an_\infty!!}K\xrightarrow{\ \alpha\ } L\xrightarrow{\ \beta\ } {}^p\bfE i_{U^\an_\infty\ast}K$$
in $\Perv(\I\CC_{X_\infty})$.
It is enough to show that 
the morphism $\alpha$ (resp.\ $\beta$) is an epimorphism (resp.\ a monomorphism).
Let us consider exact sequences
\begin{align*}
{}^p\bfE i_{U^\an_\infty!!}K\to L\to\Coker\alpha\to0,\\
0\to\Ker\beta\to L\to{}^p\bfE i_{U^\an_\infty\ast}K
\end{align*}
in $\Perv(\I\CC_{X_\infty})$.
Since $\bfE i_{U^\an_\infty}^{-1}(\Coker\alpha)\simeq0$ and $\bfE i_{U^\an_\infty}^{-1}(\Ker\beta)\simeq0$,
we have $\Coker\alpha, \Ker\beta\in\Perv_W(\I\CC_{X_\infty})$,
and hence there exist $L_{\alpha},L_\beta\in\Perv(\I\CC_{W_\infty})$ such that 
$\bfE i_{W^\an_\infty!!}(L_\alpha)\simeq\Coker\alpha$ and $\bfE i_{W^\an_\infty!!}(L_\beta)\simeq\Ker\beta$
by Proposition \ref{prop3.29}.
Moreover we have exact sequences in $\Perv(\I\CC_{W_\infty})$:
\begin{align*}
{}^p\bfE i_{W^\an_\infty}^{-1}{}^p\bfE i_{U^\an_\infty!!}K\to{}^p\bfE i_{W^\an_\infty}^{-1}L
\to{}^p\bfE i_{W^\an_\infty}^{-1}\bfE i_{W^\an_\infty!!}L_\alpha(\ \simeq L_\alpha)\to0,\\
0\to(L_\beta\simeq\ ){\ }^p\bfE i_{W^\an_\infty}^{!}\bfE i_{W^\an_\infty!!}L_\beta\to{}^p\bfE i_{W^\an_\infty}^{!}L
\to{}^p\bfE i_{W^\an_\infty}^{!}{}^p\bfE i_{U^\an_\infty\ast}K,
\end{align*}
where we used the fact that
the functor ${}^p\bfE i_{W^\an_\infty}^{-1}\colon \Perv(\I\CC_{X_\infty})\to\Perv(\I\CC_{W_\infty})$
(resp.\ ${}^p\bfE i_{W^\an_\infty}^{!}\colon \Perv(\I\CC_{X_\infty})\to\Perv(\I\CC_{W_\infty}$)
is right (resp.\ left) exact. 
Since $L$ satisfies the condition (2) (resp.\ (3)) as above,
we have ${}^p\bfE i_{W^\an_\infty}^{-1}L\simeq0$ (resp.\ ${}^p\bfE i_{W^\an_\infty}^{!}L\simeq0$).
Thus we obtain $L_\alpha\simeq0$ (resp.\ $L_\beta\simeq0$).
This implies that the morphism $\alpha$ (resp.\ $\beta$) is an epimorphism (resp.\ a monomorphism).
\end{proof}

Furthermore the minimal extensions along $U$ have following properties.
\begin{proposition}\label{prop3.36}
In the situation as above, for any $K\in \Perv(\I\CC_{U_\infty})$
\begin{itemize}
\item[\rm(1)]
${}^p\bfE i_{U^\an_\infty\ast}K\in\Perv(\I\CC_{X_\infty})$
has no non-trivial subobject in $\Perv(\I\CC_{X_\infty})$ whose support is contained in $W^\an$.
\item[\rm(2)]
${}^p\bfE i_{U^\an_\infty!!}K\in\Perv(\I\CC_{X_\infty})$
has no non-trivial quotient object in $\Perv(\I\CC_{X_\infty})$ whose support is contained in $W^\an$.
\end{itemize}
\end{proposition}

\begin{proof}
Since the proof of (2) is similar, we only prove (1).

Let $L\in\Perv_W(\I\CC_{X_\infty})$ be a subobject of ${}^p\bfE i_{U^\an_\infty\ast}K\in\Perv(\I\CC_{X_\infty})$.
By Proposition \ref{prop3.29},
there exists an isomorphism
$\bfE i_{W^\an_\infty!!}{}^p\bfE i_{W^\an_\infty}^{!}L\simto L$
in $\Perv_W(\I\CC_{X_\infty})$.
Hence it is enough to prove ${}^p\bfE i_{W^\an_\infty}^{!}L\simeq0$.
Since $L$ is subobject of ${}^p\bfE i_{U^\an_\infty\ast}K$, there exists a monomorphism
$L\hookrightarrow {}^p\bfE i_{U^\an_\infty\ast}K$ in $\Perv(\I\CC_{X_\infty})$.
Moreover we have a monomorphism 
${}^p\bfE i_{W^\an_\infty}^{!}L\hookrightarrow {}^p\bfE i_{W^\an_\infty}^{!}{}^p\bfE i_{U^\an_\infty\ast}K$
in $\Perv(\I\CC_{X_\infty})$,
because the functor ${}^p\bfE i_{W^\an_\infty}^{!}$ is left t-exact with respect to the perverse t-structures
by Corollary \ref{cor3.27} (4).
Since ${}^p\bfE i_{W^\an_\infty}^{!}{}^p\bfE i_{U^\an_\infty\ast}K
\simeq {}^p\SH^0\big(\bfE i_{W^\an_\infty}^{!}\bfE i_{U^\an_\infty\ast}K\big)\simeq0$,
we obtain ${}^p\bfE i_{W^\an_\infty}^{!}L\simeq0$.
The proof is completed.
\end{proof}

\begin{corollary}\label{cor3.38}
In the situation as above,
the minimal extension ${}^p\bfE i_{U^\an_\infty!!\ast}K$ of $K\in\Perv(\I\CC_{U_\infty})$ along $U$
has neither non-trivial subobject nor non-trivial quotient object in $\Perv(\I\CC_{X_\infty})$
whose support is contained in $W^\an$.
\end{corollary}

\begin{proof}
Recall that there exist the canonical morphisms in $\Perv(\I\CC_{X_\infty})$
\[{}^p\bfE i_{U^\an_\infty!!}K\twoheadrightarrow
{}^p\bfE i_{U^\an_\infty!!\ast}K\hookrightarrow
{}^p\bfE i_{U^\an_\infty\ast}K.\]
Since any non-trivial subobject (resp.\ quotient object) of ${}^p\bfE i_{U^\an_\infty!!\ast}K$ is also
non-trivial subobject (resp.\ quotient object) of ${}^p\bfE i_{U^\an_\infty\ast}K$
(resp.\ ${}^p\bfE i_{U^\an_\infty!!}K$),
this follows from Proposition \ref{prop3.36}.
\end{proof}

\begin{corollary}\label{cor3.44}
In the situation as above.
\begin{itemize}
\item[\rm(1)]
For an exact sequence $0\to K\to L$ in $\Perv(\I\CC_{U_\infty})$,
the associated sequence $0\to {}^p\bfE i_{U^\an_\infty!!\ast}K\to {}^p\bfE i_{U^\an_\infty!!\ast}L$ 
in $\Perv(\I\CC_{X_\infty})$
is also exact.
\item[\rm(2)]
For an exact sequence $K\to L\to 0$ in $\Perv(\I\CC_{U_\infty})$,
the associated sequence ${}^p\bfE i_{U^\an_\infty!!\ast}K\to {}^p\bfE i_{U^\an_\infty!!\ast}L\to0$
in $\Perv(\I\CC_{X_\infty})$
is also exact.
\end{itemize}
\end{corollary}

\begin{proof}
Since the proof of (2) is similar, we only prove (1).

Let $\alpha \colon K\hookrightarrow L$ be a monomorphism in $\Perv(\I\CC_{U_\infty})$.
It is enough to show $\Ker({}^p\bfE i_{U^\an_\infty!!\ast}\alpha)\simeq0$.
Recall that there exist isomorphisms in $\Perv(\I\CC_{U_\infty})$
$$\big({}^p\bfE i_{U^\an_\infty!!\ast}K\big)|_{U^\an_\infty}\simeq K,\hspace{3pt}
\big({}^p\bfE i_{U^\an_\infty!!\ast}L\big)|_{U^\an_\infty}\simeq L.$$
Hence we obtain 
$$\big(\Ker({}^p\bfE i_{U^\an_\infty!!\ast}\alpha)\big)|_{U^\an_\infty}
\simeq
\Ker\big(({}^p\bfE i_{U^\an_\infty!!\ast}\alpha)|_{U^\an_\infty}\big)
\simeq \Ker\alpha
\simeq 0,$$
where we used the fact that the functor $(\cdot)|_{U^\an_\infty} := \bfE i_{U^\an_\infty}^{-1}$ is t-exact
with respect to the perverse t-structures (see Corollary \ref{cor3.27} for the details).
This implies that the support of $\Ker({}^p\bfE i_{U^\an_\infty!!\ast}\alpha)$ is contained
in $W^\an = X^\an\setminus U^\an$.
Namely the kernel $\Ker({}^p\bfE i_{U^\an_\infty!!\ast}\alpha)$
of ${}^p\bfE i_{U^\an_\infty!!\ast}\alpha \colon {}^p\bfE i_{U^\an_\infty!!\ast}K\to {}^p\bfE i_{U^\an_\infty!!\ast}L$
is a subobject of ${}^p\bfE i_{U^\an_\infty!!\ast}K$ whose support is contained in $W^\an$.
Therefore $\Ker({}^p\bfE i_{U^\an_\infty!!\ast}\alpha)\simeq0$ by Proposition \ref{prop3.36} (1).
\end{proof}

\begin{corollary}\label{cor3.45}
In the situation as above,
for any simple object in $\Perv(\I\CC_{U_\infty})$,
its minimal extension along $U$
is also a simple object in $\Perv(\I\CC_{X_\infty})$.
\end{corollary}

\begin{proof}
Let $K$ be a simple object in $\Perv(\I\CC_{U_\infty})$
and $L$ be a subobject of ${}^p\bfE i_{U^\an_\infty!!\ast}K$.
Then we have an exact sequence in $\Perv(\I\CC_{X_\infty})$:
\[0\to L\to{}^p\bfE i_{U^\an_\infty!!\ast}K\to L'\to0.\]
Since $\big({}^p\bfE i_{U^\an_\infty!!\ast}K\big)|_{U^\an_\infty}\simeq K$ (see Remark \ref{rem3.37}) and
the functor $\bfE i_{U^\an_\infty}^{-1}$ is t-exact
with respect to the perverse t-structures (see Corollary \ref{cor3.27} for the details),
there exists an exact sequence in $\Perv(\I\CC_{U_\infty})$:
\[0\to \bfE i_{U^\an_\infty}^{-1} L\to K\to \bfE i_{U^\an_\infty}^{-1}L'\to0.\]
Hence we have $\bfE i_{U^\an_\infty}^{-1}L \simeq0$ or $\bfE i_{U^\an_\infty}^{-1}L' \simeq0$
because $K$ is simple.
This implies that $\supp(L)\subset W^\an$ or $\supp(L')\subset W^\an$.
Namely $L\simeq0$ or $L'\simeq0$ by Corollary \ref{cor3.38}.
Hence ${}^p\bfE i_{U^\an_\infty!!\ast}K$ is a simple object in $\Perv(\I\CC_{X_\infty})$.
\end{proof}

Therefore by Proposition \ref{prop3.29} and Lemma \ref{lem3.41},
we obtain the following results.
\begin{theorem}\label{thm-main-4}
Let $X$ be a smooth algebraic variety over $\CC$ and $Z$ a locally closed smooth subvariety of $X$
$($not necessarily the natural embedding $i_Z\colon Z\hookrightarrow X$ is affine$)$.
We assume that $Z = U\cap W$
where $U\subset X$ is an open subset whose complement  $X\setminus U$ is smooth
and $W\subset X$ is a closed subvariety.

\begin{itemize}
\item[\rm(1)]
%
\begin{itemize}
\item[\rm(i)]
For an exact sequence $0\to K\to L$ in $\Perv(\I\CC_{Z_\infty})$,
the associated sequence $0\to {}^p\bfE i_{Z^\an_\infty!!\ast}K\to {}^p\bfE i_{Z^\an_\infty!!\ast}L$ 
in $\Perv(\I\CC_{X_\infty})$
is also exact.
\item[\rm(ii)]
For an exact sequence $K\to L\to 0$ in $\Perv(\I\CC_{Z_\infty})$,
the associated sequence ${}^p\bfE i_{Z^\an_\infty!!\ast}K\to {}^p\bfE i_{Z^\an_\infty!!\ast}L\to0$
in $\Perv(\I\CC_{X_\infty})$
is also exact.
\end{itemize}

\item[\rm(2)]
For any simple object in $\Perv(\I\CC_{Z_\infty})$,
its minimal extension along $Z$ is also a simple object in $\Perv(\I\CC_{X_\infty})$.
\end{itemize}
\end{theorem}

\begin{proof}
Let us consider a commutative diagram:
\[\xymatrix@R=20pt@C=40pt{
U^\an_\infty\ar@{->}[r]^-{i_{U^\an_\infty}} & X^\an_\infty\\
Z^\an_\infty
\ar@{->}[ru]_-{i_{Z^\an_\infty}}\ar@{->}[u]^-{k} 
}.\]
Then we have ${}^p\bfE i_{Z^\an_\infty!!\ast} = {}^p\bfE i_{U^\an_\infty!!\ast} \circ {}^p\bfE k_{!!\ast}$
by Lemma \ref{lem3.41}.
Furthermore there exists an equivalence of categories
$${}^p\bfE k_{!!\ast} = \bfE k_{!!}\colon\Perv(\I\CC_{Z_\infty})\simto\Perv_Z(\I\CC_{U_\infty})$$
by Proposition \ref{prop3.29}.
Therefore the assertion (1) (resp.\ (2)) follows from Corollary \ref{cor3.44} (resp.\ Corollary \ref{cor3.45}).
\end{proof}

\end{document}